\documentclass[12pt]{amsart}
\usepackage{amsmath,amsthm,amscd,amssymb}
\usepackage{geometry}
\usepackage{subcaption}
\usepackage{pb-diagram}
\usepackage{graphicx}
\usepackage{setspace}

\usepackage{amssymb}

\usepackage{pstricks}
\usepackage{pst-node}
\usepackage{url}
 \usepackage{hyperref}

\usepackage{mathrsfs}
\usepackage{MnSymbol}
\usepackage{gensymb} 
\usepackage{enumerate}
\usepackage{mathtools} 

\newtheorem{thm}{Theorem}[section]

\newtheorem{lem}[thm]{Lemma}
\newtheorem{prop}[thm]{Proposition}
\newtheorem{cor}[thm]{Corollary}

\theoremstyle{definition}

\newtheorem{rem}[thm]{Remark}
\newtheorem{ques}[thm]{Question}
\newtheorem{conj}[thm]{Conjecture}
\numberwithin{equation}{section}
\numberwithin{figure}{section}

\newcommand{\Z}{\mathbb{Z}}
\newcommand{\Q}{\mathbb{Q}}
\newcommand{\DK}{D_{K}(S^{3})}
\newcommand{\DF}{D_{F}(B^{4})}

\usepackage{titlesec}
\titleformat{\section} {\normalfont\scshape\bfseries\filcenter}{\thesection }{1em}{}
\titleformat{\subsection} {\normalfont\scshape\bfseries\filcenter}{\thesubsection}{1em}{}
\renewcommand \thesubsection{}

\begin{document}

\title[{THE NON-ORIENTABLE 4-GENUS OF 11 CROSSING NON-ALTERNATING KNOTS}]{THE NON-ORIENTABLE 4-GENUS OF 11 CROSSING NON-ALTERNATING KNOTS}

\author[MEGAN FAIRCHILD]{MEGAN FAIRCHILD}
\address{Department of Mathematics \\ Louisiana State University}
\email{mfarr17@lsu.edu}

\maketitle


ABSTRACT. The non-orientable 4-genus of a knot $K$ in $S^{3}$ is defined to be the minimum first Betti number of a non-orientable surface $F$ in $B^{4}$ so that $K$ bounds $F$. We will survey the tools used to compute the non-orientable 4-genus, and use various techniques to calculate this invariant for non-alternating 11 crossing knots. We also will view obstructions to a knot bounding a M\"{o}bius band given by the double branched cover of $S^{3}$ branched over $K$. 

\section{INTRODUCTION}

Knots bounding orientable surfaces, both in $S^{3}$ and $B^{4}$, has been extensively studied, however much is still to be learned about the non-orientable surfaces in $B^{4}$ bounded by knots. Recently, the non-orientable 4-genus of torus knots has been computed for all knots $T(2, q)$ and $T(3, q)$ by Allen \cite{Allen}, and most knots $T(4, q)$ by Binns, Kang, Simone, Tru\"ol, and Sabloff \cite{torus1, torus2}. The non-orientable 4-genus of knots with 10 or fewer crossings has also been computed in detail by Ghanbarian, Jabuka, and Kelly \cite{N10, JK}, with much focus on alternating knots. This paper aims to shed light on the non-alternating case and strategies to calculate the non-orientable 4-genus. We will explore various techniques in finding this invariant, as well as examining obstructions to knots bounding a M\"{o}bius band.

For this paper, a knot $K$ is in $S^{3}$. The orientable 4-genus of a knot is the minimum genus of an orientable surface in the 4-ball that is bounded by $K$ and is denoted $g_{4}(K)$, and knots with $g_{4}(K) = 0$ are called slice knots. Following Murakami and Yasuhara in \cite{MY}, the non-orientable 4-genus of a knot $K$, denoted $\gamma_{4}(K)$, is defined to be the minimum first Betti number of non-orientable surfaces $F$ smoothly embedded in $B^{4}$ bounded by $K$, that is min$\{b_{1}(F) | \partial F = K\}$. Note that the first Betti number is defined to be $b_{1}(F) = \text{dim} H_{1}(F; \mathbb{Z})$. We have, by definition, for any knot $K$, $\gamma_{4}(K) \geq 1$ where equivalence applies when $K$ bounds a M\"{o}bius band. Slice knots that bound a smooth disk embedded in $B^{4}$ have non-orientable 4-genus one, as we may attach a non-oriented band to such an embedded disk.   

\begin{thm}
    For the 185 non-alternating 11 crossing knots,
        \begin{enumerate}[(a)]
        \item 121 knots have $\gamma_{4}(K) = 1$
        \item 58 knots have $\gamma_{4} (K) = 2$
    \end{enumerate}
    The remaining 6 knots have $\gamma_{4}(K) = 1$ or $2$.
\end{thm}


The paper is organized as follows: Section 2 is the background on knot invariants, double branched covers, and useful bounds and obstructions for the non-orientable 4-genus. Section 3 is a survey of the techniques used to solve this problem as well as results.

\textbf{Acknowledgements.} I would like to thank Chuck Livingston, Pat Gilmer, and Slaven Jabuka for helpful conversations and comments. Additional thanks to Dror Bar-Natan for permitting my use of the Knot Atlas figures \cite{knotatlas}.\\

\section{BACKGROUND}

We begin by reviewing knot invariants and examining bounds for the non-orientable 4-genus as well as obstructions to a knot bounding a M\"{o}bius band. First, the crossing number of a knot is denoted $n(K)$ and is the crossing number of a diagram of a knot with the fewest crossings that could be drawn on the plane to represent the knot. The unknotting number of a knot $u(K)$ is the minimum number of crossing changes required to transform $K$ into the unknot. Similarly, $u_{s}(K)$ is the minimum number of crossing changes to change $K$ into a slice knot. The 4-dimensional clasp number, $c_{4}(K)$, is the minimum number of double points of transversely immersed 2-disks in the 4-ball bounded by $K$ \cite{MY}. We then have the following triple inequality from Jabuka and Kelly \cite{JK}:
$$ g_{4}(K) \leq c_{4}(K) \leq u_{s}(K) \leq u(K) $$

The orientable genus of a knot also offers an upper bound for the non-orientable 4-genus, respective with smooth and topological for $i = 4$, we have \cite{JK}:

$$ \gamma_{i}(K) \leq 2g_{i}(K) + 1 \text{ for } i = 3, 4 $$

Similar to the orientable 4-genus, we obtain an upper bound for the non-orientable 4-genus from the non-orientable 3-genus of a knot called the \textit{crosscap number} \cite{knotinfo}, which is the minimum genus non-orientable surface a knot bounds in $S^{3}$, denoted $c(K)$, so we have $\gamma_{4}(K) \leq c(K)$.

Following the notation of Murakami and Yasuhara \cite{MY95}, we define $\Gamma_{4}(K) = \text{min} \{ b_{1} (F) | \partial F = K \} $, or similarly $\Gamma_{4}(K) =$ min$\{ 2g_{4}(K), \gamma_{4}(K) \}$, and thus $\Gamma_{4}(K) \leq \gamma_{4}(K)$. Murakami and Yasuhara then give us the following proposition \cite{MY}:

\begin{prop}[Proposition 2.3 in \cite{MY}]\label{crosscapBounds}
    For any knot K, the following inequalities hold. 
           \[ \Gamma_{4} (K) \leq 
  \begin{cases*}
    c_{4}(K) & $\text{ if } c_{4}(K) \text{ is even } $ \\
    c_{4}(K) + 1 & \text{ otherwise}
  \end{cases*}\]

       \[ \gamma_{4} (K) \leq 
  \begin{cases*}
    c_{4}(K) & $\text{ if } c_{4}(K) \text{ is even and } c_{4}(K) \neq 2 $\\
    c_{4}(K) + 1 & \text{ otherwise}
  \end{cases*}\]
\end{prop}

\begin{cor}[Corollary 2.4  in \cite{MY}]\label{bigG}
For a knot K, if $g_{4}(K) = c_{4}(K) \geq 1$, then $\Gamma_{4}(K) = \gamma_{4}(K)$.
    
\end{cor}

The crossing number of a knot offers an upper bound, so we have \cite{MY95}:
 $$ \Gamma(K) \leq \left\lfloor \dfrac{n(K)}{2} \right\rfloor \text{ and } \gamma_{4}(K) \leq \left\lfloor \dfrac{n(K)}{2} \right\rfloor $$

The signature of a knot $\sigma (K)$ is defined to be the signature of the sum of knot's Seifert matrix and it's transpose, $\sigma (V + V^{t})$. The Arf invariant of a knot is denoted Arf$(K)$ and is a concordance invariant in $\mathbb{Z}_{2}$ which is calculated using the Seifert form of a knot \cite{knotinfo}. These two invariants form a lower bound for the non-oriented 4-genus of a knot, so we have the following proposition. 

\begin{prop}[Proposition 2.4 in \cite{N10}]\label{sigArf}
    Given a knot $K$, if $\sigma(K) + 4 \rm{Arf}(K) \equiv 4 \pmod{8} $, then $\gamma_{4}(K) \geq 2$.
\end{prop}

\subsection{Double Branched Cover}

Recall the definition of the non-orientable 4-genus is $\gamma_{4}(K) = \text{min} \{ b_{1} (F) | \partial F = K \} $ and note that $b_{1}(F) = \text{dim} H_{1} (F, \mathbb{Q}) $. Let $K$ in $S^{3}$ bound a connected surface $F$ in $B^{4}$ and denote $D_{F}(B^{4})$ as the double branched cover of $B^{4}$ branched over $F$. Gilmer and Livingston proved in \cite{GL}, Lemma 1, that $b_{2}(D_{F}(B^{4})) = b_{1}(F)$. The reasoning here is that the double branched cover of $S^{3}$ branched over $K$, denoted $\DK$, is a rational homology sphere and $H_{1}(\DF; \Q) = 0 $. We thus may use the linking form of $\DK$ to provide information on the intersection form of $\DF$. 

We also have that the first homology of $D_{K}(S^{3})$ is finite, so we have a linking form $\lambda$, and this is explored in detail by Murakami and Yasuhara in \cite{MY} 
$$ \lambda : H_{1}(D_{K}(S^{3}); \mathbb{Z}) \times H_{1}(D_{K}(S^{3}); \mathbb{Z}) \to \mathbb{Q}/ \mathbb{Z}  $$ 

Given a Goeritz matrix $G$ for $K$ (see Section III for details), we have that $G$ is a relation matrix for $H_{1}(D_{K}(S^{3}); \mathbb{Z})$ and the linking form $\lambda$ is given by $\pm G^{-1}$, where the sign depends on orientation of $\DK$ \cite{MY}. The double branched cover is a useful tool in obstructing knots bounding a M\"obius band or a Klein bottle.

\begin{cor}[Corollary 3 in \cite{GL}]\label{linkingNumThm}

Suppose that $H_{1}(D_{K}(S^{3})) = \mathbb{Z}_{n} $ where $n$ is the product of primes, all with odd exponent. Then if $K$ bounds a M\"{o}bius band in $B^{4}$, there is a generator $a \in H_{1} ( D_{K}(S^{3})) $ such that $ \lambda (a, a) = \pm 1/n $
    
\end{cor}

\begin{thm}[Theorem 4 in \cite{GL}]

Suppose that $H_{1}(D_{K}(S^{3})) = \mathbb{Z}_{p} \oplus \mathbb{Z}_{p} $ where $p$ is prime. Then if $K$ bounds a punctured Klein bottle in $B^{4}$, the discriminant of the linking form is $\pm 1 \in \mathbb{F}^{*}_{p} / (\mathbb{F}^{*}_{p})^{2} $
    
\end{thm}

\begin{thm}[Theorem 11 in \cite{GL}] 
    Suppose that $H_{1}(D_{K}(S^{3})) = \mathbb{Z}_{p} \oplus \mathbb{Z}_{p} \oplus \mathbb{Z}_{q} $ where $q \equiv 1 \in \textbf{F}^{*}_{p} / (\textbf{F}^{*}_{p})^{2}$. If $H_{1}(D_{K}(S^{3}))$ is the boundary of a 4-manifold $W$ with second Betti number 2 which has an indefinite intersection form, then the linking form restricted to $\mathbb{Z}_{p} \oplus \mathbb{Z}_{p} \subset H_{1}(D_{K}(S^{3}))$ is metabolic. 
\end{thm}

\section{RESULTS AND TECHNIQUES}

There are a total of 185 knots that are non-alternating and have 11 crossings, according to the knot info database \cite{knotinfo}. Of those knots, there are 16 that are smoothly slice and thus have $\gamma_{4}(K) = 1$.

\begin{rem}
    There are 16 non-alternating 11 crossing knots that are slice and thus bound a M\"obius band:
    $$ 11n_{4}, \hspace{2mm} 11n_{21}, \hspace{2mm}  11n_{37}, \hspace{2mm}  11n_{39}, \hspace{2mm}  11n_{42}, \hspace{2mm}  11n_{49}, \hspace{2mm}  11n_{50}, \hspace{2mm}  11n_{67}, $$ 
    $$  11n_{73}, \hspace{2mm}  11n_{74}, \hspace{2mm} 11n_{83},   \hspace{2mm} 11n_{97}, \hspace{2mm} 11n_{116}, \hspace{2mm} 11n_{132}, \hspace{2mm} 11n_{139}, \hspace{2mm} 11n_{172} $$
\end{rem}

\begin{prop}\label{propforThm1}
    The following knots have $\gamma_{4}(K) = 1$: 
    $$ 11n_{1}, \hspace{2mm} 11n_{3}, \hspace{2mm} 11n_{5}, \hspace{2mm} 11n_{6}, \hspace{2mm} 11n_{7}, \hspace{2mm} 11n_{8}, \hspace{2mm} 11n_{9}, \hspace{2mm} 11n_{11}, \hspace{2mm} 11n_{13}, \hspace{2mm} 11n_{14}, \hspace{2mm} 11n_{15},  $$
    $$   \hspace{2mm} 11n_{16}, \hspace{2mm} 11n_{18}, \hspace{2mm} 11n_{19}, \hspace{2mm} 11n_{20}, \hspace{2mm} 11n_{23}, \hspace{2mm} 11n_{24}, \hspace{2mm} 11n_{25}, \hspace{2mm} 11n_{26}, \hspace{2mm} 11n_{27}, \hspace{2mm} 11n_{31}, \hspace{2mm} 11n_{34}, $$
    $$  \hspace{2mm} 11n_{36}, \hspace{2mm} 11n_{41}, \hspace{2mm} 11n_{44}, \hspace{2mm} 11n_{45}, \hspace{2mm} 11n_{46}, \hspace{2mm} 11n_{47}, \hspace{2mm} 11n_{52}, \hspace{2mm} 11n_{54}, \hspace{2mm} 11n_{57}, \hspace{2mm} 11n_{58}, \hspace{2mm} 11n_{59}, $$
    $$  \hspace{2mm} 11n_{60}, \hspace{2mm} 11n_{62}, \hspace{2mm} 11n_{64}, \hspace{2mm} 11n_{65}, \hspace{2mm} 11n_{66}, \hspace{2mm} 11n_{68}, \hspace{2mm} 11n_{69}, \hspace{2mm} 11n_{70}, \hspace{2mm} 11n_{71}, \hspace{2mm} 11n_{75}, \hspace{2mm} 11n_{76}, $$ 
    $$  \hspace{2mm} 11n_{77}, \hspace{2mm} 11n_{78}, \hspace{2mm} 11n_{79}, \hspace{2mm} 11n_{80}, \hspace{2mm} 11n_{81}, \hspace{2mm} 11n_{82}, \hspace{2mm} 11n_{86}, \hspace{2mm} 11n_{87}, \hspace{2mm} 11n_{88}, \hspace{2mm}  11n_{89}, \hspace{2mm} 11n_{91}, $$ 
    $$ \hspace{2mm} 11n_{93}, \hspace{2mm} 11n_{94}, \hspace{2mm} 11n_{96}, \hspace{2mm} 11n_{102}, \hspace{2mm} 11n_{104}, \hspace{2mm}  11n_{105}, \hspace{2mm} 11n_{106}, \hspace{2mm} 11n_{107}, \hspace{2mm} 11n_{110}, \hspace{2mm} 11n_{111},  $$
    $$  \hspace{2mm} 11n_{113}, \hspace{2mm} 11n_{117}, \hspace{2mm} 11n_{118}, \hspace{2mm} 11n_{120}, \hspace{2mm} 11n_{121}, \hspace{2mm} 11n_{122}, \hspace{2mm} 11n_{123}, \hspace{2mm} 11n_{124}, \hspace{2mm} 11n_{126}, \hspace{2mm} 11n_{127},  $$
    $$  \hspace{2mm} 11n_{128}, \hspace{2mm} 11n_{129}, \hspace{2mm} 11n_{134}, \hspace{2mm} 11n_{135}, \hspace{2mm} 11n_{136}, \hspace{2mm} 11n_{142}, \hspace{2mm} 11n_{143}, \hspace{2mm} 11n_{145}, \hspace{2mm} 11n_{146}, \hspace{2mm} 11n_{147},  $$ 
    $$ \hspace{2mm} 11n_{148},  \hspace{2mm} 11n_{150}, \hspace{2mm} 11n_{151}, \hspace{2mm} 11n_{152}, \hspace{2mm} 11n_{153}, \hspace{2mm} 11n_{154}, \hspace{2mm} 11n_{157}, \hspace{2mm} 11n_{158}, \hspace{2mm} 11n_{160}, \hspace{2mm} 11n_{162}, \hspace{2mm}  $$ 
    $$ 11n_{163}, \hspace{2mm} 11n_{164},   \hspace{2mm} 11n_{167}, \hspace{2mm} 11n_{168}, \hspace{2mm} 11n_{169}, \hspace{2mm} 11n_{170}, \hspace{2mm} 11n_{173}, \hspace{2mm} 11n_{180}, \hspace{2mm} 11n_{181}, \hspace{2mm} 11n_{183} $$
    
\end{prop}


\begin{prop}\label{propforThm2}
      The following knots have $\gamma_{4}(K) = 2$: 
    $$ 11n_{2}, \hspace{2mm} 11n_{10}, \hspace{2mm} 11n_{12}, \hspace{2mm}  11n_{22}, \hspace{2mm} 11n_{28}, \hspace{2mm} 11n_{29}, \hspace{2mm} 11n_{30},  \hspace{2mm} 11n_{32}, \hspace{2mm} 11n_{33}, \hspace{2mm} 11n_{35}, $$
    $$   \hspace{2mm} 11n_{38}, \hspace{2mm} 11n_{43}, \hspace{2mm} 11n_{48}, \hspace{2mm} 11n_{51}, \hspace{2mm} 11n_{53}, \hspace{2mm}  11n_{55}, \hspace{2mm} 11n_{56}, \hspace{2mm} 11n_{61}, \hspace{2mm} 11n_{63}, \hspace{2mm} 11n_{72},  $$
    $$ \hspace{2mm} 11n_{84}, \hspace{2mm} 11n_{85}, \hspace{2mm} 11n_{90}, \hspace{2mm} 11n_{92}, \hspace{2mm} 11n_{95}, \hspace{2mm} 11n_{98}, \hspace{2mm} 11n_{99}, \hspace{2mm} 11n_{100}, \hspace{2mm} 11n_{101}, \hspace{2mm} 11n_{103}, $$
    $$  \hspace{2mm} 11n_{108},  \hspace{2mm} 11n_{109}, \hspace{2mm} 11n_{112}, \hspace{2mm} 11n_{114}, \hspace{2mm} 11n_{115}, \hspace{2mm} 11n_{119}, \hspace{2mm}  11n_{125}, \hspace{2mm} 11n_{130}, \hspace{2mm} 11n_{131},  \hspace{2mm} 11n_{133}, $$ 
    $$  \hspace{2mm} 11n_{137},  \hspace{2mm} 11n_{138}, \hspace{2mm} 11n_{140}, \hspace{2mm} 11n_{141}, \hspace{2mm} 11n_{144}, \hspace{2mm} 11n_{149}, \hspace{2mm} 11n_{155}, \hspace{2mm} 11n_{156}, \hspace{2mm} 11n_{161}, \hspace{2mm} 11n_{165}   $$
    $$ \hspace{2mm} 11n_{171}, \hspace{2mm} 11n_{174}, \hspace{2mm} 11n_{175}, \hspace{2mm} 11n_{176}, \hspace{2mm} 11n_{179}, \hspace{2mm} 11n_{182}, \hspace{2mm} 11n_{184}, \hspace{2mm} 11n_{185}, $$
\end{prop}

\subsection{Constraints on Invariants} 
The knot invariant information for this paper was extracted from Knot Info \cite{knotinfo}.
\begin{lem}\label{1pfofT2}
    Given K is a knot satisfying $\sigma (K) + 4 \rm{Arf} (K) \equiv 4 \pmod{8}$, and $c_{4} (K) = 1$, then $\gamma_{4}(K) = 2$. 
\end{lem}

The result is clear from Proposition \ref{crosscapBounds} and Corollary \ref{sigArf}. We now examine knots that have $g_{4}(K) = u(K) = 1 $ (or $g_{4}(K) = u_{s}(K) = 1 $ )  and $\sigma (K) + 4 \rm{Arf} (K) \equiv 4 \pmod{8}$ to see the following knots have $\gamma_{4}(K) = 2$:

    $$ 11n_{12}, \hspace{2mm} 11n_{28}, \hspace{2mm} 11n_{48}, \hspace{2mm} 11n_{53}, \hspace{2mm} 11n_{55}, \hspace{2mm} 11n_{85}, \hspace{2mm} 11n_{100} $$ 
    $$\hspace{2mm} 11n_{114}, \hspace{2mm} 11n_{115}, \hspace{2mm} 11n_{119}, \hspace{2mm} 11n_{130}, \hspace{2mm} 11n_{156}, \hspace{2mm} 11n_{179}, \hspace{2mm} 11n_{182} $$

All the above listed knots satisfy $\sigma (K) + 4 \rm{Arf}(K) \equiv 4 \pmod{8}$. Since they satisfy $g_{4}(K) = 1 = u(K)$ (or $g_{4}(K) = u_{s}(K) = 1 $ ) by the hypothesis, we have $c_{4}(K) = 1$, and thus by Lemma \ref{1pfofT2} we may conclude $\gamma_{4}(K) = 2$. 

\begin{lem}\label{2pfofT2}
 Given K is a knot satisfying $\sigma (K) + 4 \rm{Arf}(K) \equiv 4 \pmod{8}$, and $c_{4} (K) = 2$, then $\gamma_{4}(K) = 2$.
   
\end{lem}

 By Corollary \ref{bigG}, we have $\Gamma_{4}(K) = \gamma_{4}(K)$, and thus applying Proposition \ref{crosscapBounds} we achieve $\gamma_{4}(K) \leq 2$. Therefore, $\gamma_{4}(K) = 2$. We now observe that the following knots have $\gamma_{4}(K) = 2$:
    $$ 11n_{2}, \hspace{2mm} 11n_{35}, \hspace{2mm} 11n_{95}, \hspace{2mm}  11n_{103}, \hspace{2mm} 11n_{108}, \hspace{2mm} 11n_{109}, \hspace{2mm} 11n_{144}, \hspace{2mm} 11n_{149}, \hspace{2mm} 11n_{174}, \hspace{2mm} 11n_{175}, \hspace{2mm} 11n_{185} $$
    The above listed knots all satisfy $\sigma (K) + 4  \rm{Arf}(K) \equiv 4 \pmod{8}$ and thus $\gamma_{4}(K) \geq 2$. Additionally, these knots all satisfy $g_{4}(K) = u(K) = 2$, and thus $c_{4}(K) = 2$.

\subsection{Non-Oriented Band Moves} 

The primary method used in calculations was via non-oriented band moves. We begin with an oriented knot $K$ and an oriented band, $[0, 1] \times [0, 1]$. Following the conventions of Jabuka and Kelly \cite{JK}, we attach the band to $K$ in the sense that the orientation of the band agrees with the orientation of $K$ on $[0, 1] \times \{0\}$ but disagrees on $[0, 1] \times \{1\}$, or vise versa. One then does surgery along the band. The result of non-orientable band surgery will always be a knot, while the result after \textit{orientable} band surgery is a link. Non-orientable band surgery is explored by Moore and Vazquez in \cite{MV} and is called \textit{non-coherent band surgery}.\\

The notation for a knot $K$ that has been transformed into a knot $K'$ by a non-oriented band move is $K \stackrel{h}{\longrightarrow} K'$ where $h$ is either 0, 1, or -1. These three band moves can be seen in the figure below. From left to right, we have $\stackrel{0}{\longrightarrow}$ is the band move without a twist, $\stackrel{-1}{\longrightarrow}$ is the band move with a left-handed twist, and $\stackrel{1}{\longrightarrow}$ is the band move with a right handed twist. 

\begin{figure}[h]
    \centering
    \includegraphics[width = 3in]{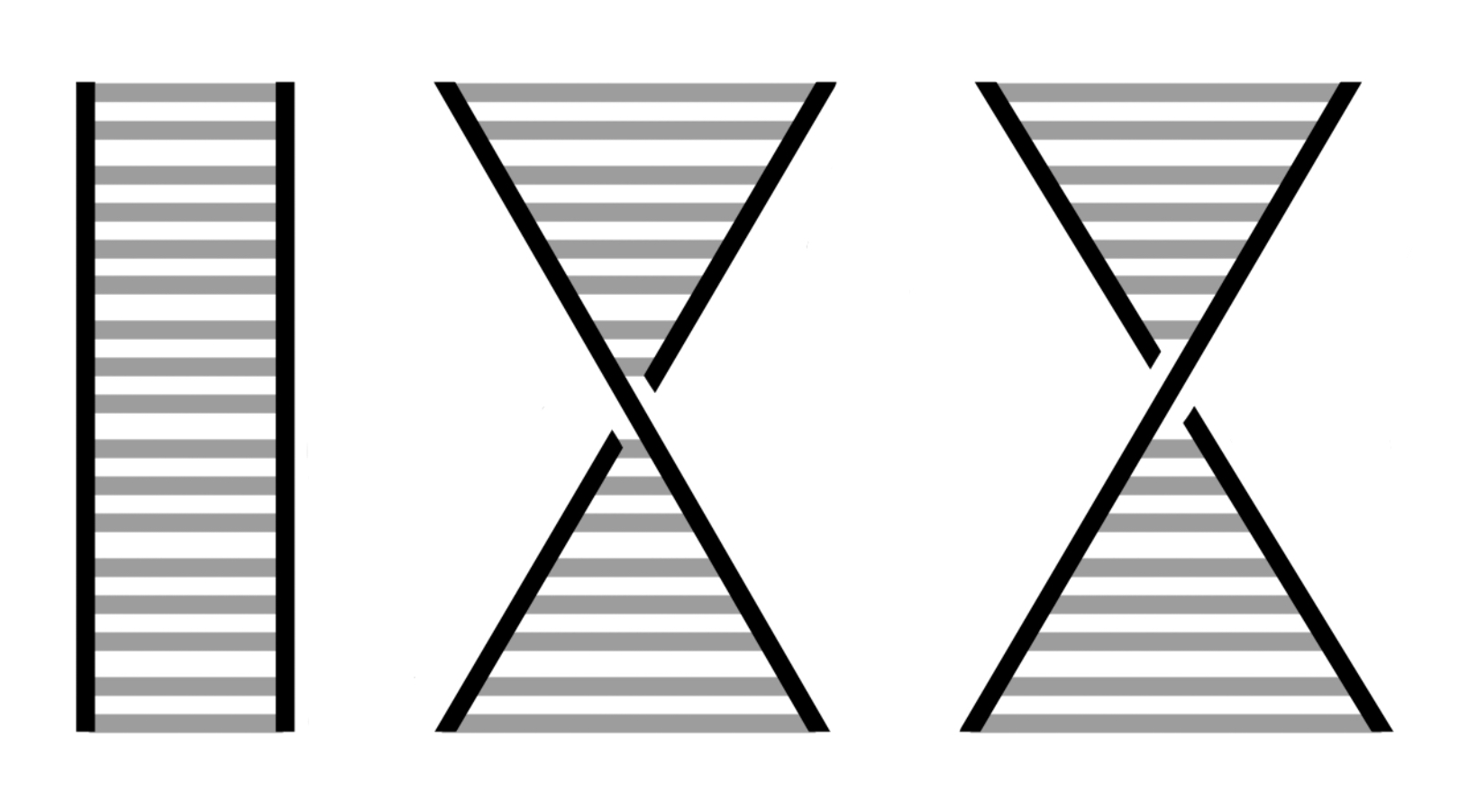}
    \caption{Band Moves}
    \label{fig:enter-label}
\end{figure}

\begin{prop}[Proposition 2.4 in \cite{JK}]\label{BMbound}

If the knots $K$ and $K'$ are related by a non-oriented band move, then 
$$ \gamma_{4}(K) \leq \gamma_{4}(K') + 1 $$

If a knot $K$ is related to a slice knot $K'$ by a non-oriented band move, then $\gamma_{4}(K) = 1$.
    
\end{prop}

\textit{Proof of Theorem 1.1 part (a)}. Every knot listed in Proposition \ref{propforThm1} is either a slice knot or one non-oriented band move away from a slice knot. See Figure~\ref{firstSlice} - Figure~\ref{lastSlice} for details.

\begin{lem}\label{3pfofT2}
    The following knots have $\gamma_{4}(K) = 2$:

     $$ 11n_{10}, \hspace{2mm} 11n_{12}, \hspace{2mm} 11n_{30}, \hspace{2mm} 11n_{32}, \hspace{2mm} 11n_{43}, \hspace{2mm} 11n_{48}, \hspace{2mm} 11n_{51}, \hspace{2mm} 11n_{55}, \hspace{2mm} 11n_{61}, \hspace{2mm} 11n_{72} $$
$$ 11n_{85}, \hspace{2mm} 11n_{90}, \hspace{2mm} 11n_{98}, \hspace{2mm} 11n_{103}, \hspace{2mm} 11n_{130}, \hspace{2mm} 11n_{133} $$
    
\end{lem}

We now recall Proposition \ref{sigArf} and note the knots listed in the above lemma all satisfy $\sigma(K) + 4 \rm{Arf}(K) \equiv 4 \pmod{8} $.
So we know the above knots have $\gamma_{4}(K) \geq 2$. The above listed knots all are one non-oriented band move away from a knot $K'$ so that $\gamma_{4}(K') = 1$ (see Figure~\ref{first1G} - Figure~\ref{last1G}), thus we conclude $\gamma_{4}(K) = 2$.

\subsection{Linking Form Calculation}

We look for a knot $K$ so that $\sigma(K) + 4  \rm{Arf}(K) \equiv 0, \pm 2 \pmod{8} $, and thus $K$ does not meet the obstruction from Proposition \ref{sigArf}. We calculate the linking form of $H_{1}(D_{K}(S^{3}))$ to see if $K$ meets the obstruction from Corollary \ref{linkingNumThm}. The first thing we do is calculate the Goeritz matrix for $K$. We will do an example here, but an interested reader is referred to Gordan and Litherland \cite{GorLit}.

To construct the Goeritz matrix, we first make a checkerboard coloring of a knot. \\

\begin{figure}[h] 
    \centering
    \includegraphics[width = 2in]{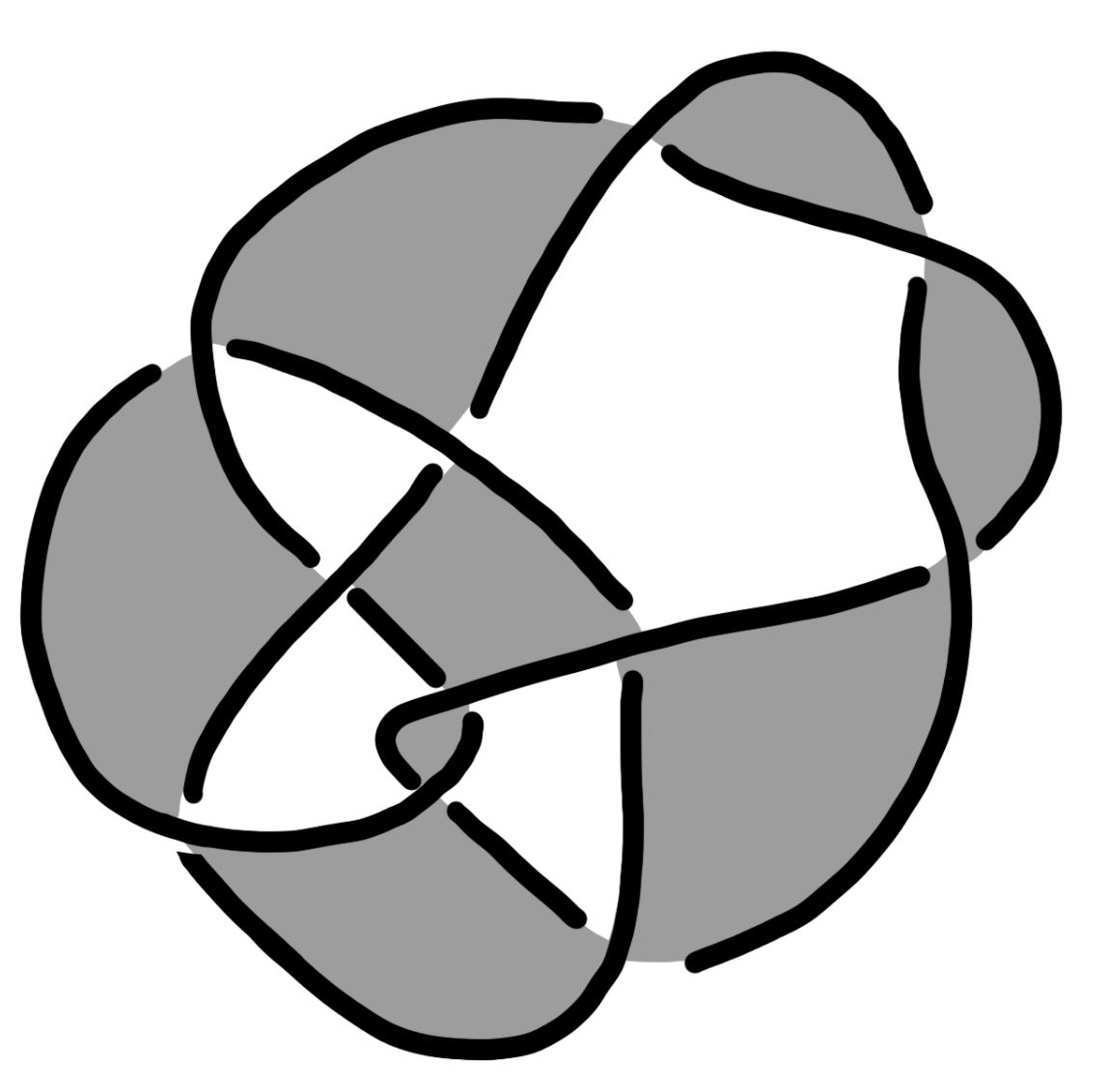}
    \caption{Checkerboard coloring for $11n_{155}$}
    \label{fig:enter-label}
\end{figure}

Each white region is labeled $R_{i}$ and the unbounded region is $R_{0}$. We then assign a value to each crossing $C$, $\eta (C) = \pm 1$, via the figure below, and following the conventions from Gordan and Litherland \cite{GorLit}. 

\begin{figure}[h]
    \centering
    \includegraphics[width = 2.5in]{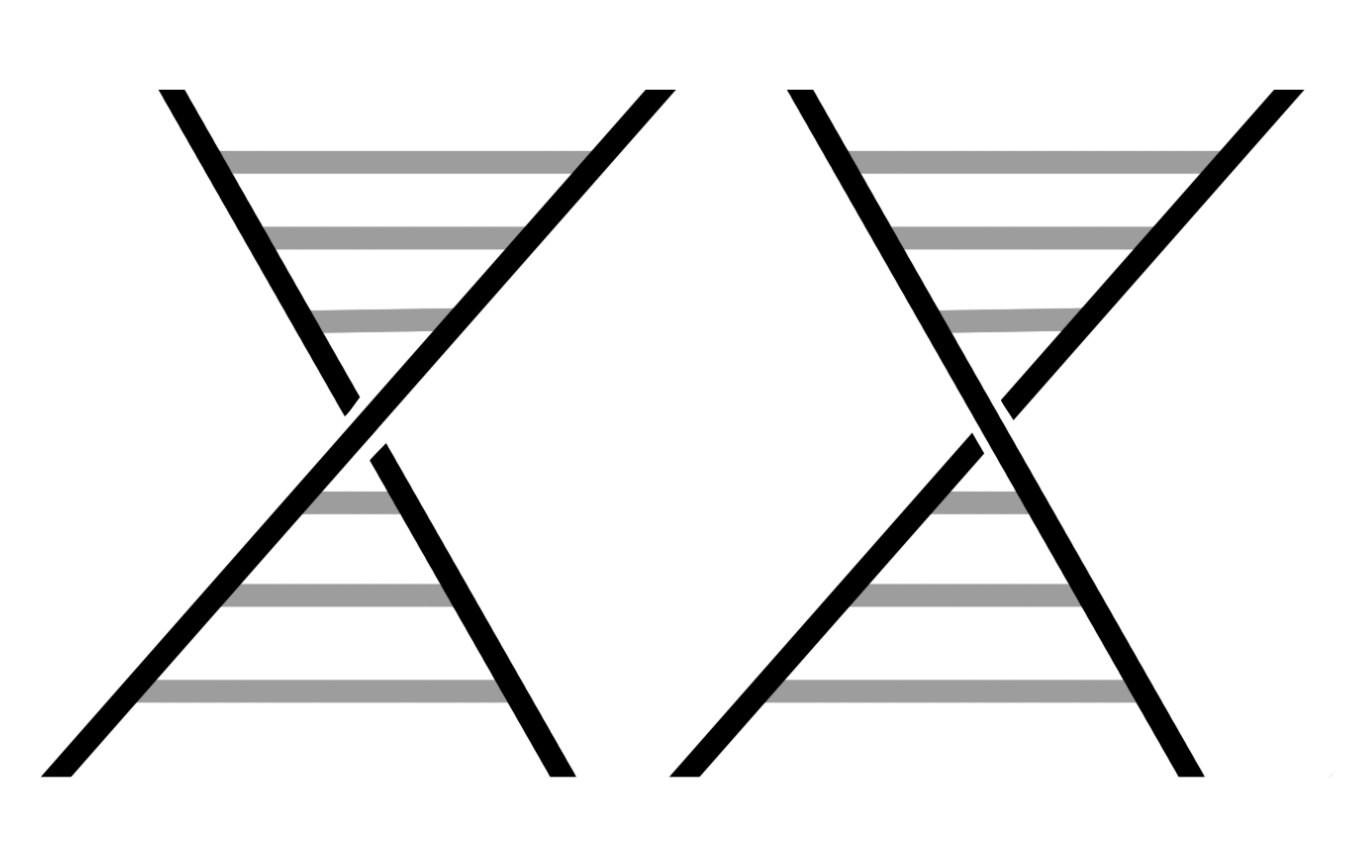}
    \caption{left: $\eta (C) = 1$, right: $\eta (C) = -1$ }
    \label{fig:enter-label}
\end{figure}

Next, we construct a matrix $G'$ with the algorithm: 

\[ g'(i, j) = 
  \begin{cases*}
    - \sum \eta (C) \text{ where the sum ranges over all crossings } C \text{ incident to } R_{i} \text{ and } R_{j}, i \neq j \\
    - \displaystyle\sum_{k\neq i} g'(i, k) = g'(i, i) \hspace{2mm} \text{if } i = j
  \end{cases*}\]

Then, the Goeritz matrix $G$ is obtained from $G'$ by deleting the $0^{th}$ row and column. The determinant of $G$ is an invariant of the knot, and $G$ is a linking matrix for $H_{1}(D_{K}(S^{3}))$ \cite{GorLit, MY}. 

Now, we may calculate the linking form. As previously mentioned, $\pm G^{-1}$ represents the linking form $\lambda$ where $ \lambda : H_{1}(D_{K}(S^{3}); \mathbb{Z}) \times H_{1}(D_{K}(S^{3}); \mathbb{Z}) \to \mathbb{Q}/ \mathbb{Z}  $. To continue the example, we have $G$ and $G^{-1}$ for the knot $11n_{155}$ as: 
$$ G = 
\begin{bmatrix}
    3 & -1 & 0 & -1 \\
    -1 & 5 & -1 & 0 \\
    0 & -1 & 0 & 2 \\
    -1 & 0 & 2 & 0
\end{bmatrix}
\hspace{15mm}G^{-1} = 
\begin{bmatrix}
    \frac{20}{51} & \frac{2}{17} & \frac{10}{51} & \frac{1}{17} \\
    \frac{2}{17} & \frac{4}{17} & \frac{1}{17} & \frac{2}{17} \\
    \frac{10}{51} & \frac{1}{17} & \frac{5}{51} & \frac{9}{17} \\
    \frac{1}{17} & \frac{2}{17} & \frac{9}{17} & \frac{1}{17}
\end{bmatrix}
$$

Now we have the linking form $\lambda(g, g) = \pm 20/51$. Suppose $11n_{155}$ bounds a M\"{o}bius band. We wish to find an $n \in \mathbb{Z}$ so that $\lambda(ng, ng) = \pm 1/51$. This means $\pm 20/51 = \lambda(ng, ng) = n^{2}\lambda(g, g) = \pm 20 n^{2}/51 = \pm 1/51$, so $20n^{2} \equiv \pm 1 \pmod{51} $. A quick calculation shows this is not possible, and thus $11n_{155}$ does not bound a M\"{o}bius band. 

\subsection{Results}

\begin{thm}[Theorem 2 in \cite{GL}]\label{lformDirectSum}

Let $K$ in $S^{3}$ be a knot. The linking form $(H_{1} (D_{K}(S^{3}), \lambda)$ splits as a direct sum $(G_{1}, \lambda_{1} ) \oplus (G_{2}, \lambda_{2}) $ where $(G_{2}, \lambda_{2})$ is metabolic and $(G_{1}, \lambda_{1} )$ has a presentation of rank $\lambda_{1}(F)$. 

\end{thm}

\begin{lem}\label{myLemma1}

Let $K$ in $S^{3}$ be a knot and suppose that $H_{1}( \DK ) = \Z_{p^{2}q}$ where p is prime and q is a product of primes, all with odd exponent. Then if $K$ bounds a M\"{o}bius band in $B^{4}$, there is a generator $a \in H_{1} (\DK ) $ such that either $ \lambda (a, a) = \pm 1/p^{2}q$ or $ \lambda (a, a) = \pm 1/q$.
    
\end{lem}

\begin{proof}
    As we see in Theorem \ref{lformDirectSum}, $(H_{1} (\DK) , \lambda)$ splits as a direct sum $(G_{1}, \lambda_{1}) \oplus (G_{2}, \lambda_{2})$ where $(G_{2}, \lambda_{2})$ is metabolic and $\lambda_{1}$ is presented by the linking matrix of $\DK$, which has a presentation of rank one. As $q$ is square-free, we have that $\Z_{q}$ is completely contained in $G_{1}$. Then either $\Z_{p^{2}}$ is completely contained in $G_{2}$, which implies it is metabolic, or $\Z_{p^{2}}$ is contained in $G_{1}$. \\
    If $\Z_{p^{2}}$ is completely contained in $G_{2}$, then there exists a subgroup $H$ of $\Z_{p^{2}}$ so that $|H|^{2} = p^{2}$ and $\lambda(g, g') = 0$ for any $g, g' \in H$, since $\lambda_{2}$ is metabolic. Then, as $\lambda_{1}$ must have a presentation of rank one, we have that the presentation matrix must be of the form $(\pm |G_{1}|) = (\pm q)$. Therefore, the linking form $\lambda_{1}$ on $G_{1}$ is given by $\pm1/q$.  \\
    If $\Z_{p^{2}}$ is completely contained in $G_{1}$, a similar argument shows $\lambda_{1}$ is given by $\pm 1/q$
\end{proof}

The following knots: 
$$ 11n_{22}, \hspace{2mm} 11n_{29} , \hspace{2mm} 11n_{33}, \hspace{2mm} 11n_{56}, \hspace{2mm} 11n_{84}, \hspace{2mm} 11n_{92}, \hspace{2mm} 11n_{101}, \hspace{2mm} 11n_{112}, \hspace{2mm} 11n_{125}, \hspace{2mm} 11n_{131},  \hspace{2mm} 11n_{138}, \hspace{2mm} 11n_{155}, $$ 
$$ 11n_{176}, \hspace{2mm} 11n_{184}  $$

have the respective linking forms: 

$$ \dfrac{42}{55}, \hspace{2mm} \dfrac{14}{51}, \hspace{2mm} \dfrac{22}{51}, \hspace{2mm} \dfrac{12}{35}, \hspace{2mm} \dfrac{18}{35}, \hspace{2mm} \dfrac{2}{15}, \hspace{2mm} \dfrac{19}{39}, \hspace{2mm} \dfrac{53}{55}, \hspace{2mm} \dfrac{61}{63}, \hspace{2mm} \dfrac{39}{67}, \hspace{2mm} \dfrac{13}{15}, \hspace{2mm} \dfrac{20}{51}, \hspace{2mm} \dfrac{11}{63}, \hspace{2mm} \dfrac{2}{87} $$
\vspace{2mm}

All of which satisfy the obstruction from Corollary \ref{linkingNumThm} and Lemma \ref{myLemma1}. Additionally, all of these knots have an non-orientable band move to a knot $K'$ where $\gamma_{4}(K') = 1$ (Figures \ref{first1G} -\ref{last1G}). Thus, each of these knots has non-orientable 4-genus equal to 2.

\subsection{Knot Floer Homology}

Ozsv\'ath, Stipsicz, and Szab\'o explored non-orientable knot floer homology and how the Upsilon invariant can be used as a lower bound for the non-orientable 4-genus \cite{OS2}. Given $K$ is a knot, denote $\Upsilon_{K}(1) $ as $\upsilon (K)$ (lower case upsilon), and then we have: 
\begin{equation*}
    \left| \upsilon(K) - \frac{\sigma(K)}{2} \right| \leq \gamma_{4} (K)
\end{equation*}

However, if $K$ is not an L-space knot, this invariant is rather difficult to compute. Additionally, we have from \cite{OS2} that for an alternating (or quasi-alternating) knot $K$, 
\begin{equation*}
    \upsilon(K) = \frac{\sigma(K)}{2}
\end{equation*}

For the 185 non-alternating 11-crossing knots, only 3 are not quasi-alternating. Of those 3, two are slice and one is not. This is thus not a useful lower bound for the knots being considered in this paper. However, this is a useful invariant for torus knots, demonstrated in detail by Binns, Kang, Simone, and Tru\"ol  in \cite{torus1}. Additionally, Allen explored a geography problem where the upsilon invariant was wonderfully utilized in \cite{Allen}.

\section{SPECIAL CASES}

\begin{lem}\label{11n38}
    The knot $11n_{38}$ does not bound a M\"{o}bius band. 
\end{lem}

The knot $11n_{38}$ has $H_{1}(\DK ) = \Z_{3}$ and thus the linking form is represented by the $1 \times 1$ matrix $[ 1/3]$. This is clear, as the non-zero elements of $\Z_{3}$ are 1 and -1. Then, if $K$ bounds a M\"{o}bius band $F$ in $B^{4}$, we have $b(F) = b(\DF) = 1$ and $\DF$ is negative definite \cite{GL}. From Theorem 3 in \cite{GorLit}, we have that the intersection form on $H_{2}(\DF))$ is represented by the linking matrix on $H_{1} (\DK)$, which can be viewed from the entries in the Goeritz matrix. The Goeritz matrix $G$ is a $4 \times 4$ matrix that is indefinite, and when diagonalized, $G = SJS^{-1}$, the matrix $J$ is also indefinite. We may suppose that there exists a presentation matrix that represents the linking form, and by checking the diagonal entries on $-G^{-1}$, we have that $1/3$ represents the form. This implies the manifold is positive definite, which is a contradiction. Thus, $11n_{38}$ does not bound a M\"{o}bius band. We then have that there is a non-orientable band move from $11n_{38}$ to the trefoil knot, which has $\gamma_{4}(3_{1}) = 1$, therefore we may conclude that $\gamma_{4}(11n_{38}) = 2$. The figure below was obtained from Knot Atlas \cite{knotatlas}.

\begin{figure}[h]
    \centering
    \includegraphics[width = 0.32\textwidth]{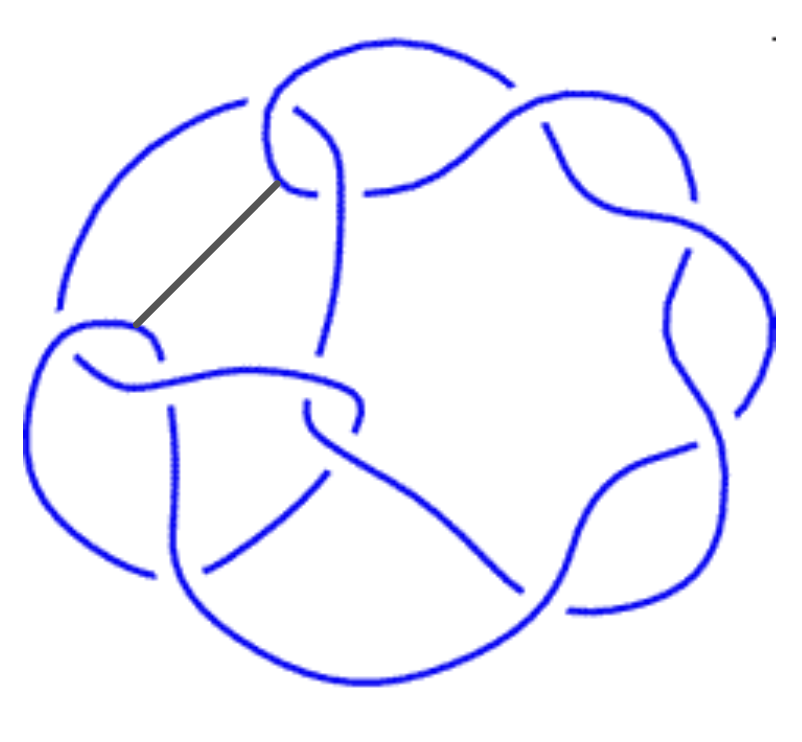}
    \caption{A non-oriented band move from $11n_{38}\stackrel{0}{\longrightarrow} 3_1$}
    \label{fig:enter-label}
\end{figure}

We thus have a combination of Lemmas \ref{1pfofT2}, \ref{2pfofT2}, and \ref{11n38}, Proposition \ref{BMbound}, and Theorem \ref{lformDirectSum} showing Proposition \ref{propforThm2} is true, thus proving part (b) of Theorem 1.1.

\begin{lem}
    The knots $11n_{17}, \hspace{2mm} 11n_{40}, \hspace{2mm} 11n_{159}, \hspace{2mm} 11n_{166}, \hspace{2mm} 11n_{177} $ and $11n_{178}$ all have $\gamma_{4}(K) =$ 1 or 2. 
\end{lem}

We have the following table:
    \renewcommand\arraystretch{2.1}
\begin{center} 
\begin{tabular}{ |c|c|c|c| } 
 \hline
 Knot & linking form & definiteness of $\DF$ & 4-genus  \\ 
 \hline
 $11n_{17}$ & $1/47$ & positive & 1 \\ 
 \hline
 $11n_{40}$ & $-1/79$ & negative & 1 \\
 \hline
 $11n_{159} $ & $1/71$ & positive & 1 \\
 \hline
 $11n_{166} $ & $1/59$ & positive & 1 \\
 \hline
 $11n_{177} $ & $1/83$ & positive & 1 \\
 \hline
 $11n_{178} $ & $-1/95$ & negative & 1 \\
 \hline

\end{tabular}
\end{center}

\begin{proof}
    Denote $K$ as a knot listed in Lemma 4.2. We first examine the knot signature and Arf invariant to see $\sigma(K) + 4 \rm{Arf}(K) \equiv \pm 2 \pmod{8} $. Thus, we do not meet the obstruction from Proposition 2.3, so we may only conclude $\gamma_{4} (K) \geq 1$. We then move on to examining the linking form of $K$. Note that the determinant of $K$, $d = \text{det}(K)$, is either a prime number or a product of exactly 2 prime numbers. As $d = | H_{1} (\DK ) |$, we cannot have a splitting of $H_{1}(\DK)$ into $G_{1} \oplus G_{2}$ where $G_{2}$ is metabolic, since $d$ is square free. We thus see that the linking form $\lambda$ for each knot is of the form $\pm 1/ d$. We also compare the linking form of the knot to the definitness of $\DF$. The sign of the 4-manifold $\DF$ corresponds to the sign of the quadratic form \cite{GL}, thus the linking form, and we see that our signs are corresponding for the linking form and definiteness of $\DF$. Additionally, each knot is one band move away from a knot $K'$ so that $\gamma_{4}(K') = 1$, see Figure \ref{problems}, and thus $\gamma_{4}(K) \leq 2$. We thus cannot find an obstruction to these knots bounding a M\"{o}bius band, but also cannot find the desired band move to a slice knot. Therefore, $\gamma_{4}(K) \leq 2$ for the knots in Lemma 4.2. 
\end{proof}

This concludes the proof for Theorem 1.1.\\

\begin{figure}[!htbp] 
	\centering
	\begin{subfigure}[b]{0.25\textwidth}
		\includegraphics[width=\textwidth]{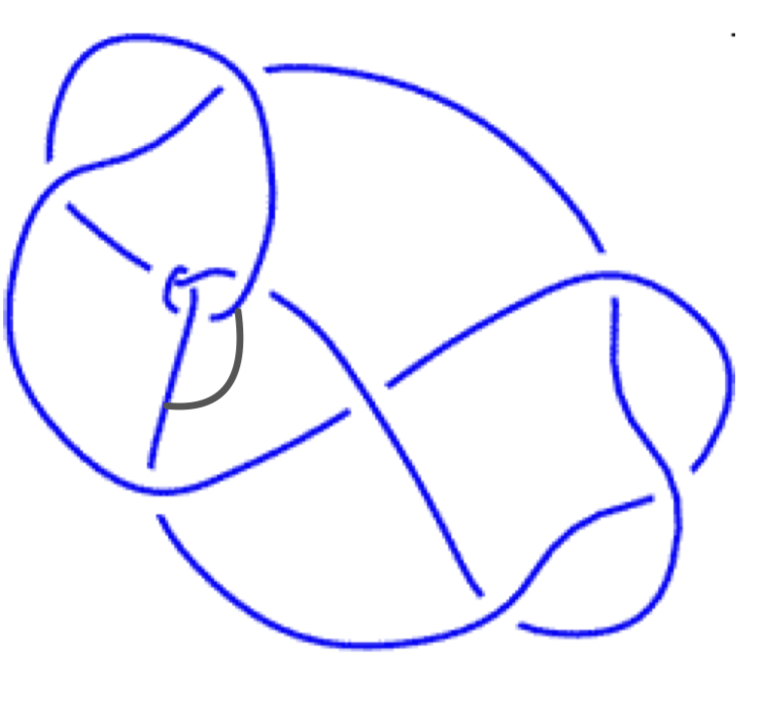}
		\caption{$11n_{17}\stackrel{1}{\longrightarrow} 10_{130}$}
		
	\end{subfigure}
	~
	\begin{subfigure}[b]{0.25\textwidth}
		\includegraphics[width=\textwidth]{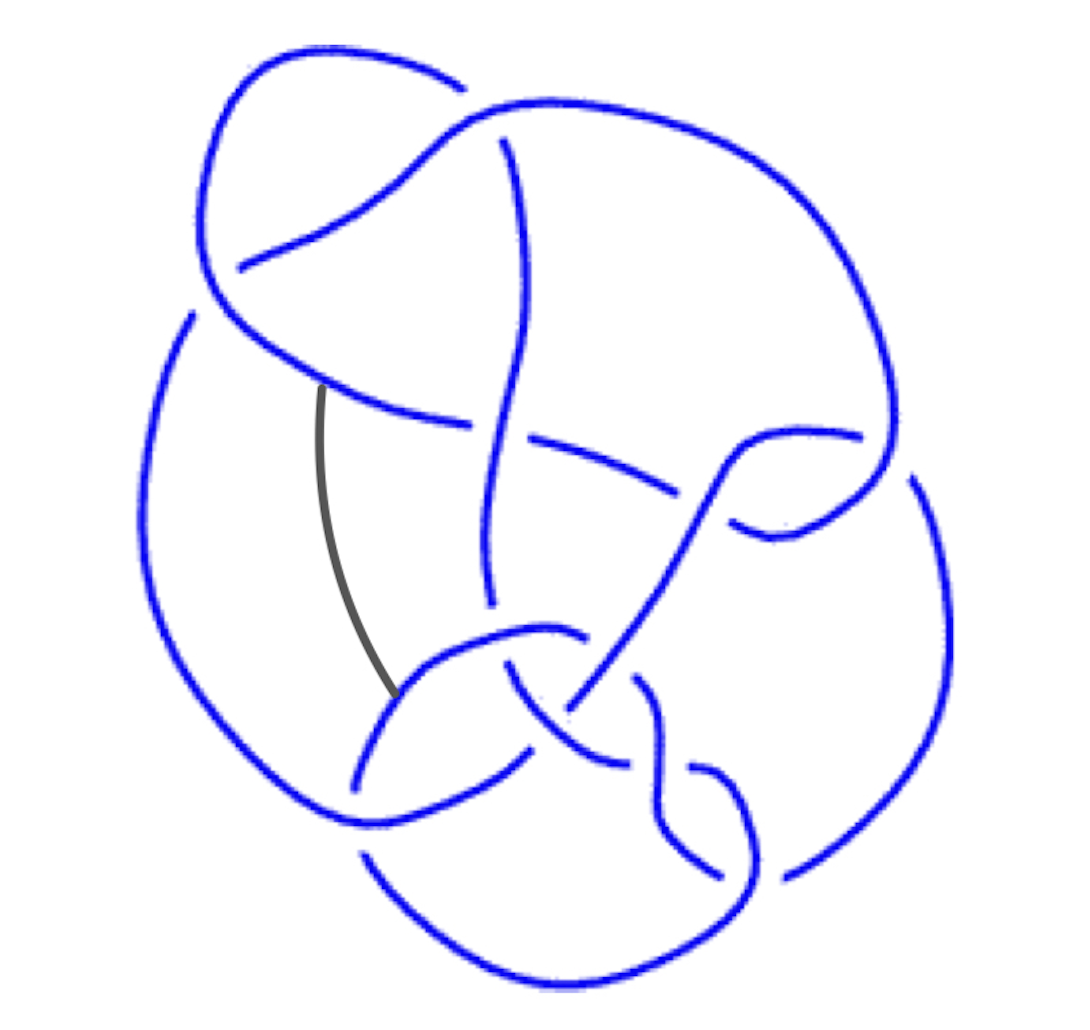}
		\caption{$11n_{40}\stackrel{-1}{\longrightarrow} 8_{4}$}
		
	\end{subfigure}
	~
	\begin{subfigure}[b]{0.25\textwidth}
		\includegraphics[width=\textwidth]{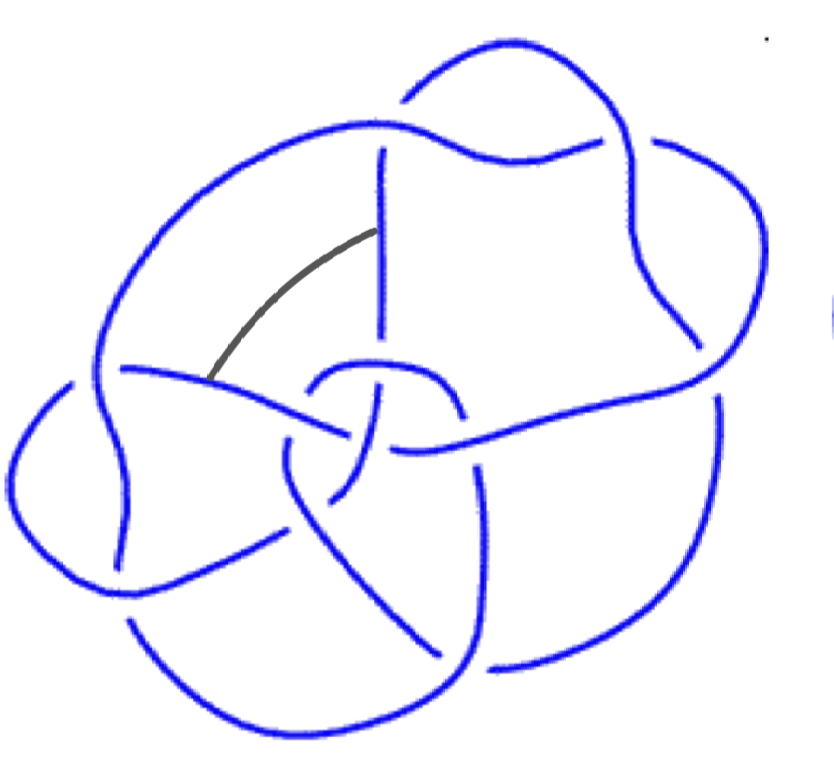}
		\caption{$11n_{159}\stackrel{0\phantom{i}}{\longrightarrow} 3_{1}$}
		
	\end{subfigure}
	\vskip3mm
	\begin{subfigure}[b]{0.25\textwidth}
		\includegraphics[width=\textwidth]{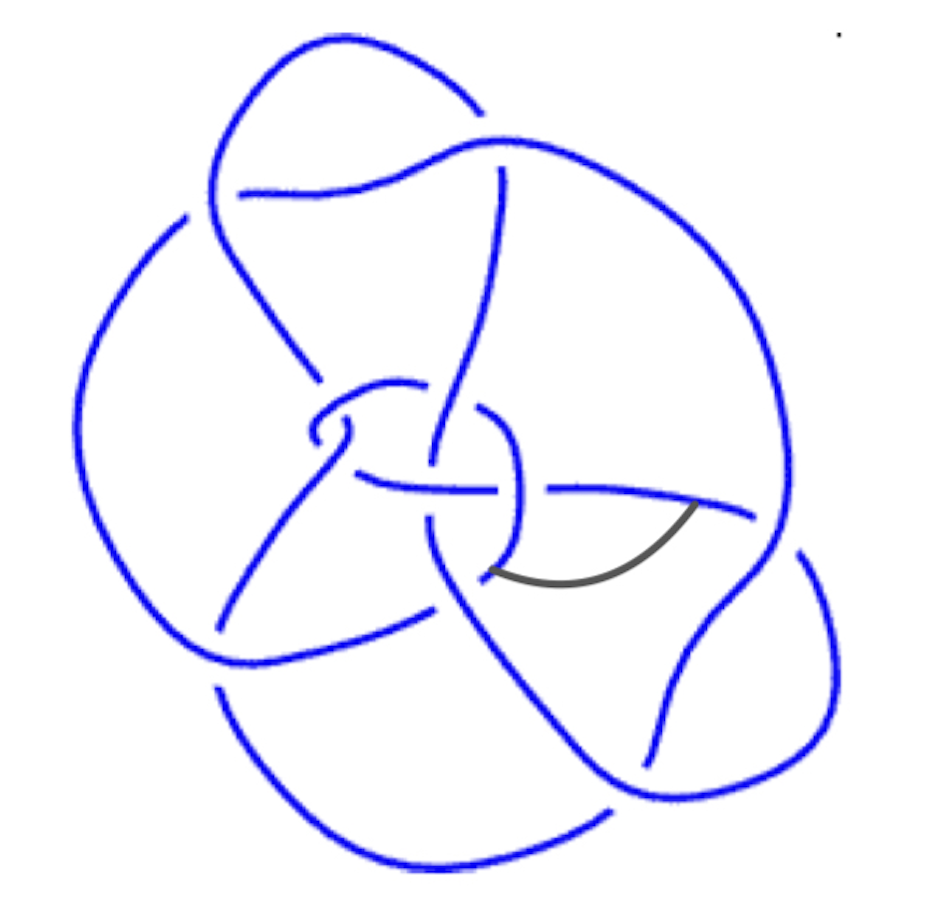}
		\caption{$11n_{166}\stackrel{1}{\longrightarrow} 10_{142}$}
		
	\end{subfigure}
	~
	\begin{subfigure}[b]{0.25\textwidth}
		\includegraphics[width=\textwidth]{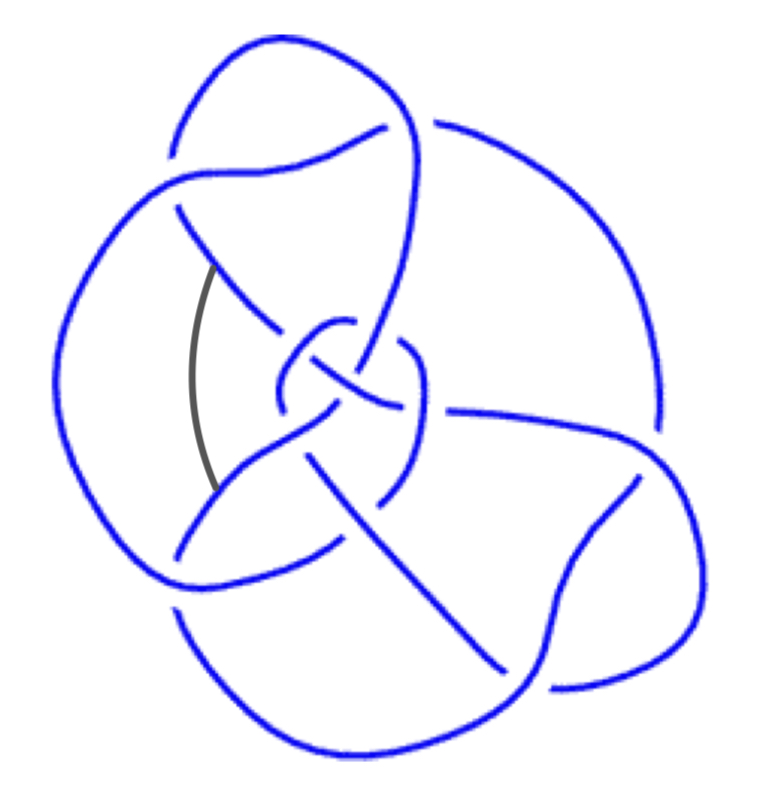}
		\caption{$11n_{177}\stackrel{0}{\longrightarrow} 3_{1}$}
		
	\end{subfigure}
	~
	\begin{subfigure}[b]{0.25\textwidth}
		\includegraphics[width=\textwidth]{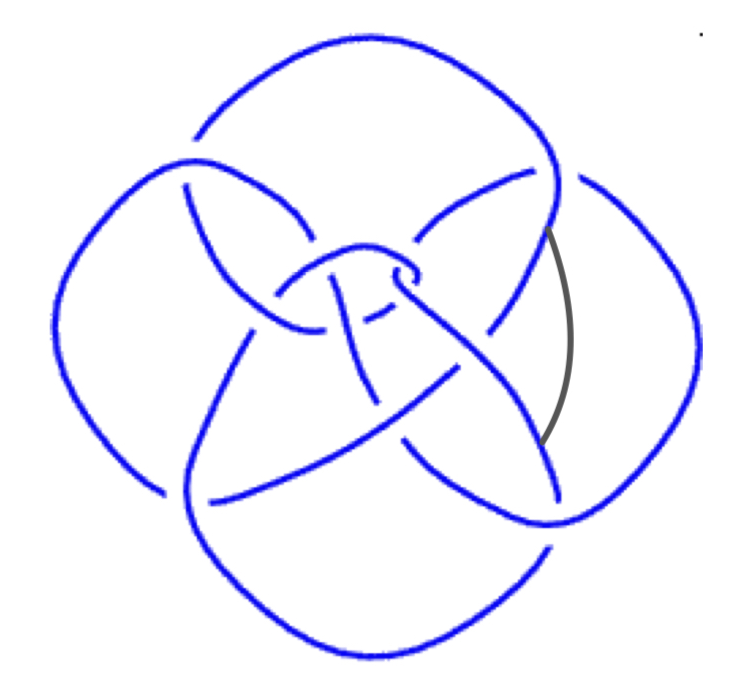}
		\caption{$11n_{178}\stackrel{-1}{\longrightarrow} 9_{32}$}
		
	\end{subfigure}
	\vskip3mm

	\caption{Non-oriented band moves from the knots $11n_{17},  11n_{40},  11n_{159}, $ \\ $ 11n_{166}, 11n_{177}, \text{ and } 11n_{178} $ to knots with non-orientable genus 1.}\label{problems}
\end{figure}

\newpage

\subsection{Concordance}

Knot concordance is a great tool that could be used to solve for the non-orientable 4-genus. For the six remaining knots, their concordance genus is known \cite{knotinfo}, however the knots to which they are concordant is still unknown. Suppose a given knot $K$ is concordant to $K'$, then it is clear that $\gamma_{4}(K) = \gamma_{4} (K')$. 

\begin{ques}
    Is $11n_{40}$ concordant to $10_{57}$?
\end{ques}

$10_{57}$ is a wonderful candidate for concordance to $11n_{40}$, just by a simple analysis of their invariants \cite{knotinfo}. If the answer to Question 4.3 is yes, then the $11n_{40}$ knot has $\gamma_{4} (11n_{40}) = 1$.  

\begin{conj}
    The knots $11n_{17}$, $11n_{159}$, $11n_{166}$, $11n_{177}$, and $11n_{178}$ are not concordant to any knot with 11 or fewer crossings. Moreover, $11n_{17}$, $11n_{159}$, and $11n_{166}$ are not concordant to any knot with 12 or fewer crossings. 
\end{conj}

It should be noted that Kearny has found the concordance genus of 11-crossing knots in \cite{K11}, as well as specific concordances from 11-crossing knots to knots of lower crossings. 


\newpage

\begin{figure}[!htbp] 
	\centering
	\begin{subfigure}[b]{0.26\textwidth}
		\includegraphics[width=\textwidth]{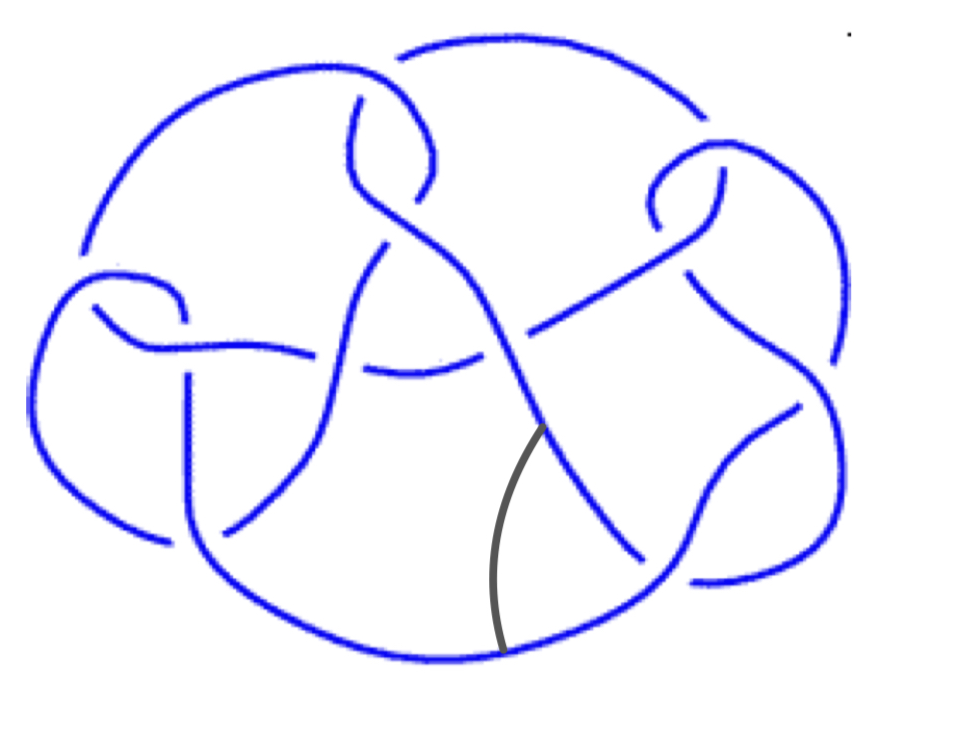}
		\caption{$11n_{1}\stackrel{-1}{\longrightarrow} 0_1$}
		
	\end{subfigure}
	~
	\begin{subfigure}[b]{0.25\textwidth}
		\includegraphics[width=\textwidth]{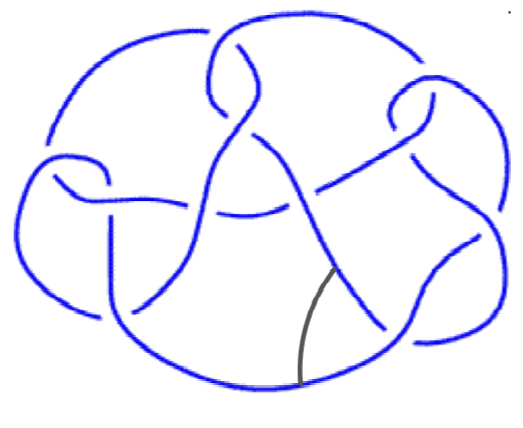}
		\caption{$11n_{3}\stackrel{-1}{\longrightarrow} 10_{137}$}
		
	\end{subfigure}
	~
	\begin{subfigure}[b]{0.25\textwidth}
		\includegraphics[width=\textwidth]{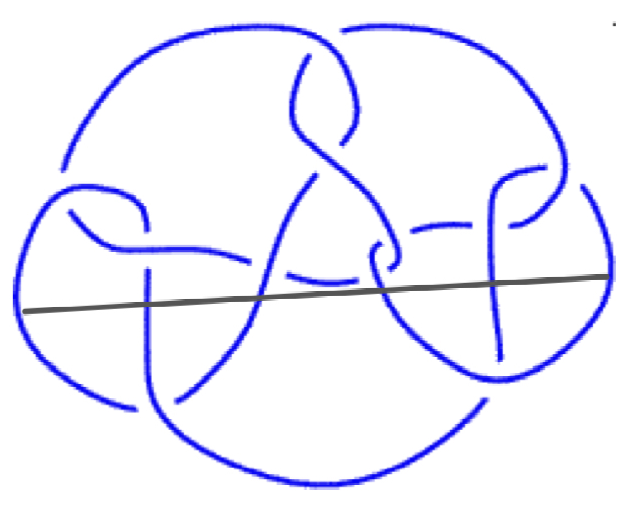}
		\caption{$11n_{5}\stackrel{-1\phantom{i}}{\longrightarrow} 8_{20}$}
		
	\end{subfigure}
	\vskip3mm
	\begin{subfigure}[b]{0.25\textwidth}
		\includegraphics[width=\textwidth]{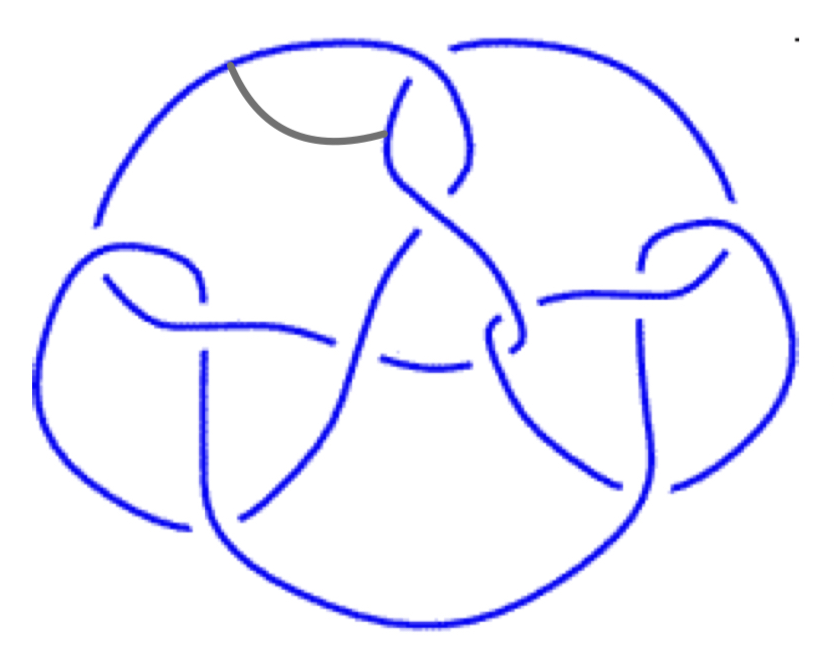}
		\caption{$11n_{6}\stackrel{1}{\longrightarrow} 8_{20}$}
		
	\end{subfigure}
	~
	\begin{subfigure}[b]{0.25\textwidth}
		\includegraphics[width=\textwidth]{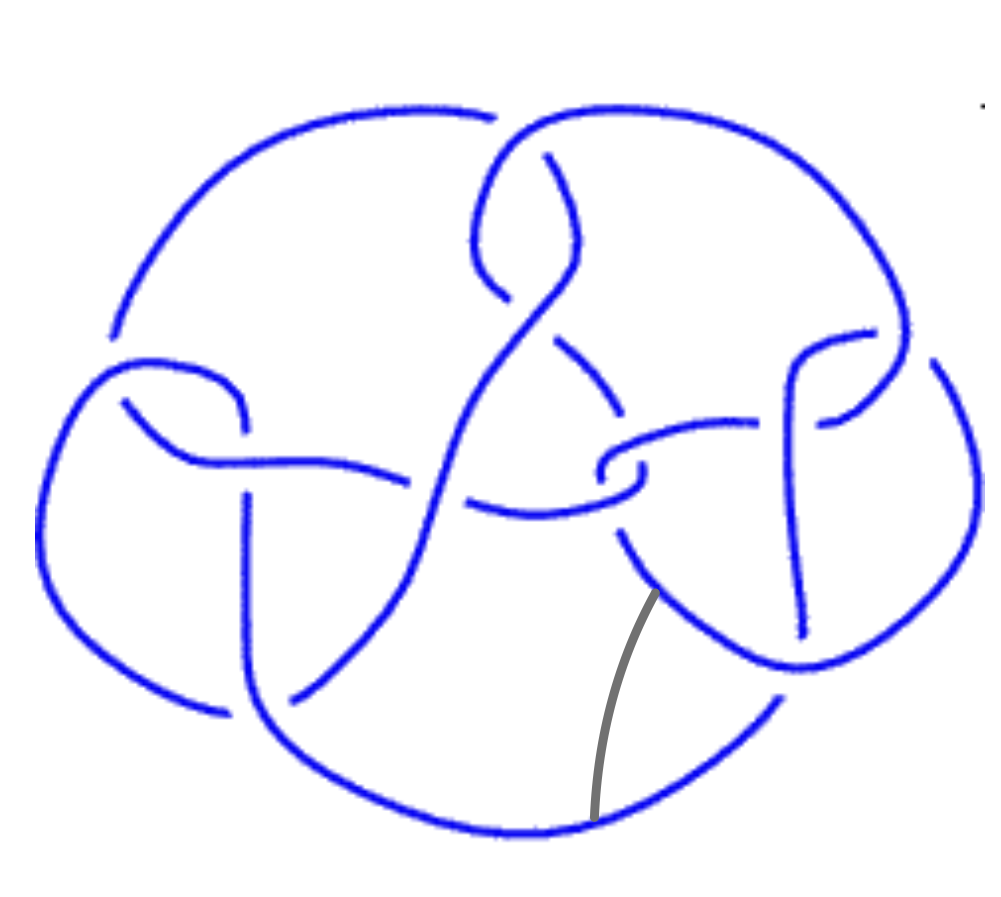}
		\caption{$11n_{7}\stackrel{1}{\longrightarrow} 10_{137}$}
		
	\end{subfigure}
	~
	\begin{subfigure}[b]{0.25\textwidth}
		\includegraphics[width=\textwidth]{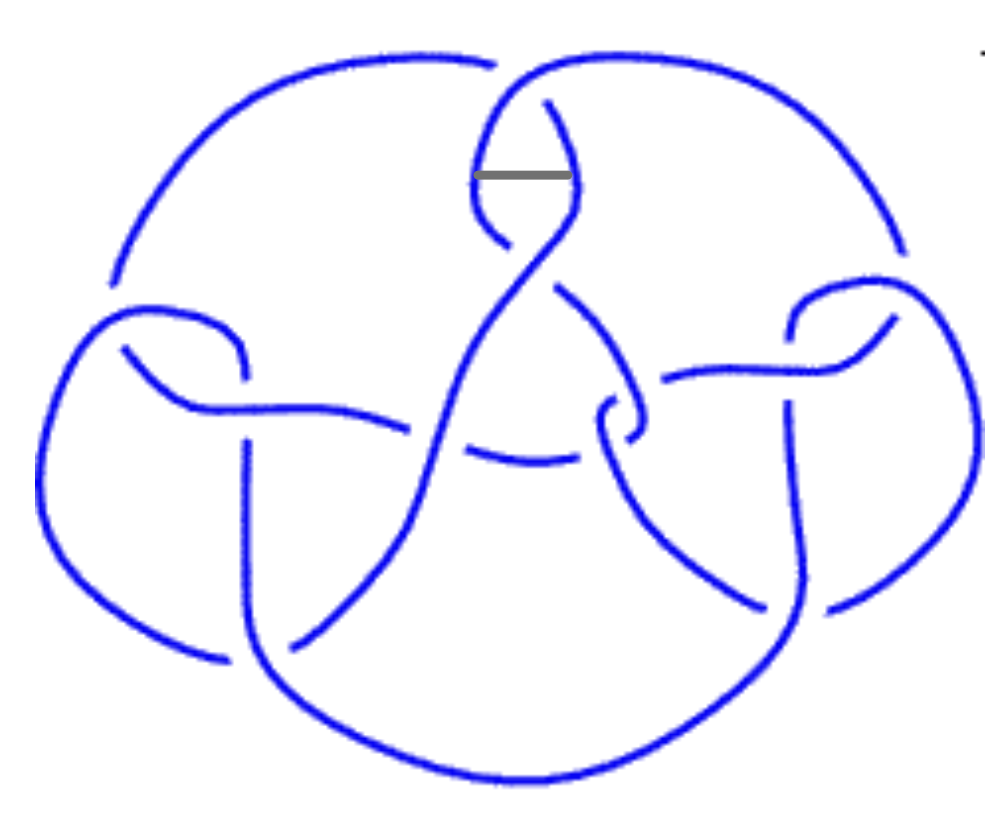}
		\caption{$11n_{8}\stackrel{0}{\longrightarrow} 8_{20}$}
		
	\end{subfigure}
	\vskip3mm
	\begin{subfigure}[b]{0.25\textwidth}
		\includegraphics[width=\textwidth]{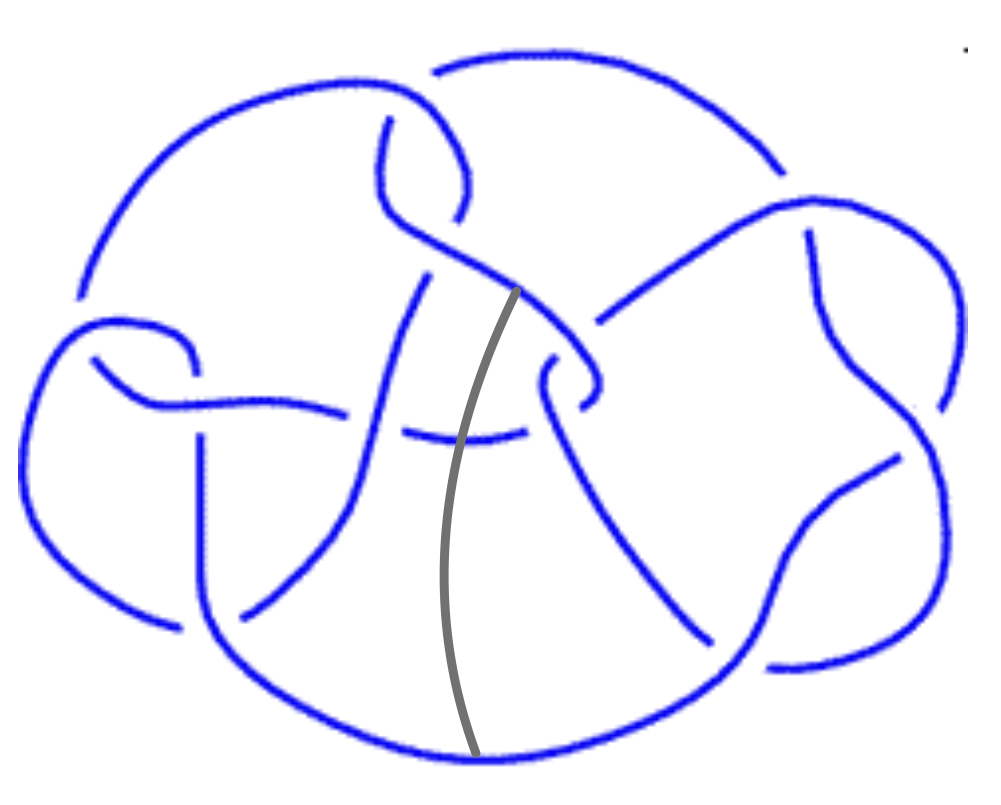}
		\caption{$11n_{9}\stackrel{0}{\longrightarrow} 0_{1}$}
		
	\end{subfigure}
	~
	\begin{subfigure}[b]{0.25\textwidth}
		\includegraphics[width=\textwidth]{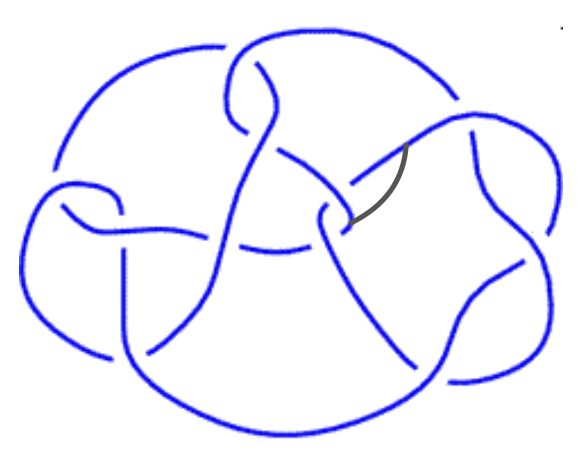}
		\caption{$11n_{11}\stackrel{-1}{\longrightarrow} 12n_{49}$}
		
	\end{subfigure}
	~
	\begin{subfigure}[b]{0.25\textwidth}
		\includegraphics[width=\textwidth]{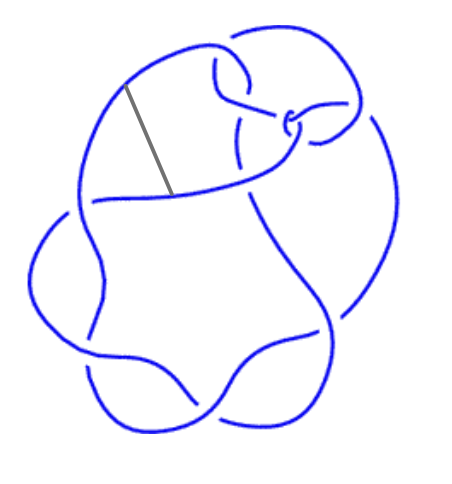}
		\caption{$11n_{13}\stackrel{0}{\longrightarrow} 0_1$}
		
	\end{subfigure}
 \vskip3mm
 ~
	\begin{subfigure}[b]{0.25\textwidth}
		\includegraphics[width=\textwidth]{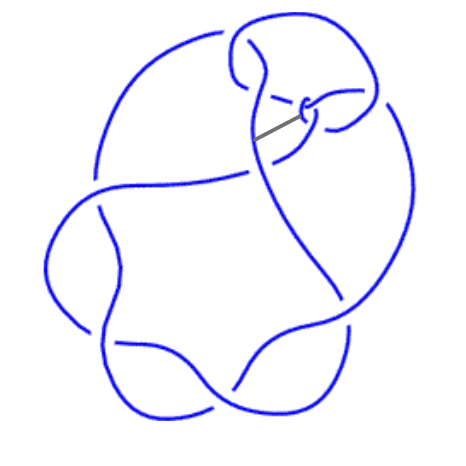}
		\caption{$11n_{14}\stackrel{0}{\longrightarrow} 6_1$}
		
	\end{subfigure}
 ~
	\begin{subfigure}[b]{0.25\textwidth}
		\includegraphics[width=\textwidth]{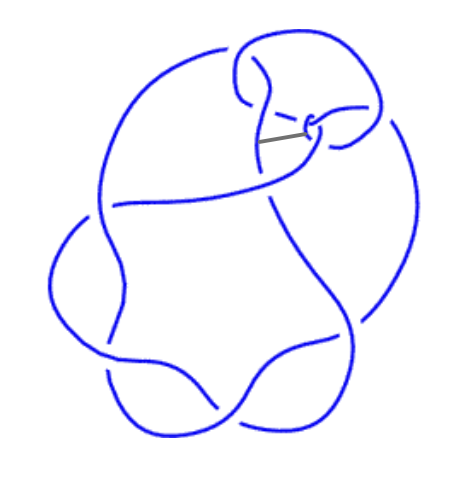}
		\caption{$11n_{15}\stackrel{0}{\longrightarrow} 0_1$}
		
	\end{subfigure}
 ~
	\begin{subfigure}[b]{0.25\textwidth}
		\includegraphics[width=\textwidth]{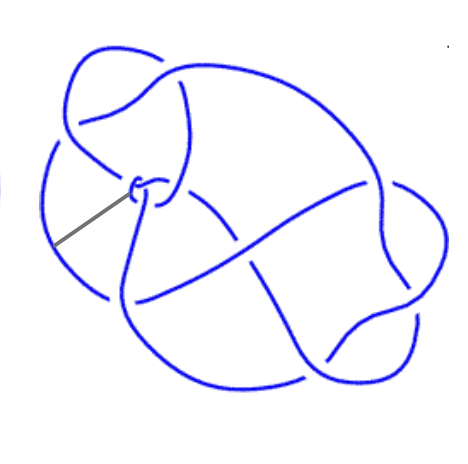}
		\caption{$11n_{16}\stackrel{0}{\longrightarrow} 0_1$}
		
	\end{subfigure}
	\vskip3mm
	\caption{Non-oriented band moves from the knots $11n_{1},  11n_{3},  11n_{5},   $ \\ $11n_{6}, 11n_{7}, 11n_{8}, 11n_{9}, 11n_{11}, 11n_{13}, 11n_{14}, 11n_{15}, \text{ and } 11n_{16} $ to smoothly slice knots.}\label{firstSlice}
\end{figure}
%
\newpage

\begin{figure}[!htbp]
	\centering
	\begin{subfigure}[b]{0.26\textwidth}
		\includegraphics[width=\textwidth]{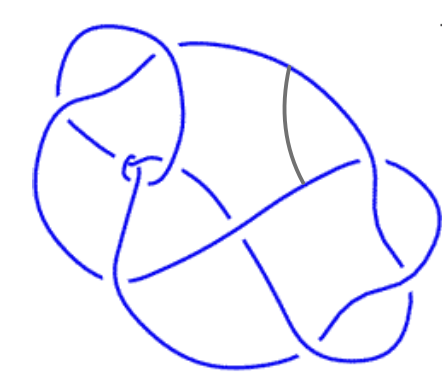}
		\caption{$11n_{18}\stackrel{1}{\longrightarrow} 10_{137}$}
		
	\end{subfigure}
	~
	\begin{subfigure}[b]{0.22\textwidth}
		\includegraphics[width=\textwidth]{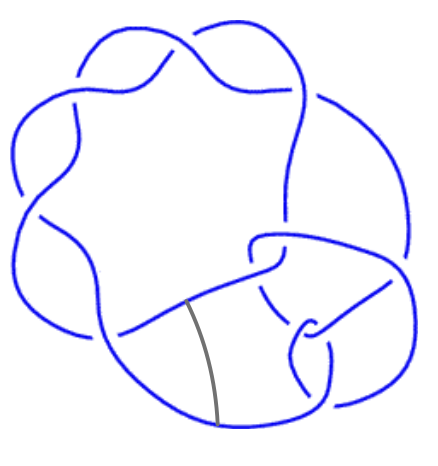}
		\caption{$11n_{19}\stackrel{0}{\longrightarrow} 0_{1}$}
		
	\end{subfigure}
	~
	\begin{subfigure}[b]{0.25\textwidth}
		\includegraphics[width=\textwidth]{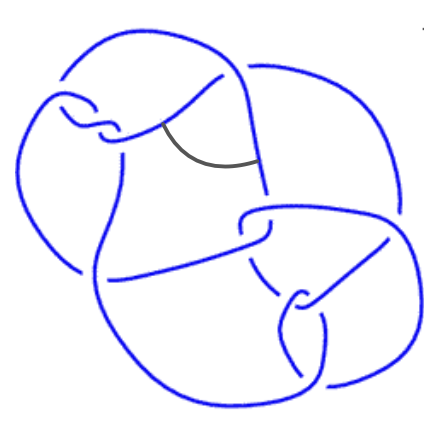}
		\caption{$11n_{20}\stackrel{1\phantom{i}}{\longrightarrow} 12n_{24}$}
		
	\end{subfigure}
	\vskip3mm
	\begin{subfigure}[b]{0.25\textwidth}
		\includegraphics[width=\textwidth]{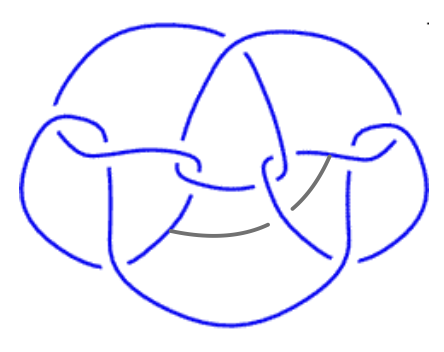}
		\caption{$11n_{23}\stackrel{0}{\longrightarrow} 6_{1}$}
		
	\end{subfigure}
	~
	\begin{subfigure}[b]{0.25\textwidth}
		\includegraphics[width=\textwidth]{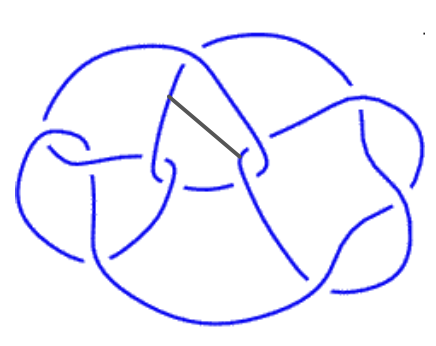}
		\caption{$11n_{24}\stackrel{0}{\longrightarrow} 0_{1}$}
		
	\end{subfigure}
	~
	\begin{subfigure}[b]{0.25\textwidth}
		\includegraphics[width=\textwidth]{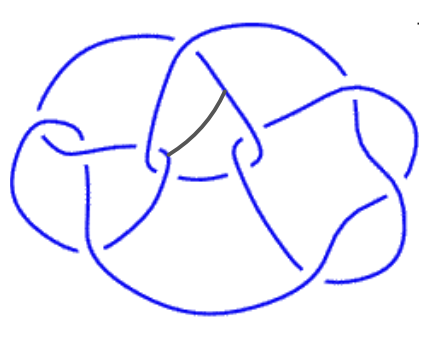}
		\caption{$11n_{25}\stackrel{-1}{\longrightarrow} 12n_{24}$}
		
	\end{subfigure}
	\vskip3mm
	\begin{subfigure}[b]{0.25\textwidth}
		\includegraphics[width=\textwidth]{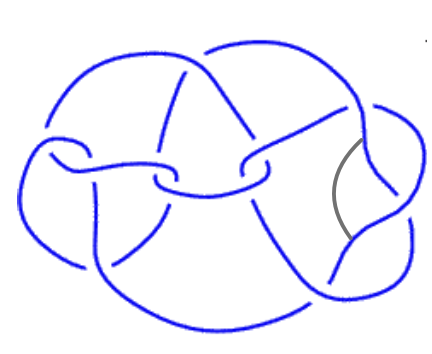}
		\caption{$11n_{26}\stackrel{-1}{\longrightarrow} 8_{20}$}
		
	\end{subfigure}
	~
	\begin{subfigure}[b]{0.25\textwidth}
		\includegraphics[width=\textwidth]{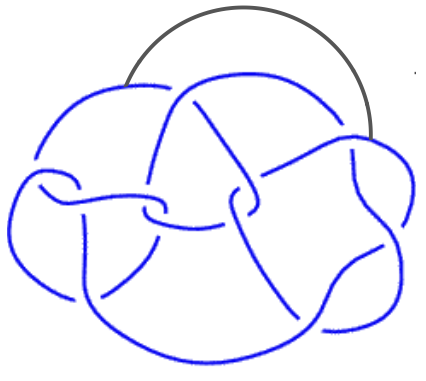}
		\caption{$11n_{27}\stackrel{-1}{\longrightarrow} 8_{8}$}
		
	\end{subfigure}
	~
	\begin{subfigure}[b]{0.22\textwidth}
		\includegraphics[width=\textwidth]{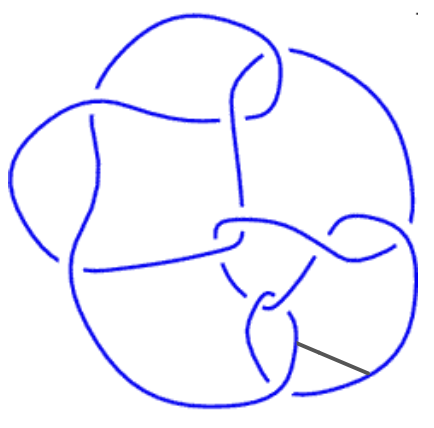}
		\caption{$11n_{31}\stackrel{0}{\longrightarrow} 10_{137}$}
		
	\end{subfigure}
 \vskip3mm
 ~
	\begin{subfigure}[b]{0.26\textwidth}
		\includegraphics[width=\textwidth]{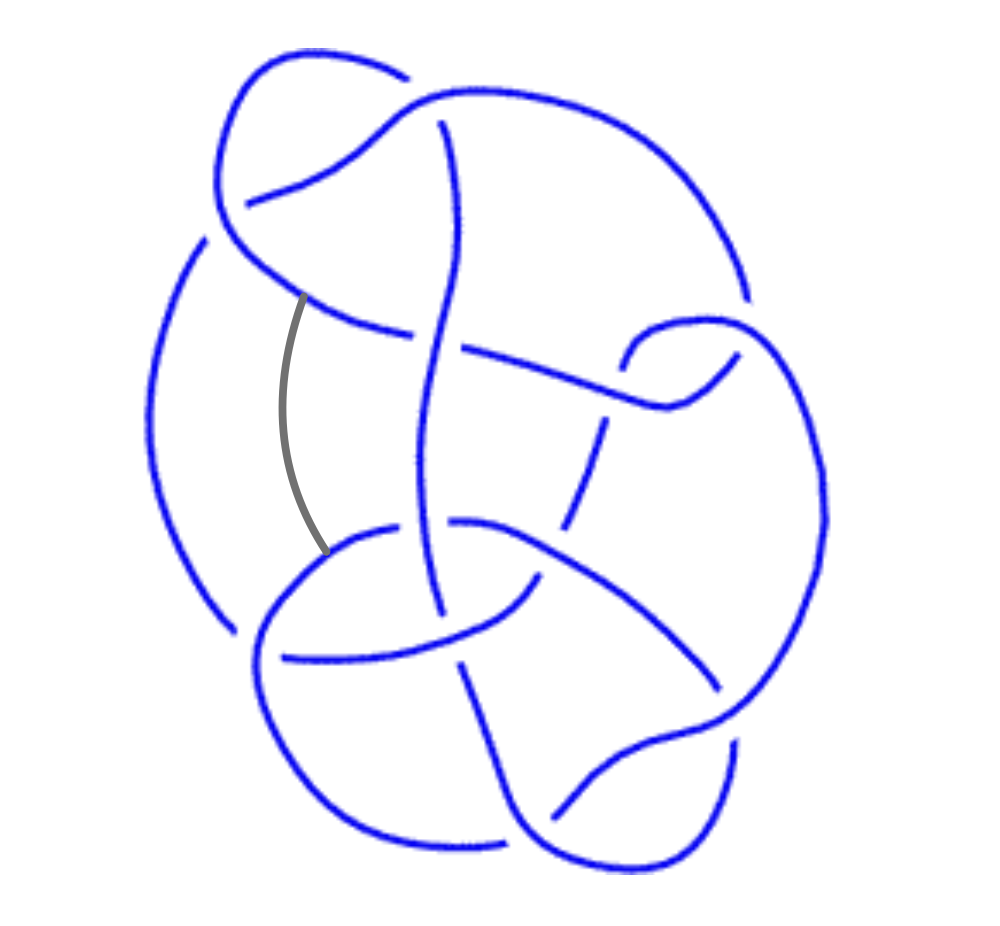}
		\caption{$11n_{34}\stackrel{0}{\longrightarrow} 0_1$}
		
	\end{subfigure}
 ~
	\begin{subfigure}[b]{0.25\textwidth}
		\includegraphics[width=\textwidth]{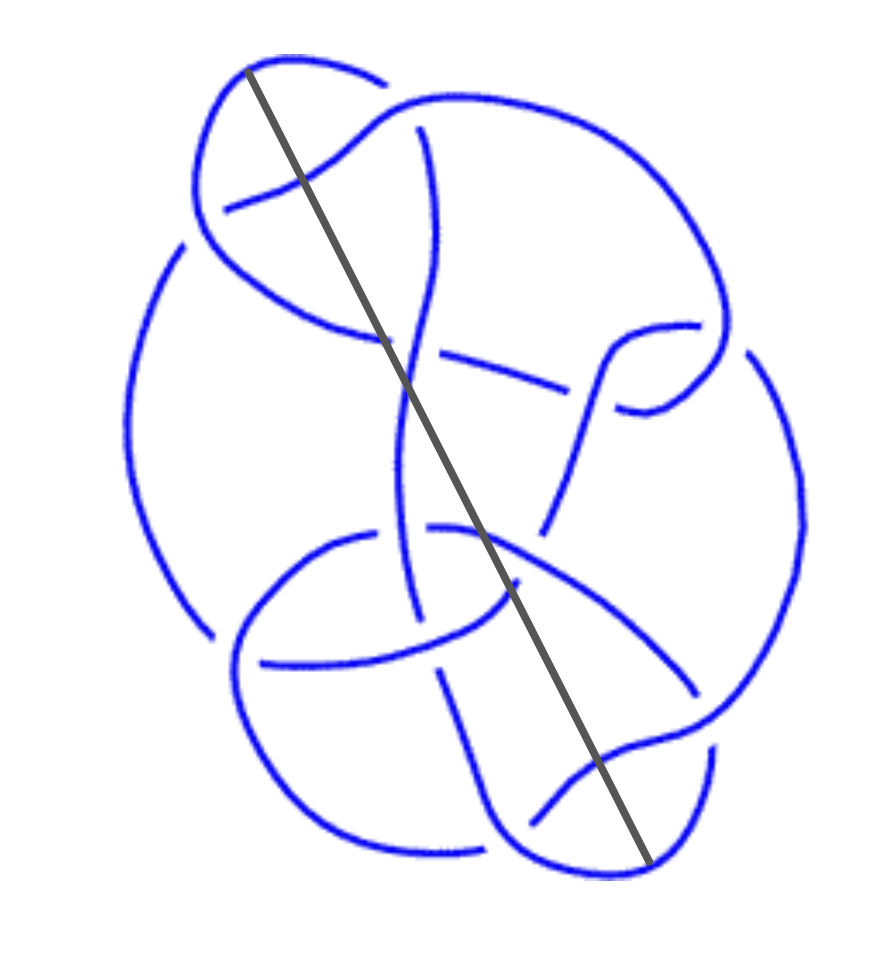}
		\caption{$11n_{36}\stackrel{0}{\longrightarrow} 10_{129}$}
		
	\end{subfigure}
 ~
	\begin{subfigure}[b]{0.25\textwidth}
		\includegraphics[width=\textwidth]{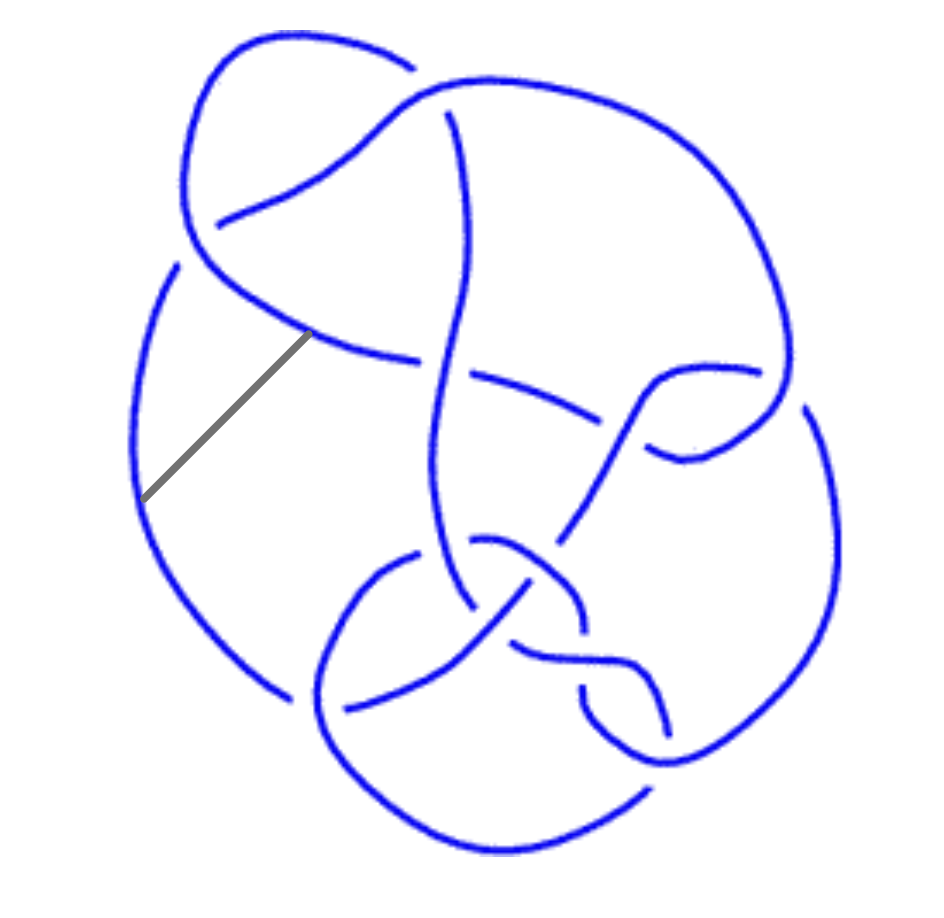}
		\caption{$11n_{41}\stackrel{0}{\longrightarrow} 8_{20}$}
		
	\end{subfigure}
	\vskip3mm
	\caption{Non-oriented band moves from the knots $11n_{18},  11n_{19},  11n_{20},    $ \\ $11n_{23}, 11n_{24},  11n_{25}, 11n_{26}, 11n_{27}, 11n_{31}, 11n_{34}, 11n_{36}, \text{ and } 11n_{41} $ to smoothly slice knots.}
\end{figure}
%

\newpage

\begin{figure}[!htbp]
	\centering
	\begin{subfigure}[b]{0.26\textwidth}
		\includegraphics[width=\textwidth]{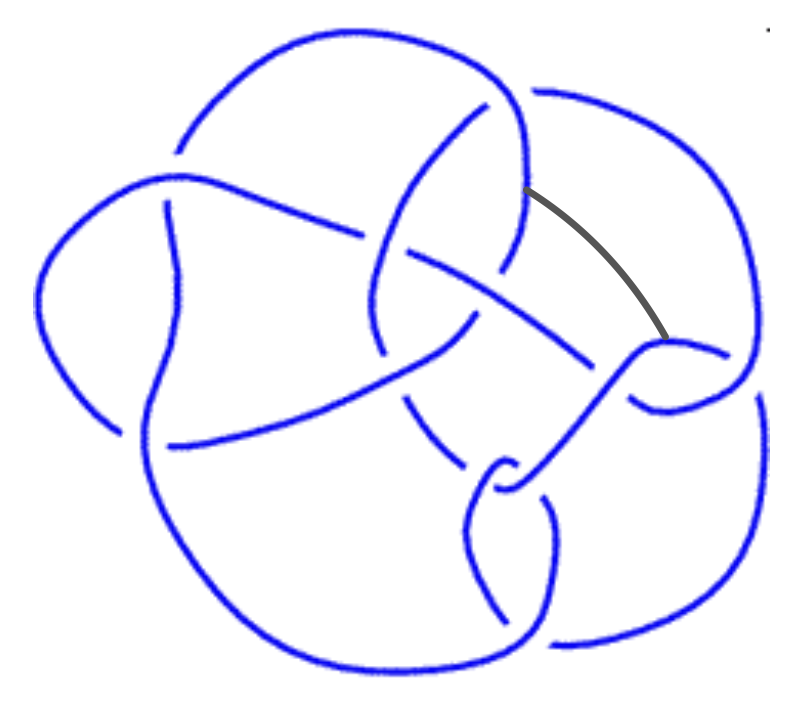}
		\caption{$11n_{44}\stackrel{1}{\longrightarrow} 6_{1}$}
		
	\end{subfigure}
	~
	\begin{subfigure}[b]{0.26\textwidth}
		\includegraphics[width=\textwidth]{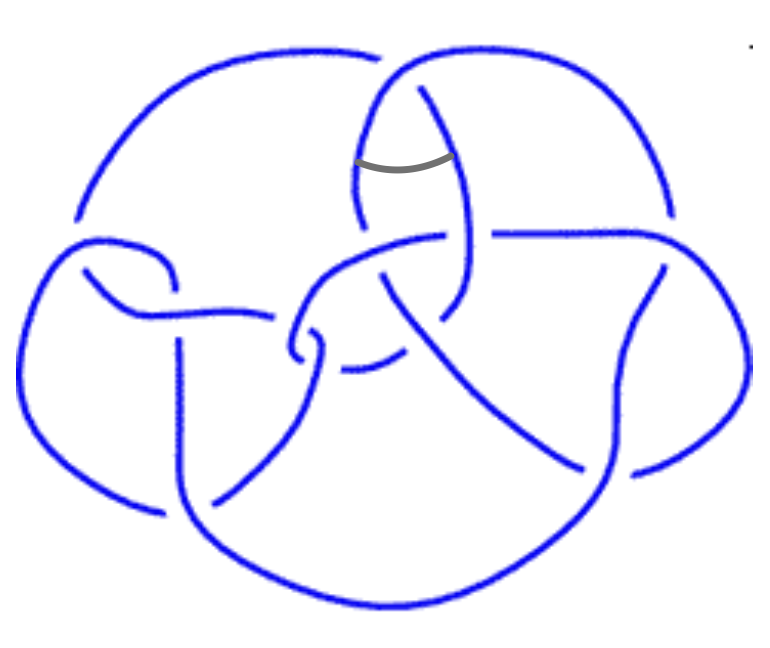}
		\caption{$11n_{45}\stackrel{1}{\longrightarrow} 10_{129}$}
		
	\end{subfigure}
	~
	\begin{subfigure}[b]{0.26\textwidth}
		\includegraphics[width=\textwidth]{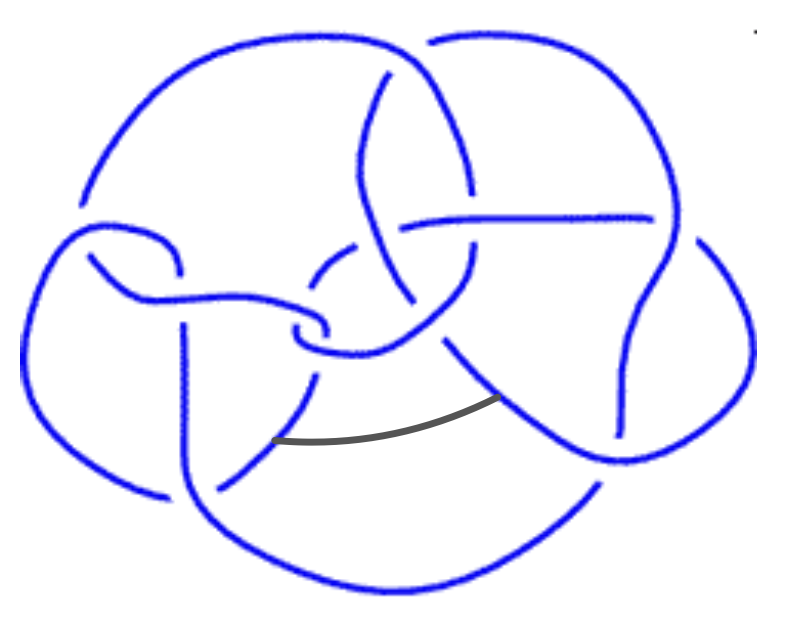}
		\caption{$11n_{46}\stackrel{0\phantom{i}}{\longrightarrow} 6_{1}$}
		
	\end{subfigure}
	\vskip3mm
	\begin{subfigure}[b]{0.26\textwidth}
		\includegraphics[width=\textwidth]{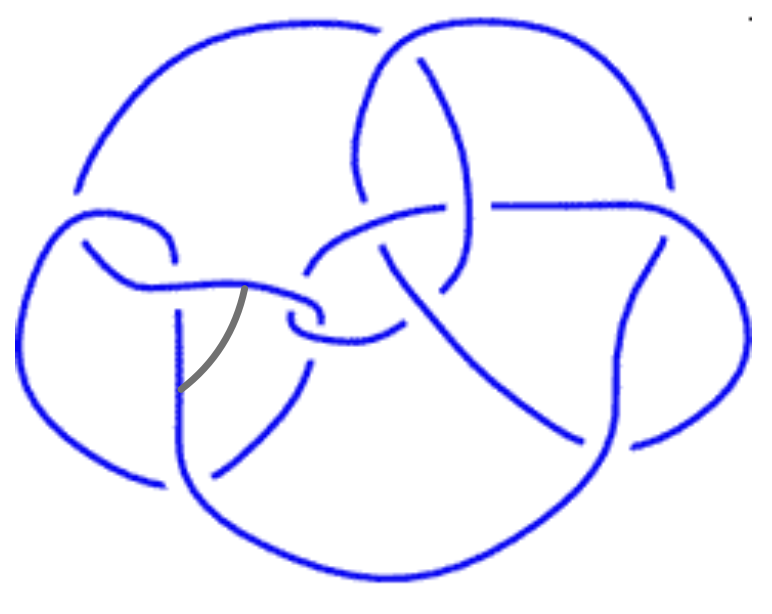}
		\caption{$11n_{47}\stackrel{0}{\longrightarrow} 8_{20}$}
		
	\end{subfigure}
	~
	\begin{subfigure}[b]{0.26\textwidth}
		\includegraphics[width=\textwidth]{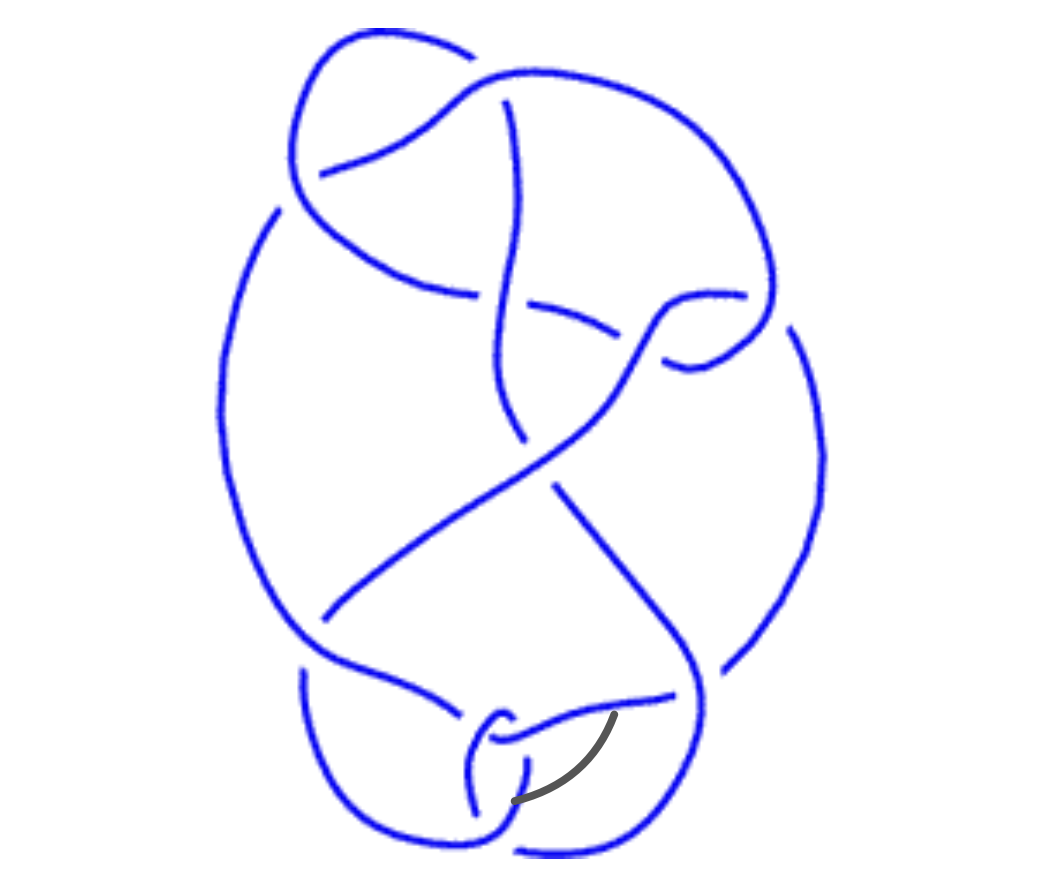}
		\caption{$11n_{52}\stackrel{1}{\longrightarrow} 12n_{170}$}
		
	\end{subfigure}
	~
	\begin{subfigure}[b]{0.26\textwidth}
		\includegraphics[width=\textwidth]{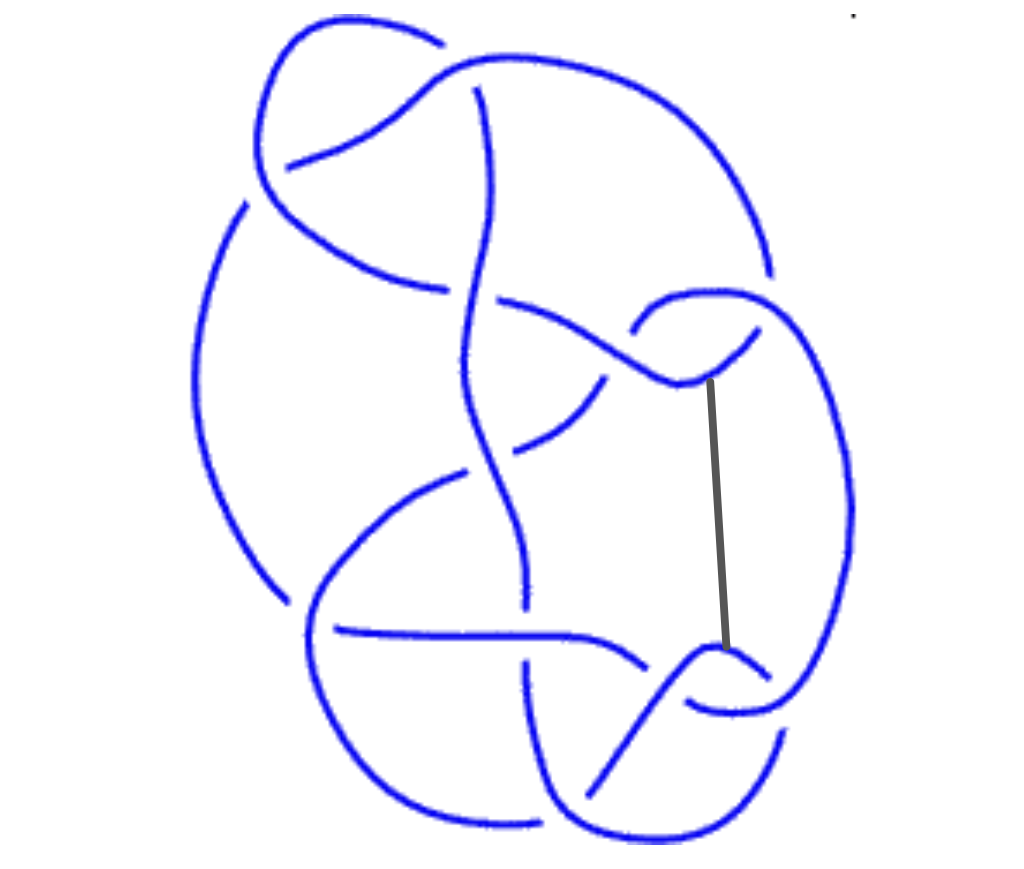}
		\caption{$11n_{54}\stackrel{0}{\longrightarrow} 6_{1}$}
		
	\end{subfigure}
	\vskip3mm
	\begin{subfigure}[b]{0.29\textwidth}
		\includegraphics[width=\textwidth]{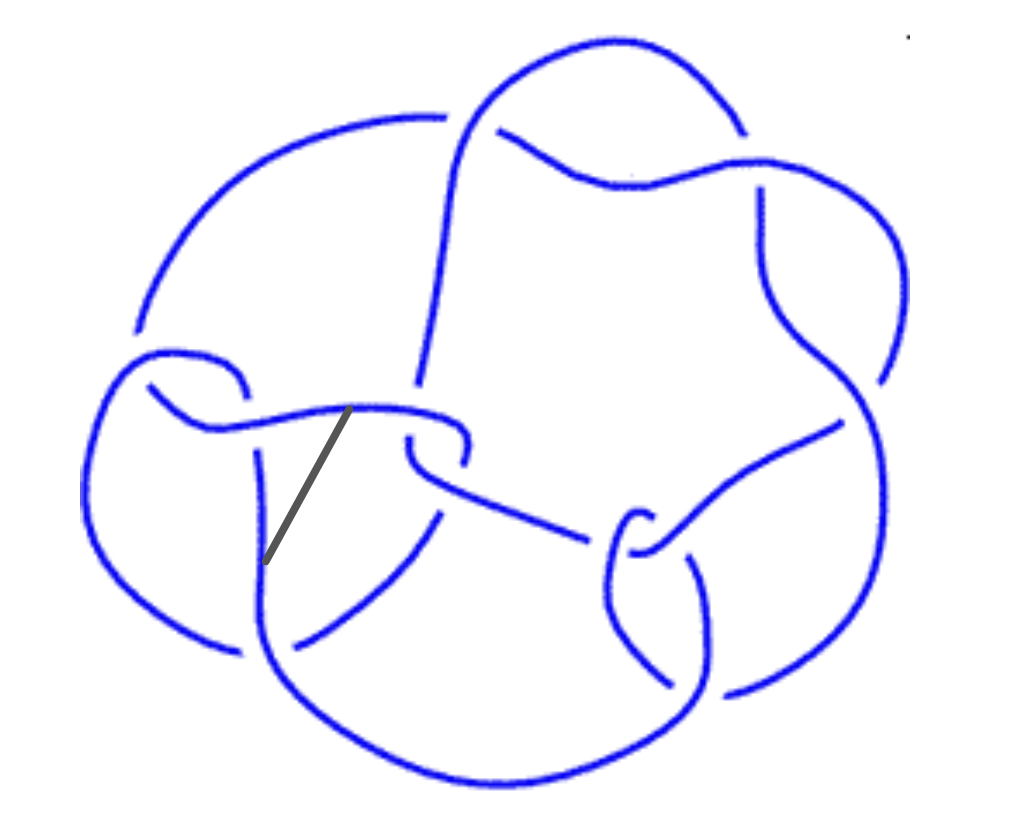}
		\caption{$11n_{57}\stackrel{0}{\longrightarrow} 0_{1}$}
		
	\end{subfigure}
	~
	\begin{subfigure}[b]{0.26\textwidth}
		\includegraphics[width=\textwidth]{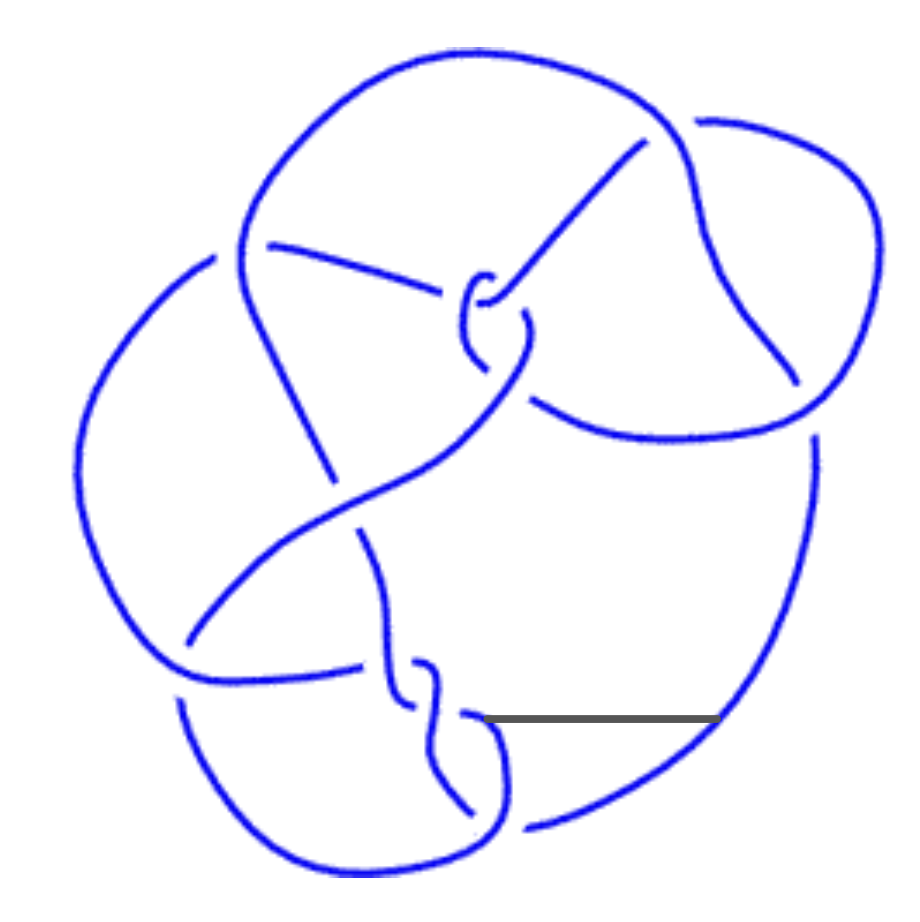}
		\caption{$11n_{58}\stackrel{-1}{\longrightarrow} 8_{20}$}
		
	\end{subfigure}
	~
	\begin{subfigure}[b]{0.26\textwidth}
		\includegraphics[width=\textwidth]{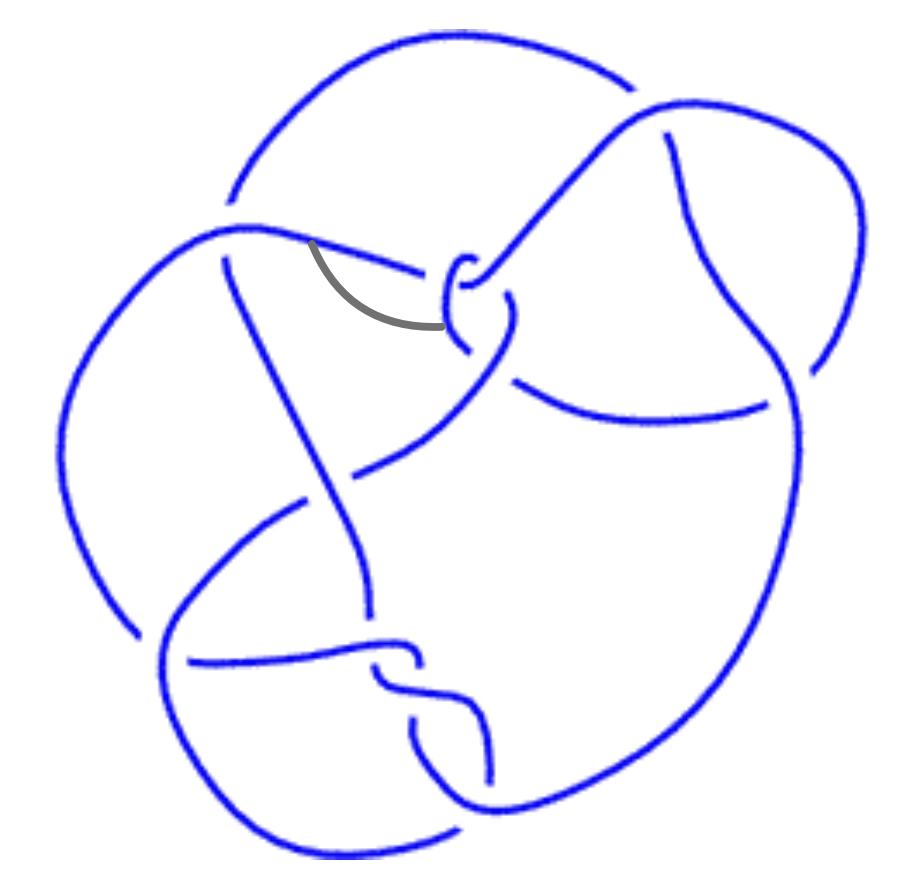}
		\caption{$11n_{59}\stackrel{0}{\longrightarrow} 8_{9}$}
		
	\end{subfigure}
 \vskip3mm
	\begin{subfigure}[b]{0.25\textwidth}
		\includegraphics[width=\textwidth]{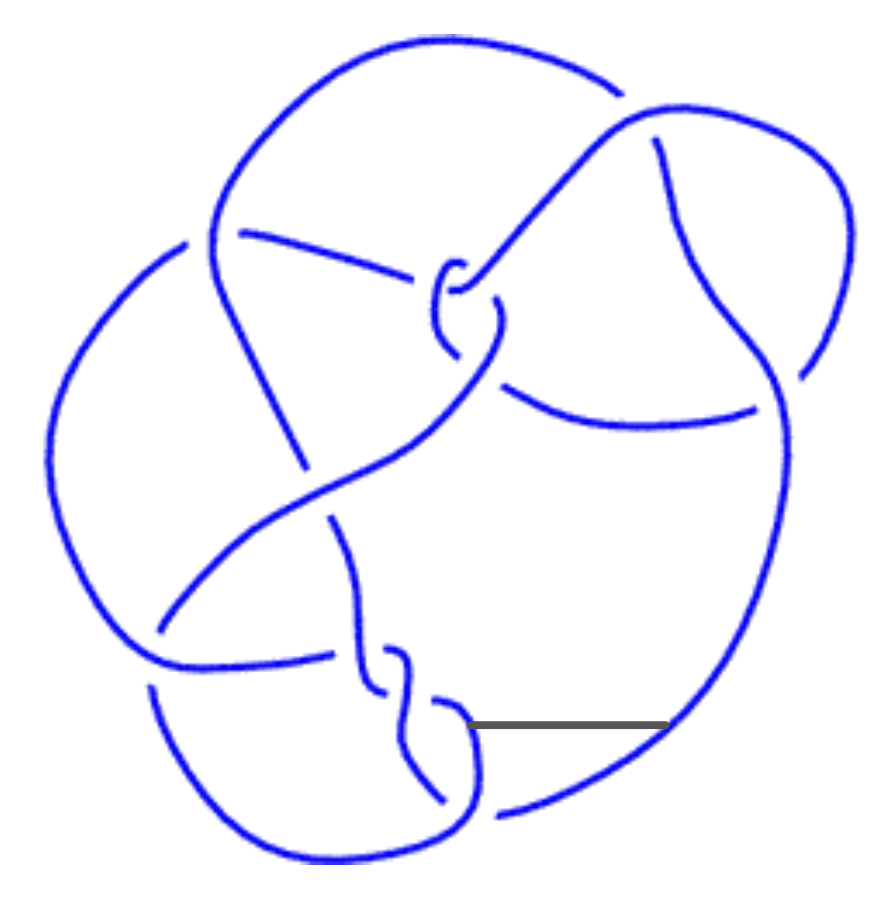}
		\caption{$11n_{60}\stackrel{-1}{\longrightarrow} 8_{20}$}
		
	\end{subfigure}
 ~
	\begin{subfigure}[b]{0.25\textwidth}
		\includegraphics[width=\textwidth]{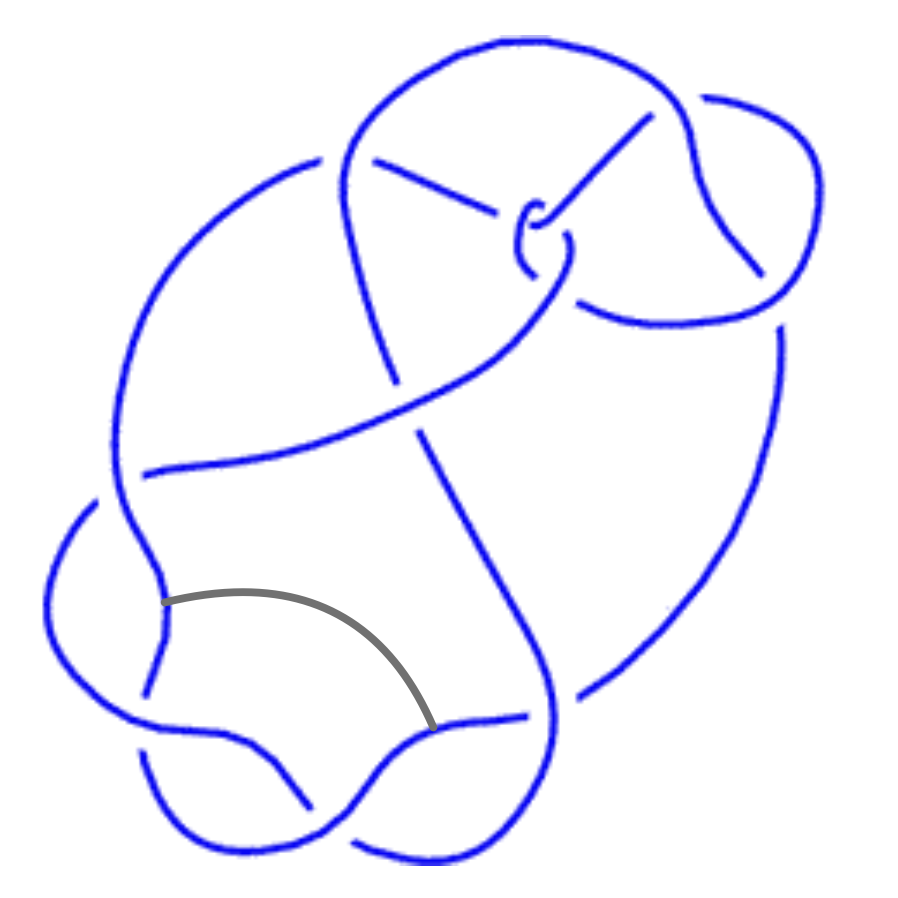}
		\caption{$11n_{62}\stackrel{1}{\longrightarrow} 0_{1}$}
		
	\end{subfigure}
 ~
	\begin{subfigure}[b]{0.25\textwidth}
		\includegraphics[width=\textwidth]{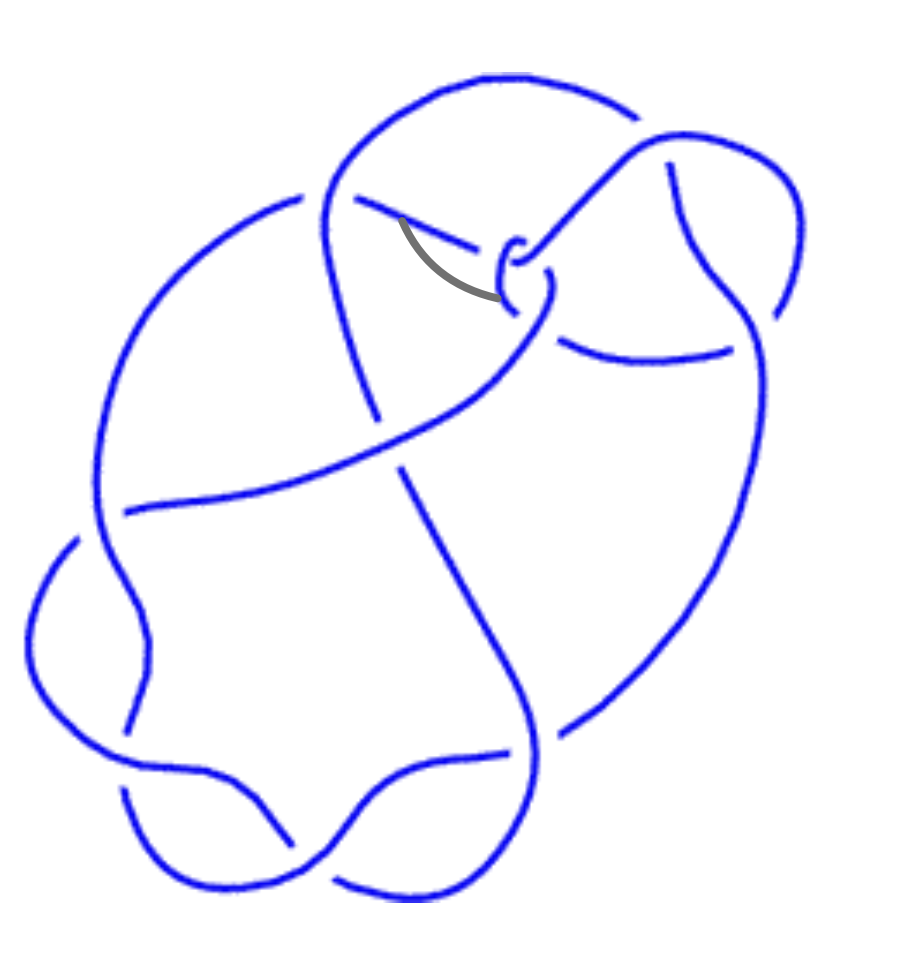}
		\caption{$11n_{64}\stackrel{0}{\longrightarrow} 0_{1}$}
		
	\end{subfigure}
	\vskip3mm
	\caption{Non-oriented band moves from the knots $11n_{44},  11n_{45},  11n_{46},    $ \\ $11n_{47},  11n_{52}, 11n_{54}, 11n_{57}, 11n_{58}, 11n_{59}, 11n_{60}, 11n_{62}, \text{ and } 11n_{64} $ to smoothly slice knots.}
\end{figure}
%


\newpage

\begin{figure}[!htbp]
	\centering
	\begin{subfigure}[b]{0.25\textwidth}
		\includegraphics[width=\textwidth]{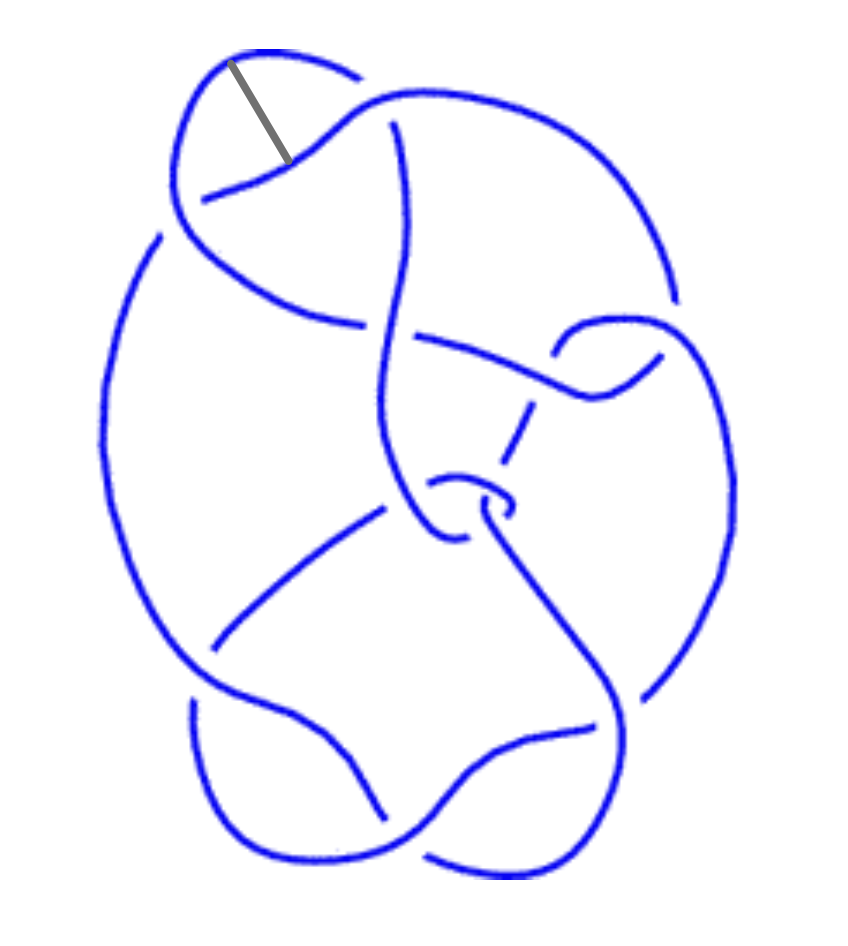}
		\caption{$11n_{65}\stackrel{0}{\longrightarrow} 9_{46}$}
		
	\end{subfigure}
	~
	\begin{subfigure}[b]{0.26\textwidth}
		\includegraphics[width=\textwidth]{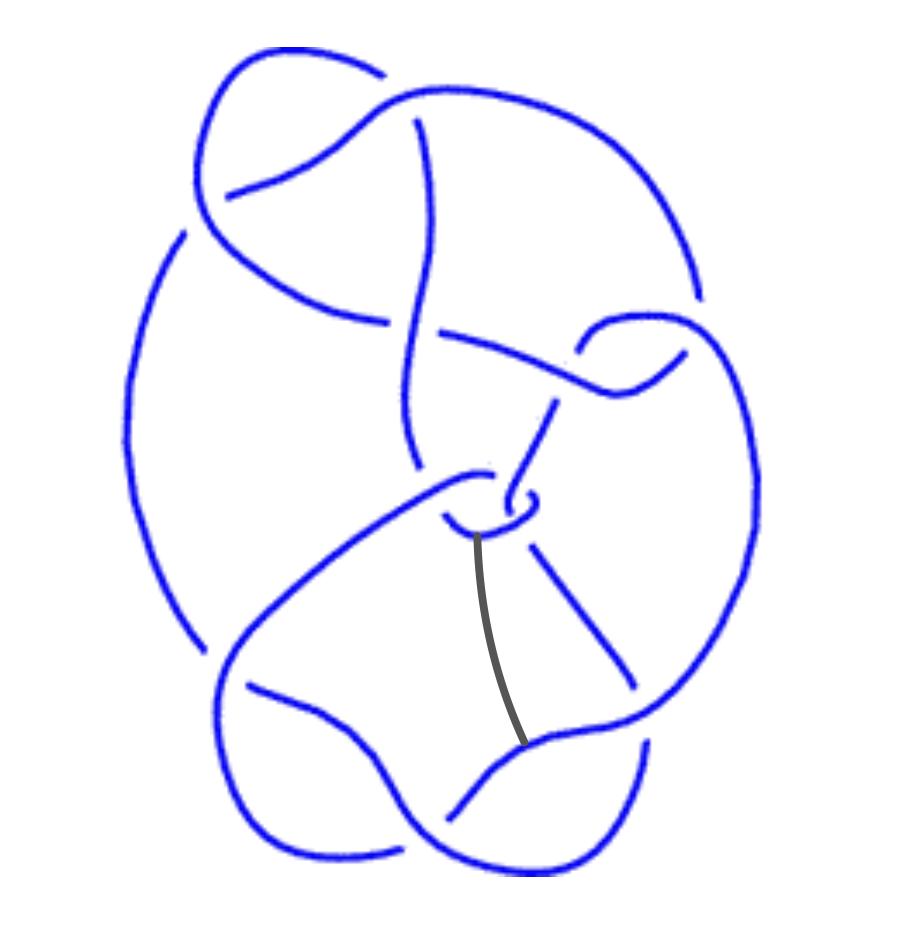}
		\caption{$11n_{66}\stackrel{0\phantom{i}}{\longrightarrow} 8_{8}$}
		
	\end{subfigure}
	~
	\begin{subfigure}[b]{0.26\textwidth}
		\includegraphics[width=\textwidth]{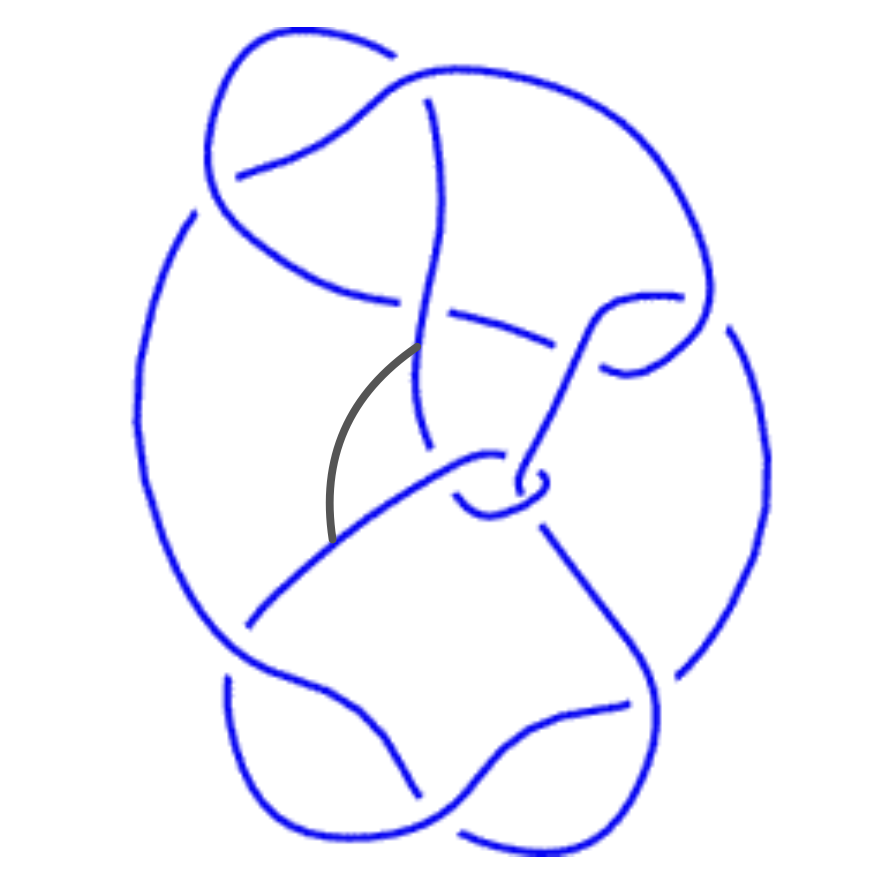}
		\caption{$11n_{68}\stackrel{-1}{\longrightarrow} 10_{129}$}
		
	\end{subfigure}
	
 \vskip3mm
	\begin{subfigure}[b]{0.25\textwidth}
		\includegraphics[width=\textwidth]{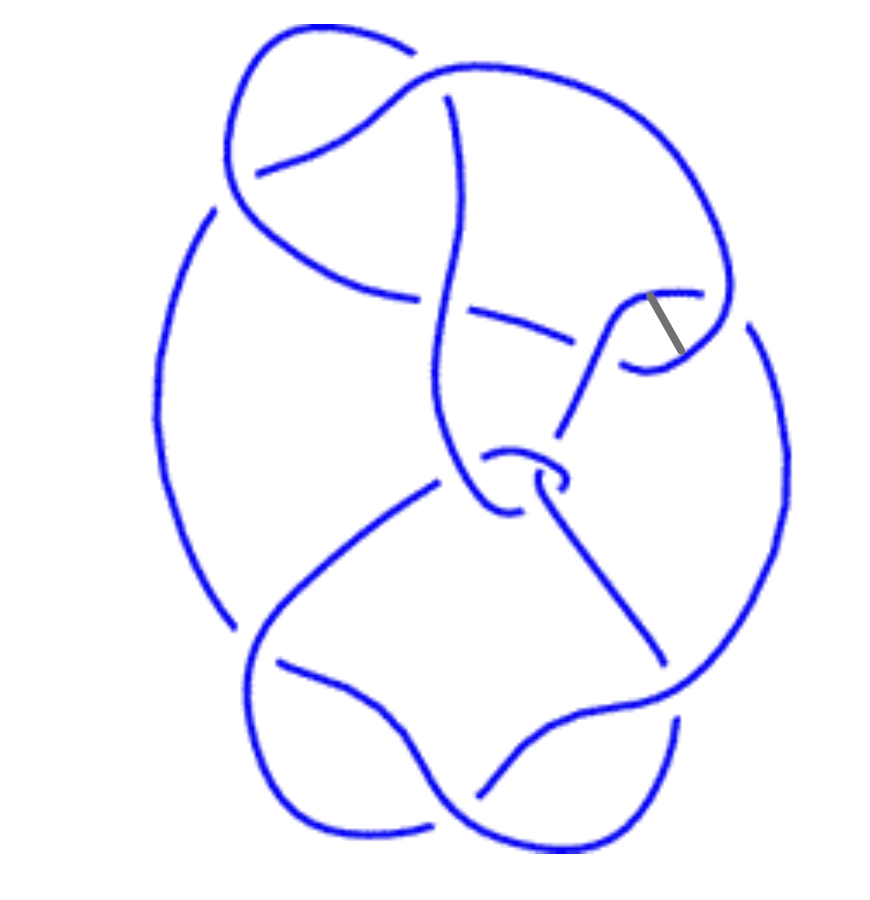}
		\caption{$11n_{69}\stackrel{0}{\longrightarrow} 8_{20}$}
		
	\end{subfigure}
	~
	\begin{subfigure}[b]{0.25\textwidth}
		\includegraphics[width=\textwidth]{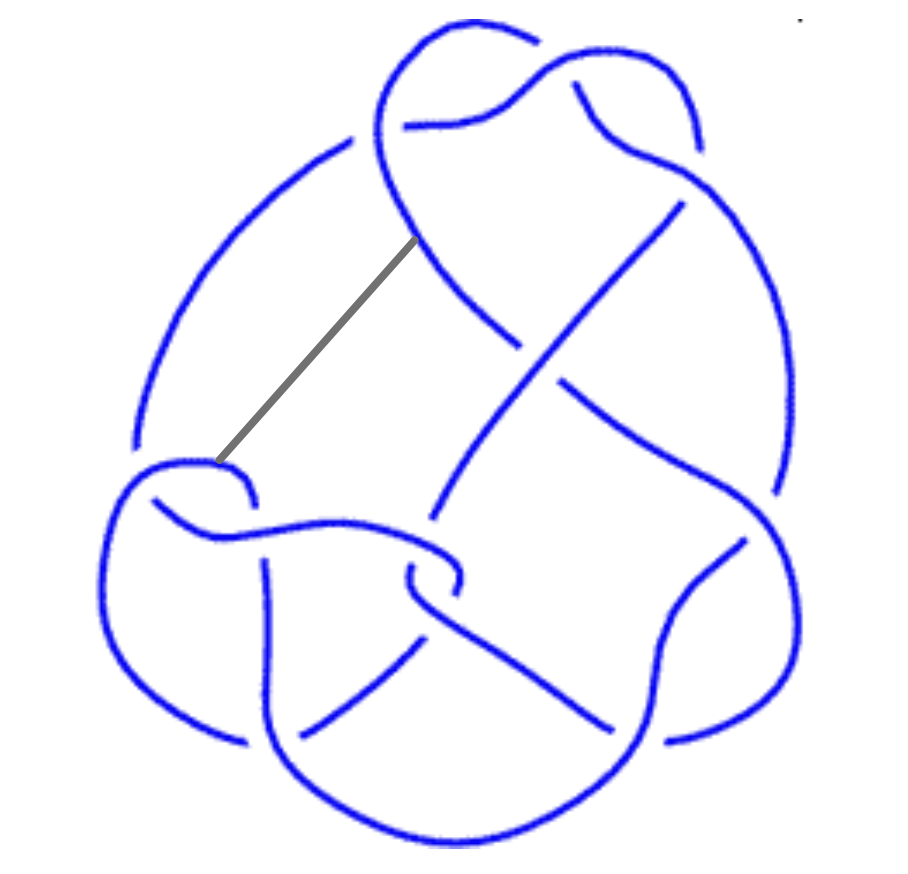}
		\caption{$11n_{70}\stackrel{0}{\longrightarrow} 8_{20}$}
		
	\end{subfigure}
	~
	\begin{subfigure}[b]{0.25\textwidth}
		\includegraphics[width=\textwidth]{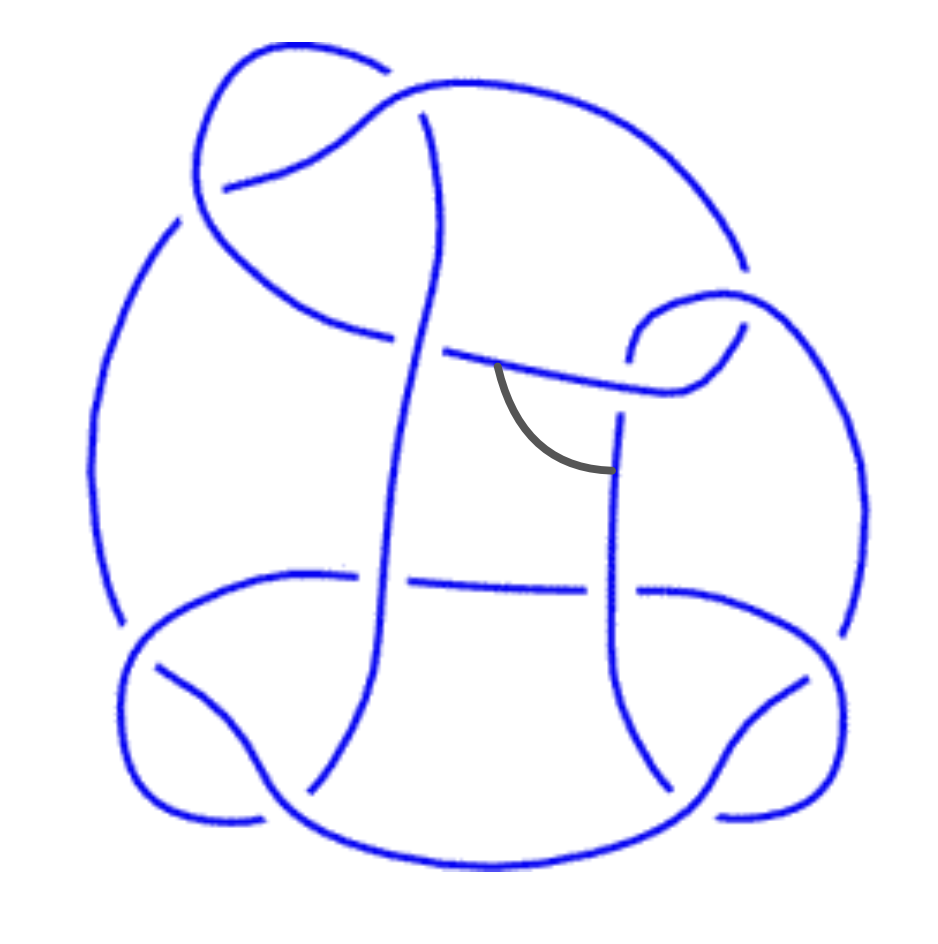}
		\caption{$11n_{71}\stackrel{-1}{\longrightarrow} 12n_{556}$}
		
	\end{subfigure}
	
 \vskip3mm
	\begin{subfigure}[b]{0.25\textwidth}
		\includegraphics[width=\textwidth]{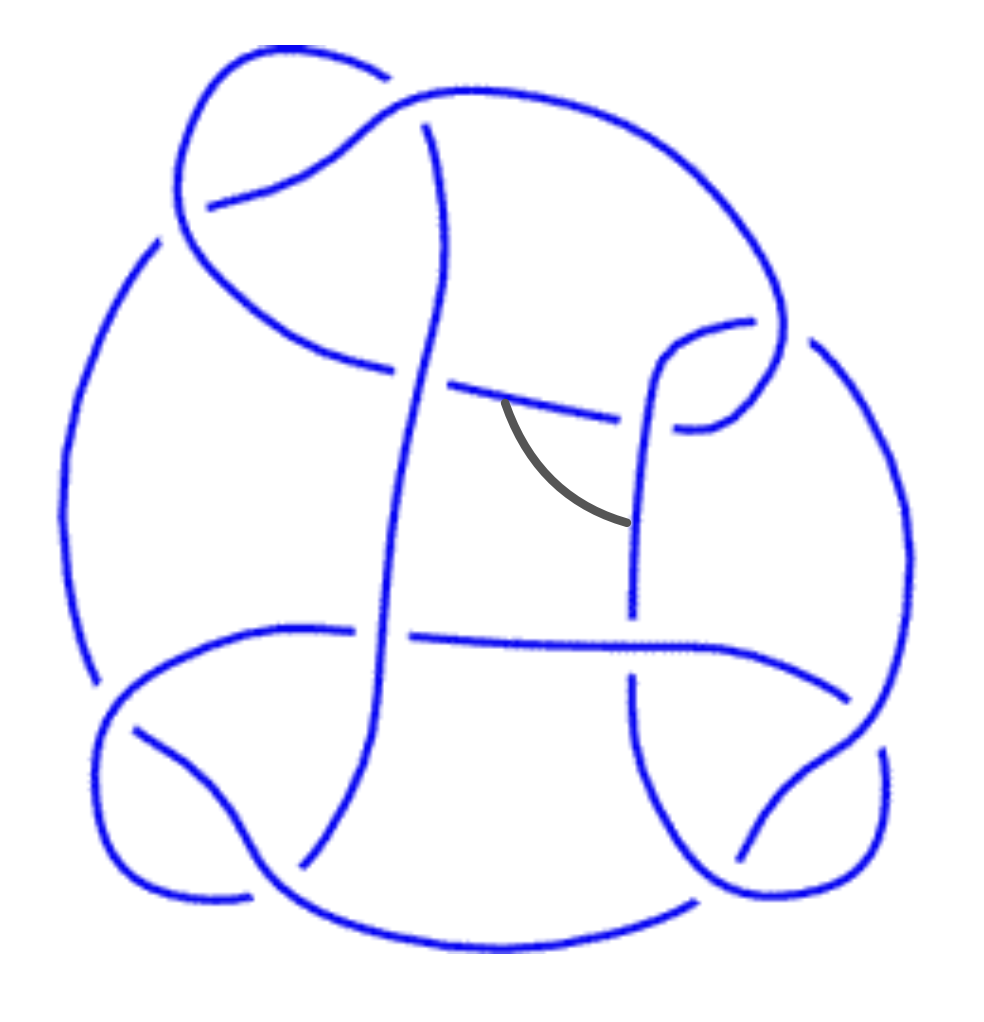}
		\caption{$11n_{75}\stackrel{1}{\longrightarrow} 12n_{553}$}
		
	\end{subfigure}
	~
	\begin{subfigure}[b]{0.25\textwidth}
		\includegraphics[width=\textwidth]{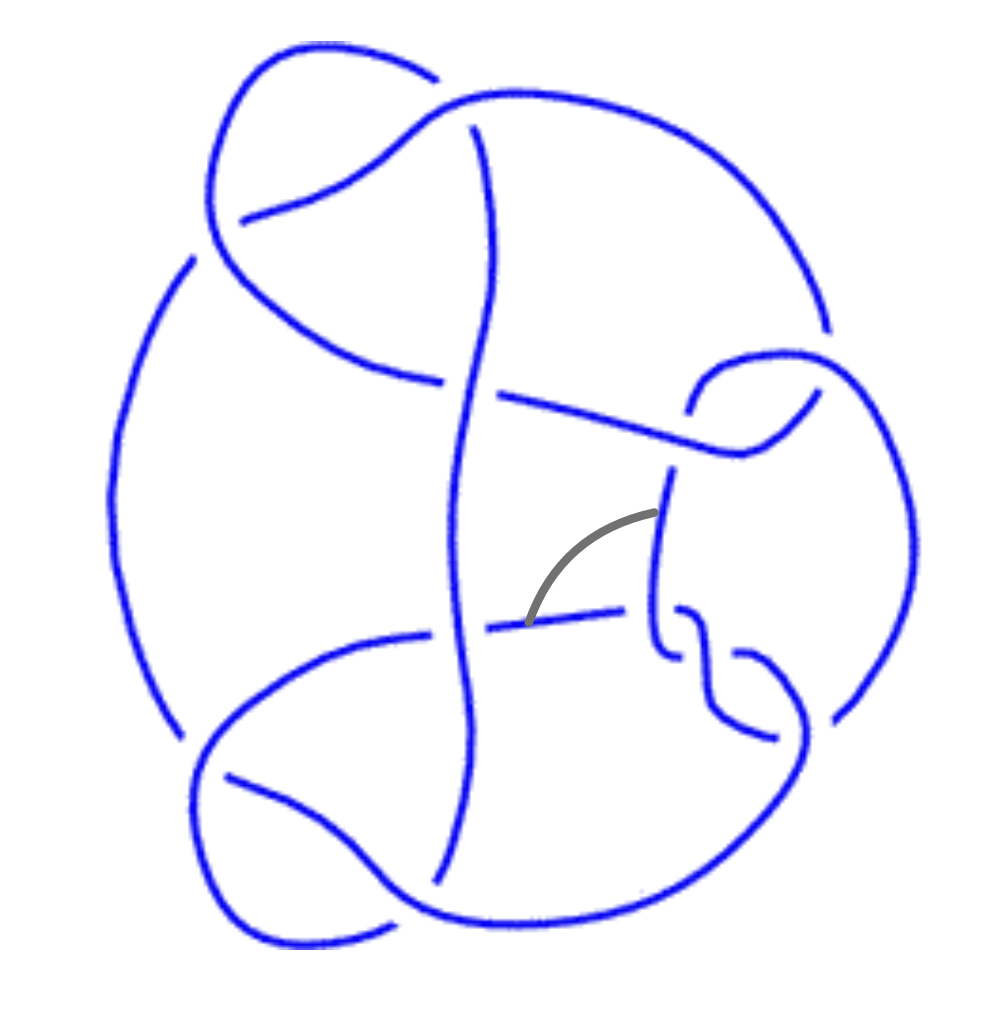}
		\caption{$11n_{76}\stackrel{0}{\longrightarrow} 8_{20}$}
		
	\end{subfigure}
 ~
	\begin{subfigure}[b]{0.25\textwidth}
		\includegraphics[width=\textwidth]{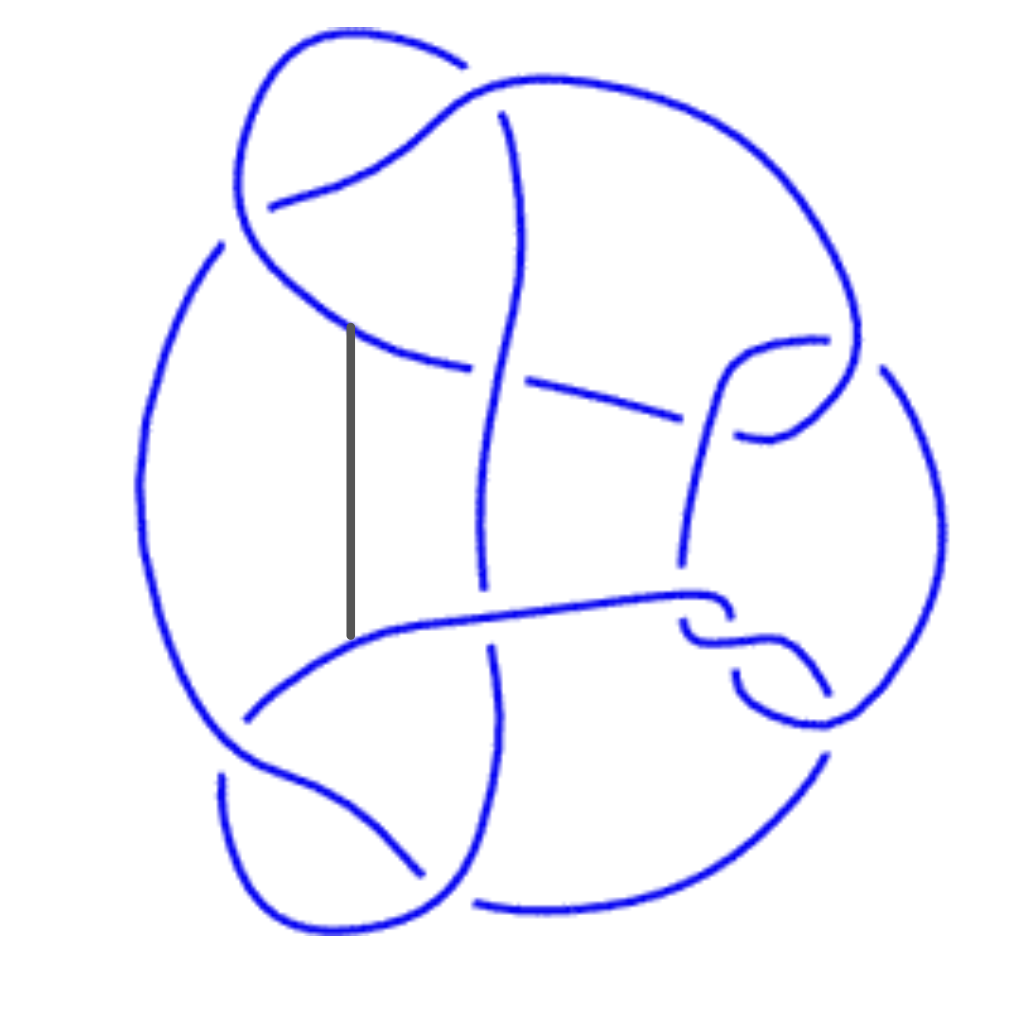}
		\caption{$11n_{77}\stackrel{-1}{\longrightarrow} 8_{20}$}
		
	\end{subfigure}
 
 \vskip3mm
	\begin{subfigure}[b]{0.25\textwidth}
		\includegraphics[width=\textwidth]{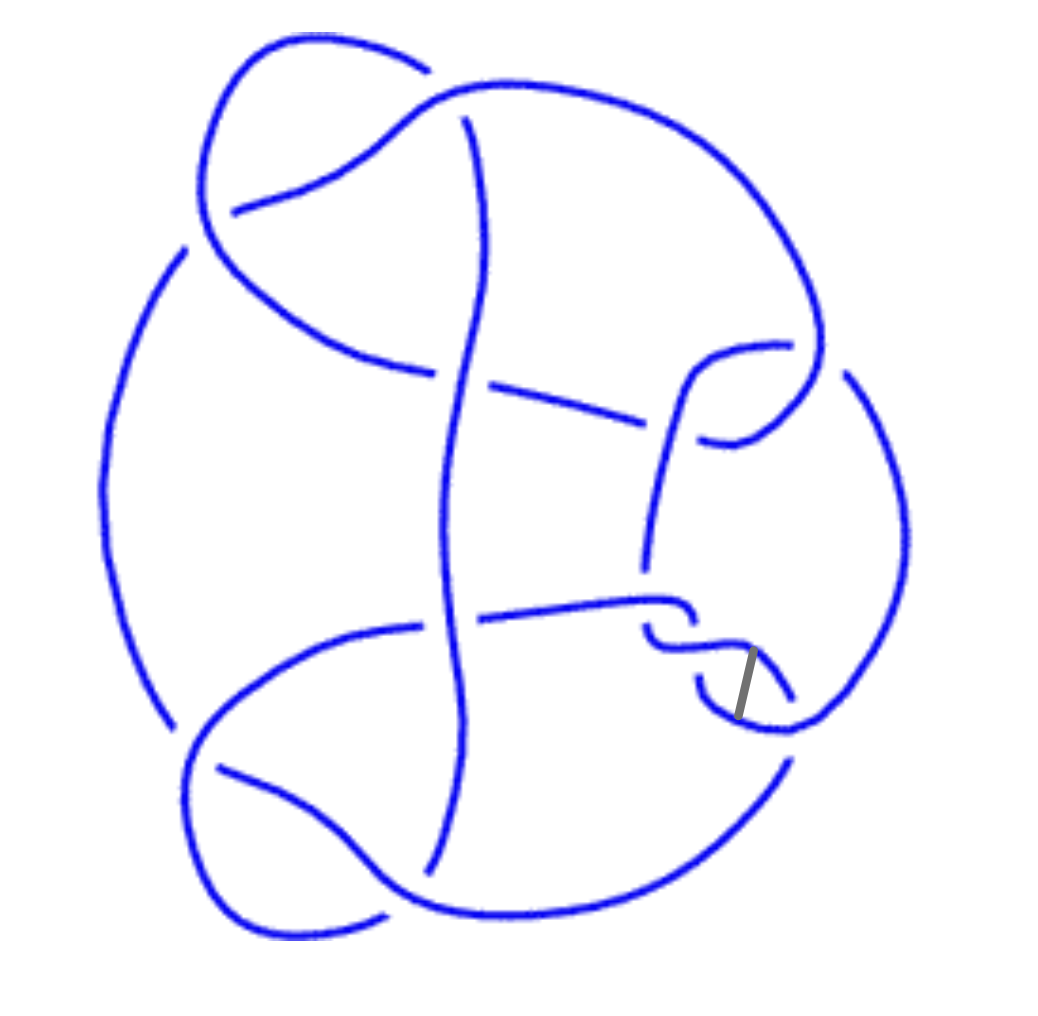}
		\caption{$11n_{78}\stackrel{0}{\longrightarrow} 8_{20}$}
		
	\end{subfigure}
 ~
	\begin{subfigure}[b]{0.25\textwidth}
		\includegraphics[width=\textwidth]{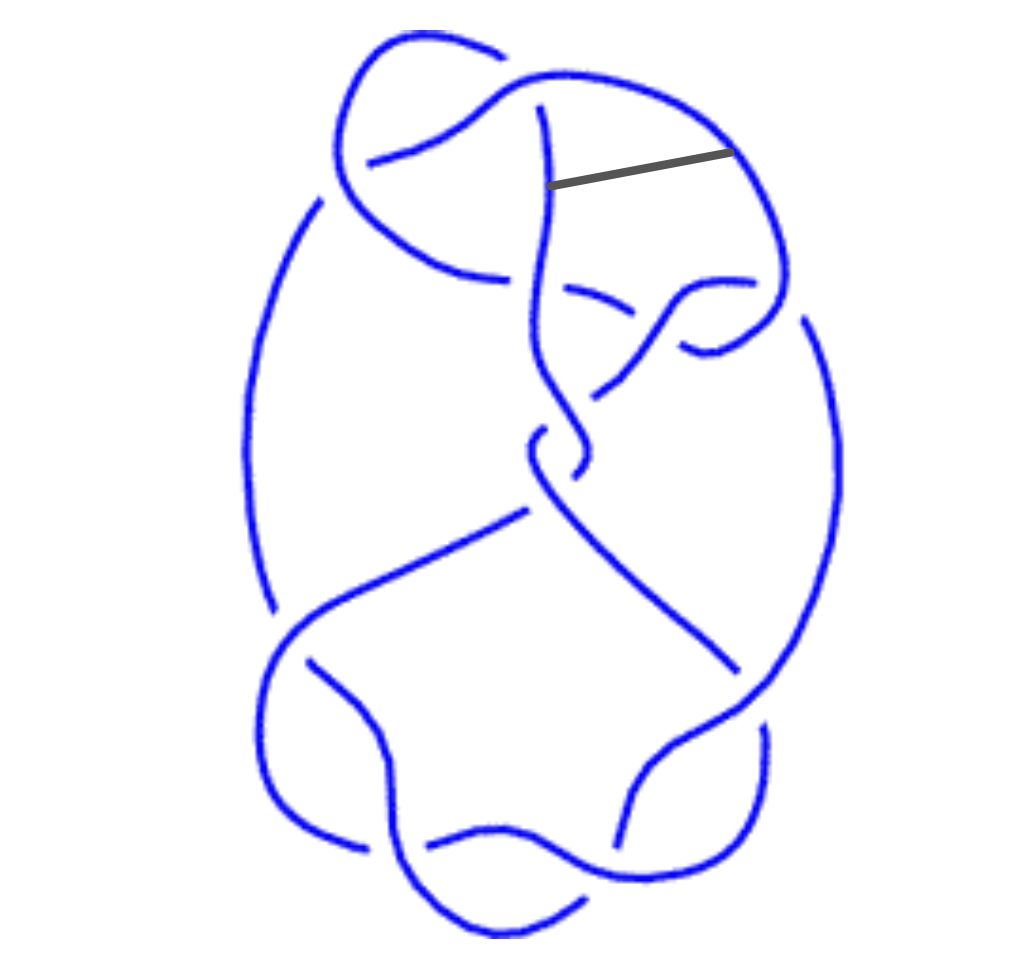}
		\caption{$11n_{79}\stackrel{0}{\longrightarrow} 0_1$}
		
	\end{subfigure}
 ~
 \begin{subfigure}[b]{0.26\textwidth}
		\includegraphics[width=\textwidth]{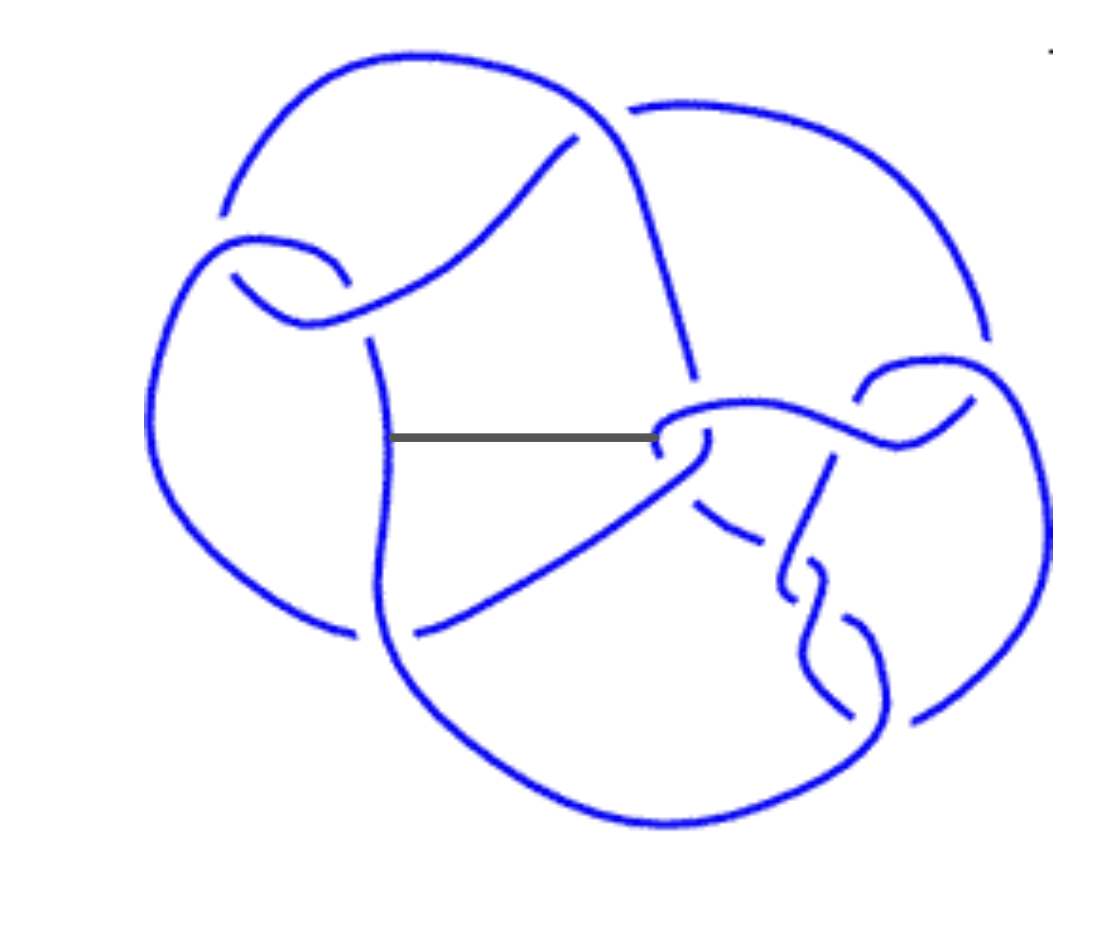}
		\caption{$11n_{80}\stackrel{-1}{\longrightarrow} 0_1$}
		
	\end{subfigure}
	\vskip3mm
	\caption{Non-oriented band moves from the knots $11n_{65},  11n_{66},  11n_{68},    $ \\ $  11n_{69}, 11n_{70}, 11n_{71}, 11n_{75}, 11n_{76}, 11n_{77}, 11n_{78}, 11n_{79}, \text{ and } 11n_{80} $ to smoothly slice knots.}
\end{figure}
%

\newpage

\begin{figure}[!htbp]
	\centering
	\begin{subfigure}[b]{0.26\textwidth}
		\includegraphics[width=\textwidth]{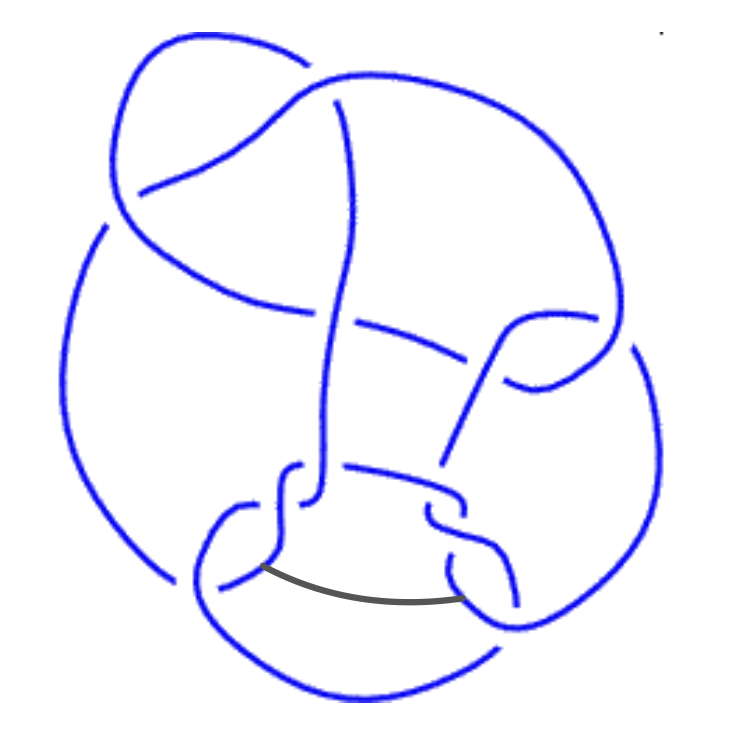}
		\caption{$11n_{81}\stackrel{1}{\longrightarrow} 8_{20}$}
		
	\end{subfigure}
	~
	\begin{subfigure}[b]{0.26\textwidth}
		\includegraphics[width=\textwidth]{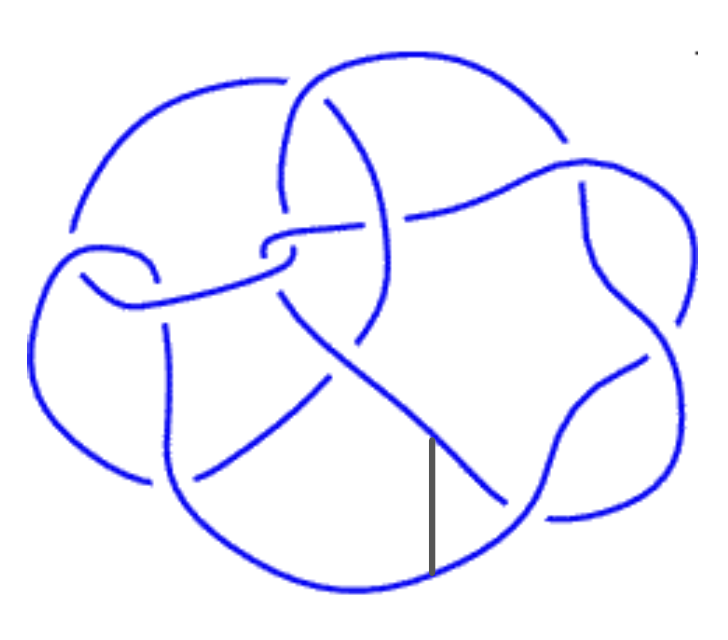}
		\caption{$11n_{82}\stackrel{0}{\longrightarrow} 0_1$}
		
	\end{subfigure}
	~
	\begin{subfigure}[b]{0.26\textwidth}
		\includegraphics[width=\textwidth]{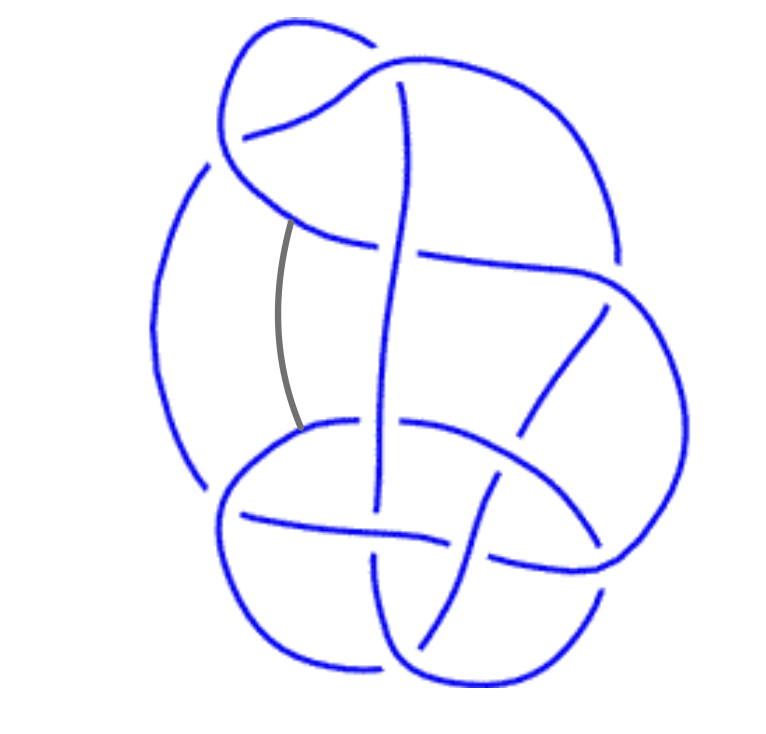}
		\caption{$11n_{86}\stackrel{0\phantom{i}}{\longrightarrow} 0_{1}$}
		
	\end{subfigure}
	\vskip3mm
	\begin{subfigure}[b]{0.26\textwidth}
		\includegraphics[width=\textwidth]{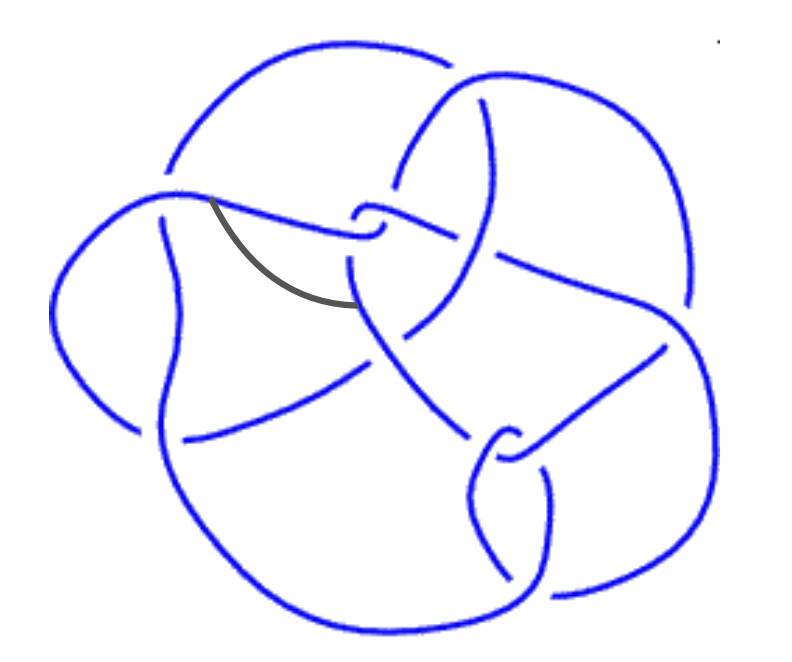}
		\caption{$11n_{87}\stackrel{0}{\longrightarrow} 8_{8}$}
		
	\end{subfigure}
	~
	\begin{subfigure}[b]{0.26\textwidth}
		\includegraphics[width=\textwidth]{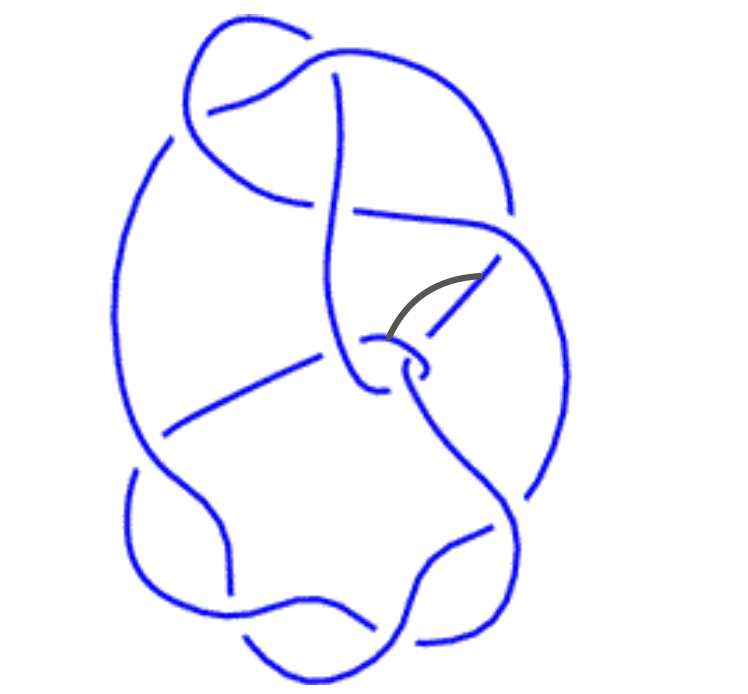}
		\caption{$11n_{88}\stackrel{0}{\longrightarrow} 6_{1}$}
		
	\end{subfigure}
	~
	\begin{subfigure}[b]{0.26\textwidth}
		\includegraphics[width=\textwidth]{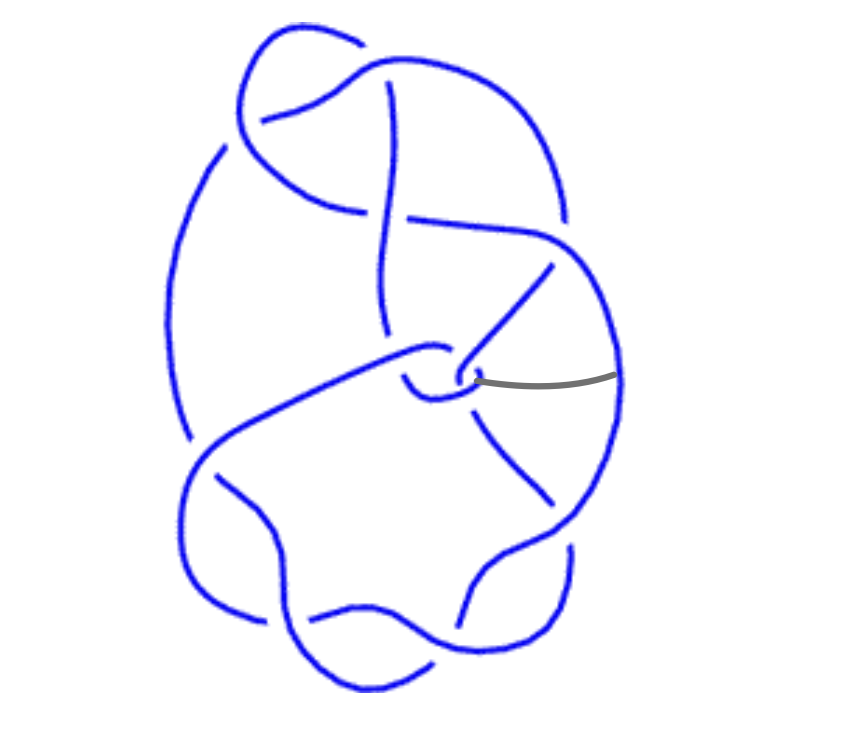}
		\caption{$11n_{89}\stackrel{0}{\longrightarrow} 8_{8}$}
		
	\end{subfigure}
	\vskip3mm
	\begin{subfigure}[b]{0.25\textwidth}
		\includegraphics[width=\textwidth]{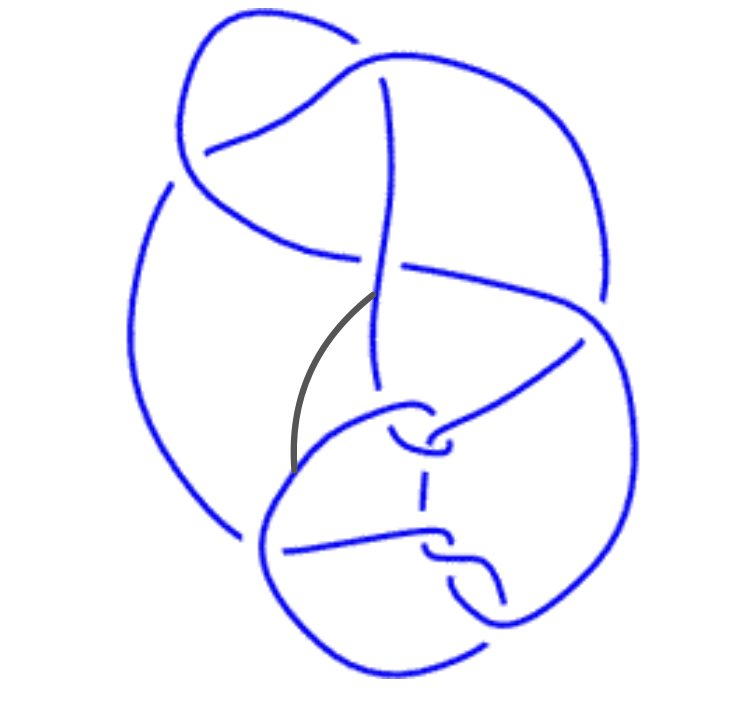}
		\caption{$11n_{91}\stackrel{1}{\longrightarrow} 12n_{145}$}
		
	\end{subfigure}
	~
	\begin{subfigure}[b]{0.25\textwidth}
		\includegraphics[width=\textwidth]{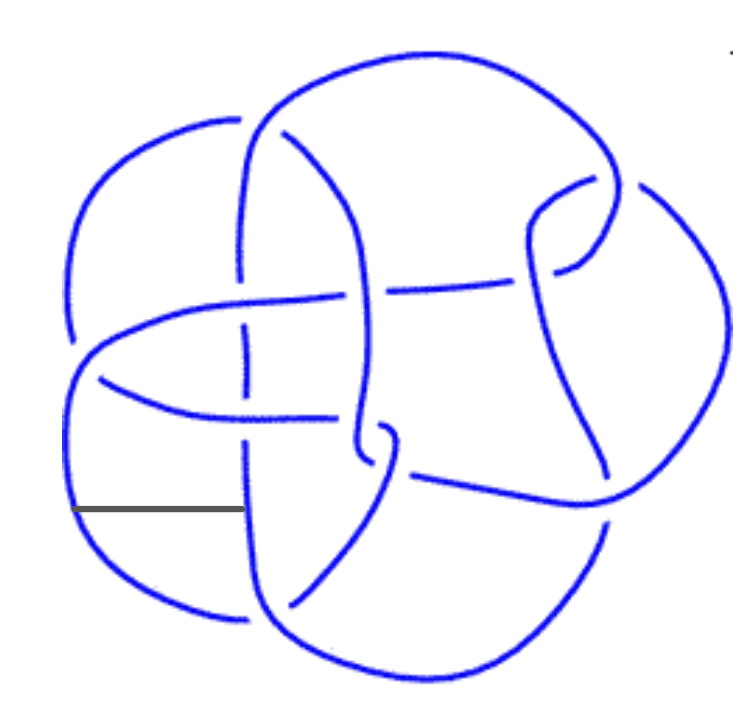}
		\caption{$11n_{93}\stackrel{1}{\longrightarrow} 10_{137}$}
		
	\end{subfigure}
	~
	\begin{subfigure}[b]{0.25\textwidth}
		\includegraphics[width=\textwidth]{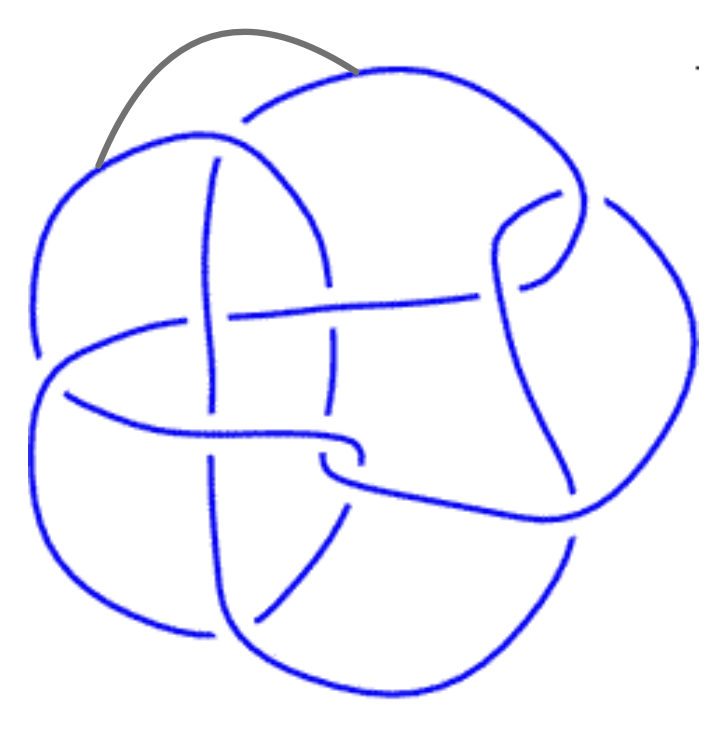}
		\caption{$11n_{94}\stackrel{-1}{\longrightarrow} 10_{137}$}
		
	\end{subfigure}
 \vskip3mm
 ~
	\begin{subfigure}[b]{0.25\textwidth}
		\includegraphics[width=\textwidth]{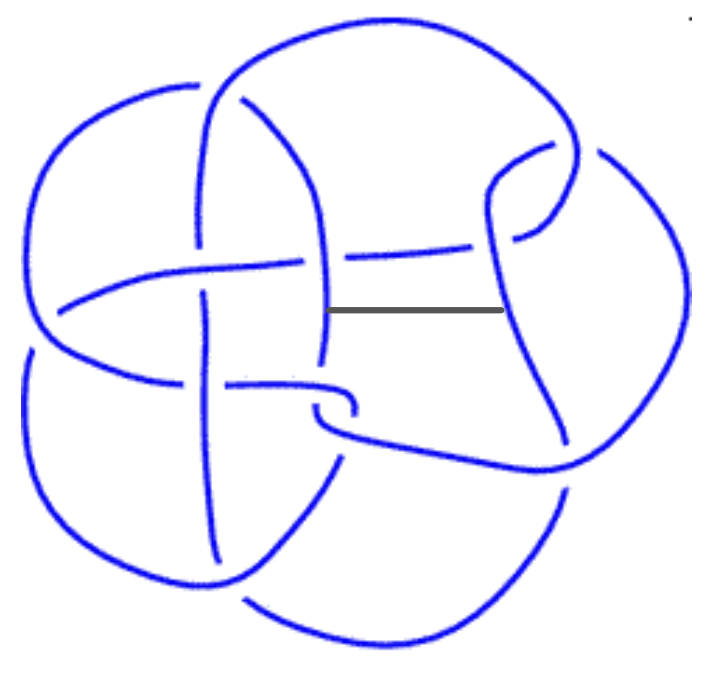}
		\caption{$11n_{96}\stackrel{0}{\longrightarrow} 0_1$}
		
	\end{subfigure}
 ~
	\begin{subfigure}[b]{0.25\textwidth}
		\includegraphics[width=\textwidth]{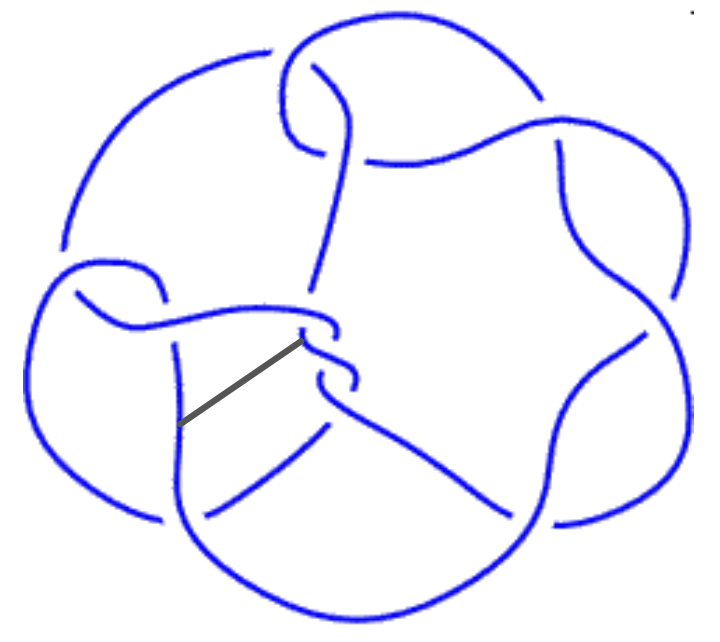}
		\caption{$11n_{102}\stackrel{0}{\longrightarrow} 0_1$}
		
	\end{subfigure}
 ~
	\begin{subfigure}[b]{0.25\textwidth}
		\includegraphics[width=\textwidth]{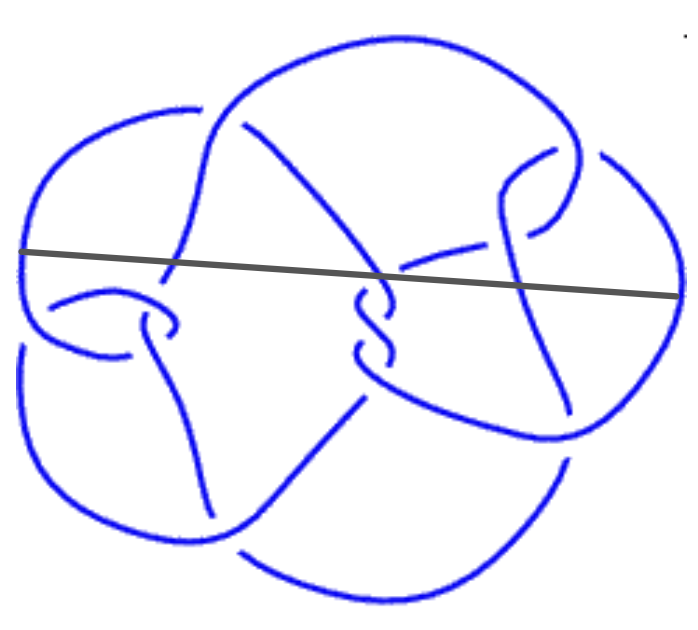}
		\caption{$11n_{104}\stackrel{0}{\longrightarrow} 0_1$}
		
	\end{subfigure}
	\vskip3mm
	\caption{Non-oriented band moves from the knots $11n_{81},  11n_{82},  11n_{86},    $ \\ $ 11n_{87},  11n_{88}, 11n_{89}, 11n_{91}, 11n_{93}, 11n_{94}, 11n_{96}, 11n_{102}, \text{ and } 11n_{104} $ to smoothly slice knots.}
\end{figure}
%

\newpage

\begin{figure}[!htbp]
	\centering
	\begin{subfigure}[b]{0.25\textwidth}
		\includegraphics[width=\textwidth]{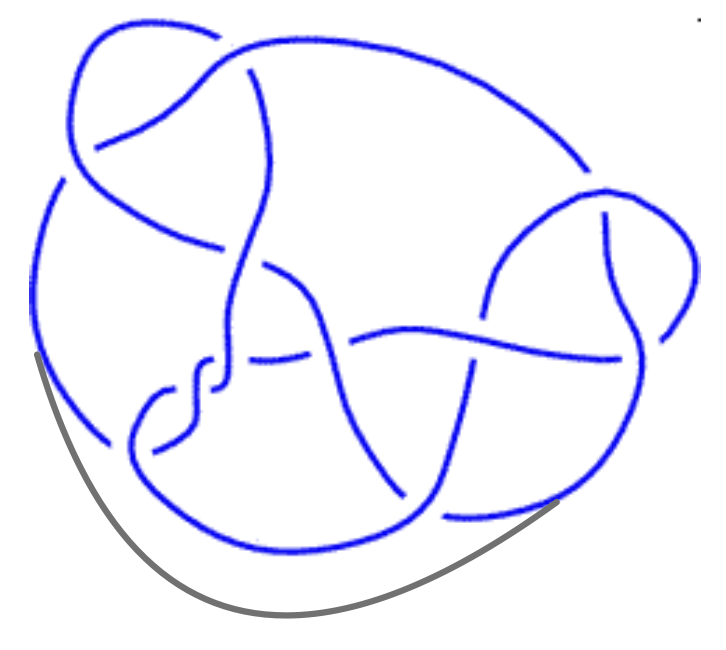}
		\caption{$11n_{105}\stackrel{1}{\longrightarrow} 9_{27}$}
		
	\end{subfigure}
	~
	\begin{subfigure}[b]{0.25\textwidth}
		\includegraphics[width=\textwidth]{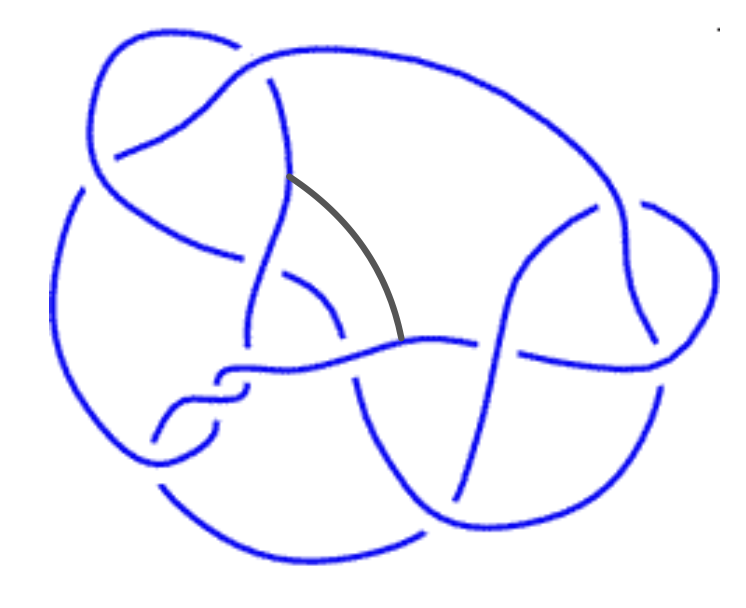}
		\caption{$11n_{106}\stackrel{0}{\longrightarrow} 0_1$}
		
	\end{subfigure}
	~
	\begin{subfigure}[b]{0.25\textwidth}
		\includegraphics[width=\textwidth]{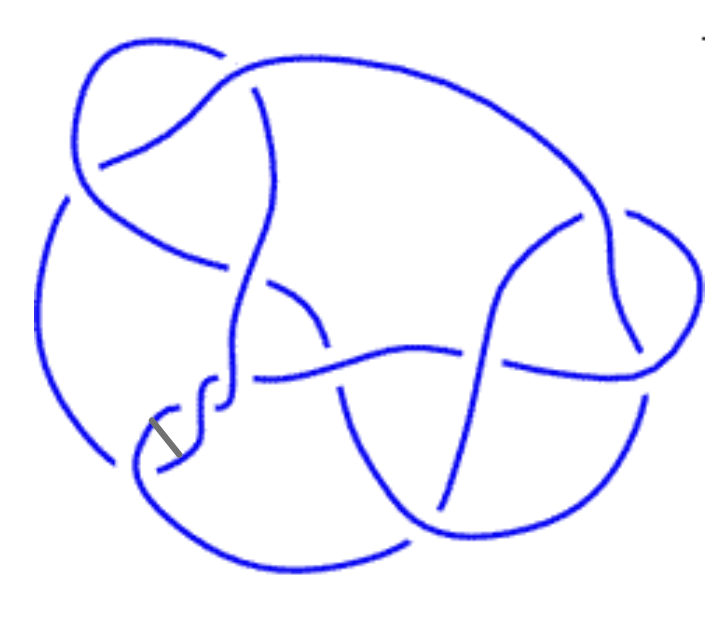}
		\caption{$11n_{107}\stackrel{0\phantom{i}}{\longrightarrow} 0_1$}
		
	\end{subfigure}
	\vskip3mm
	\begin{subfigure}[b]{0.25\textwidth}
		\includegraphics[width=\textwidth]{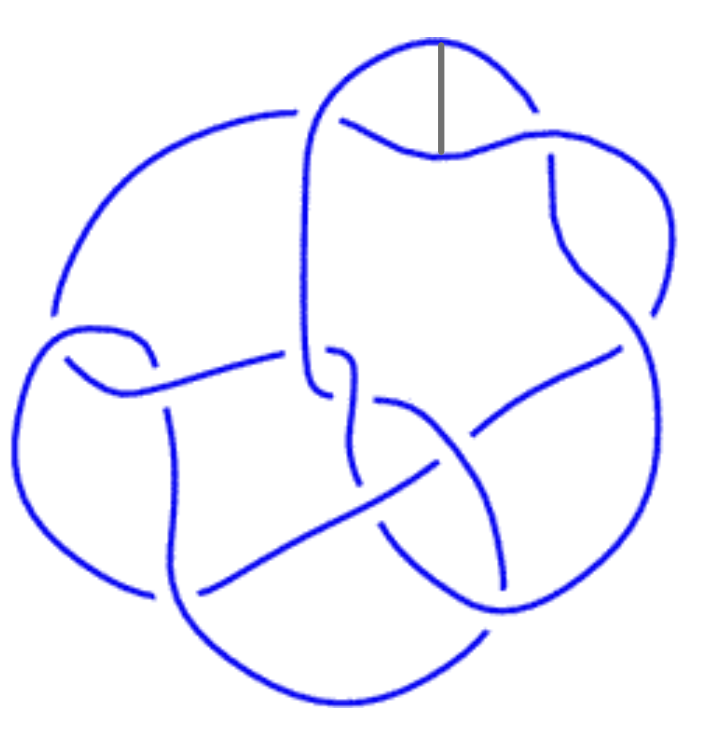}
		\caption{$11n_{110}\stackrel{0}{\longrightarrow} 8_{8}$}
		
	\end{subfigure}
	~
	\begin{subfigure}[b]{0.25\textwidth}
		\includegraphics[width=\textwidth]{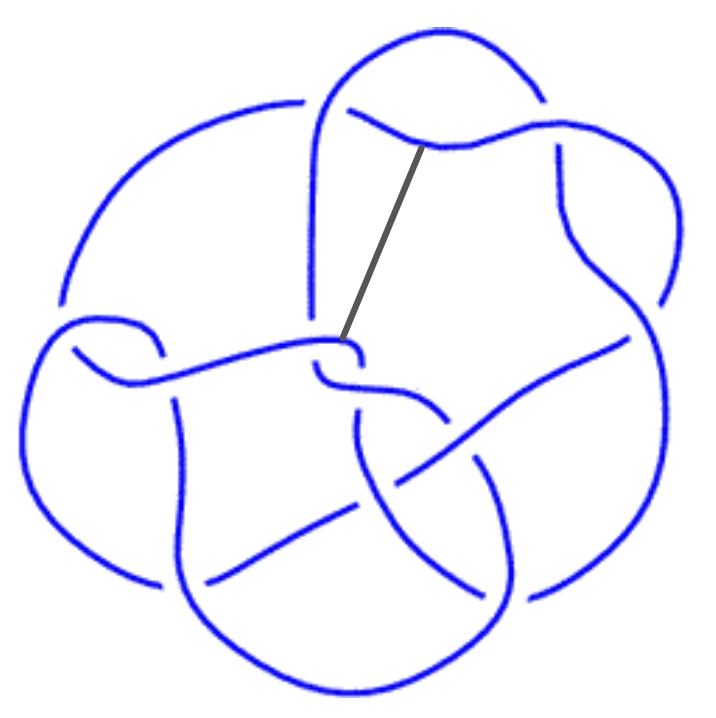}
		\caption{$11n_{111}\stackrel{1}{\longrightarrow} 0_{1}$}
		
	\end{subfigure}
	~
	\begin{subfigure}[b]{0.26\textwidth}
		\includegraphics[width=\textwidth]{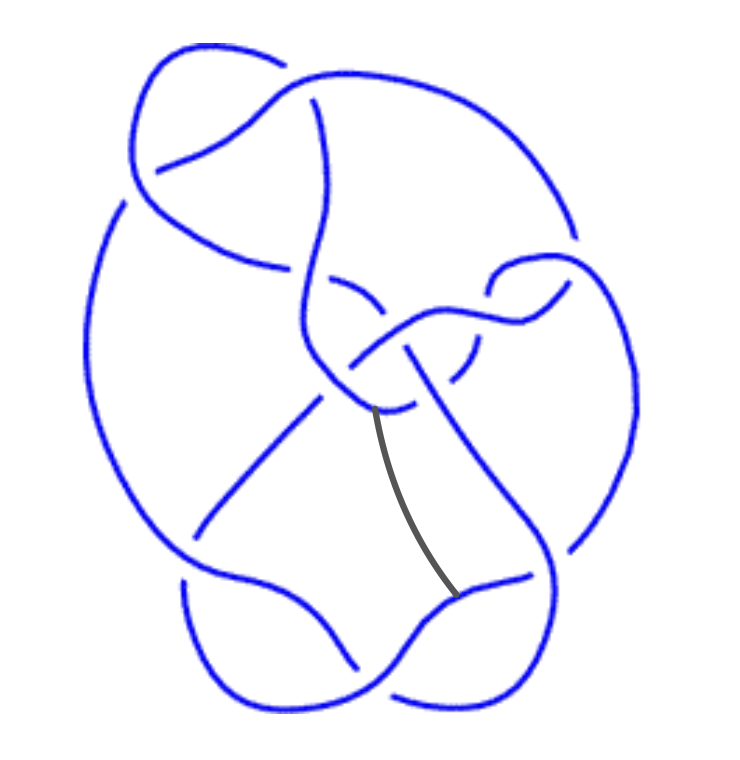}
		\caption{$11n_{113}\stackrel{-1}{\longrightarrow} 9_{47}$}
		
	\end{subfigure}
	\vskip3mm
	\begin{subfigure}[b]{0.25\textwidth}
		\includegraphics[width=\textwidth]{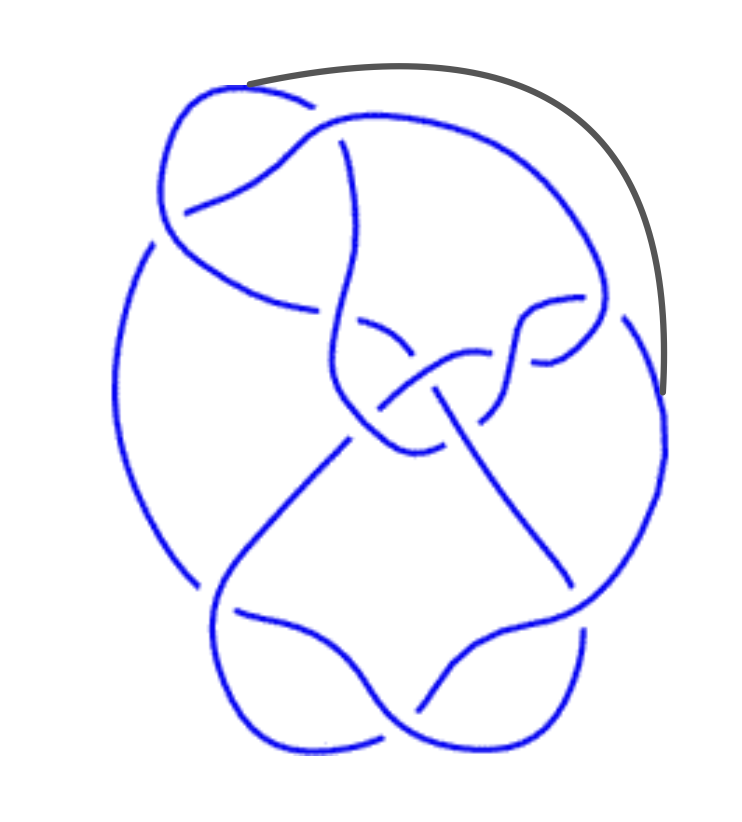}
		\caption{$11n_{117}\stackrel{1}{\longrightarrow} 12n_{414}$}
		
	\end{subfigure}
	~
	\begin{subfigure}[b]{0.26\textwidth}
		\includegraphics[width=\textwidth]{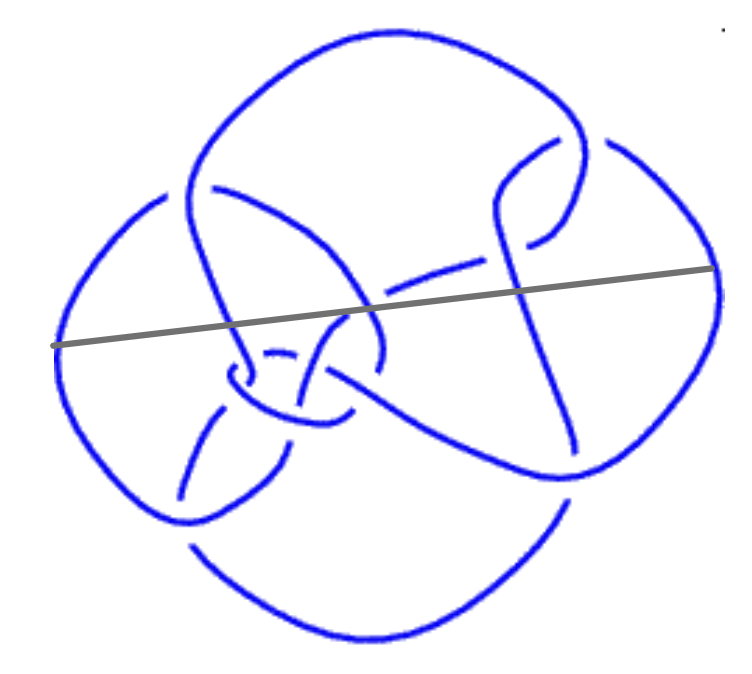}
		\caption{$11n_{118}\stackrel{0}{\longrightarrow} 0_{1}$}
		
	\end{subfigure}
	~
	\begin{subfigure}[b]{0.25\textwidth}
		\includegraphics[width=\textwidth]{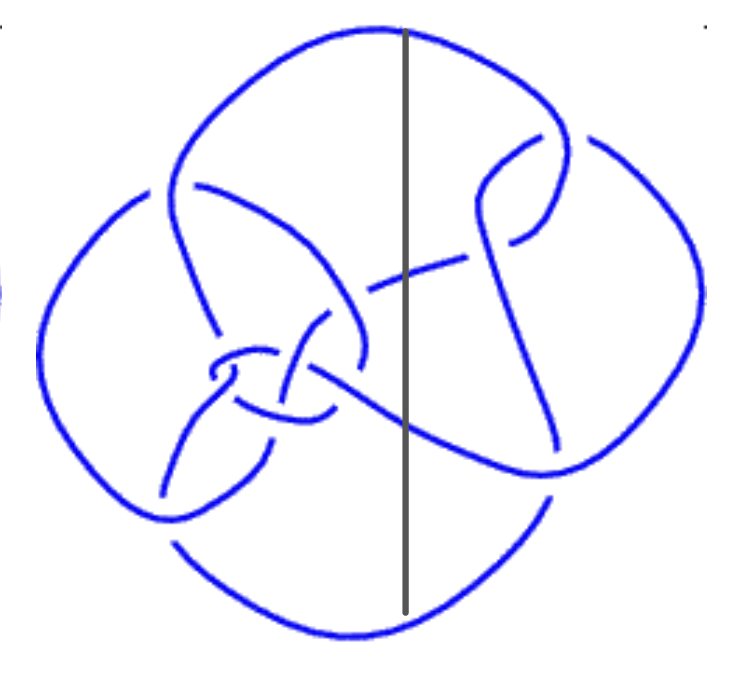}
		\caption{$11n_{120}\stackrel{1}{\longrightarrow} 12n_{312}$}
		
	\end{subfigure}
 \vskip3mm
 ~
	\begin{subfigure}[b]{0.25\textwidth}
		\includegraphics[width=\textwidth]{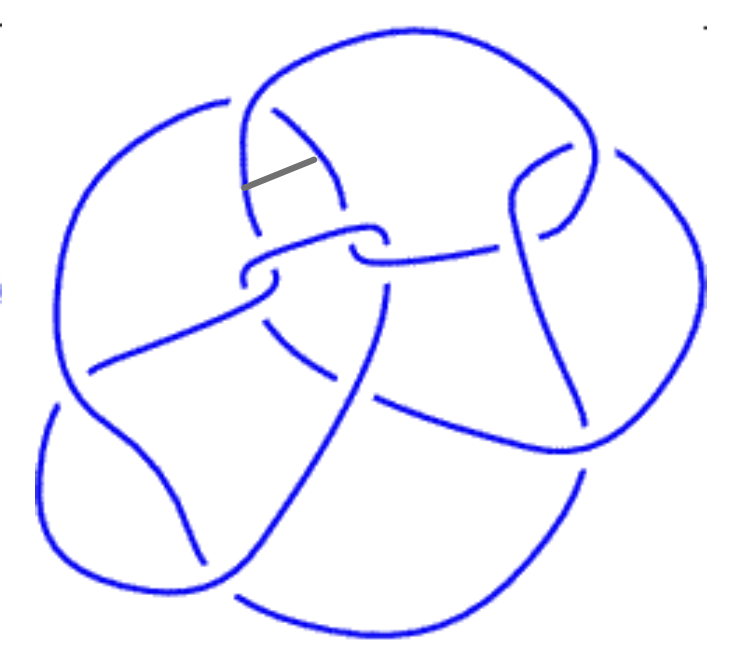}
		\caption{$11n_{121}\stackrel{0}{\longrightarrow} 0_1$}
		
	\end{subfigure}
 ~
	\begin{subfigure}[b]{0.25\textwidth}
		\includegraphics[width=\textwidth]{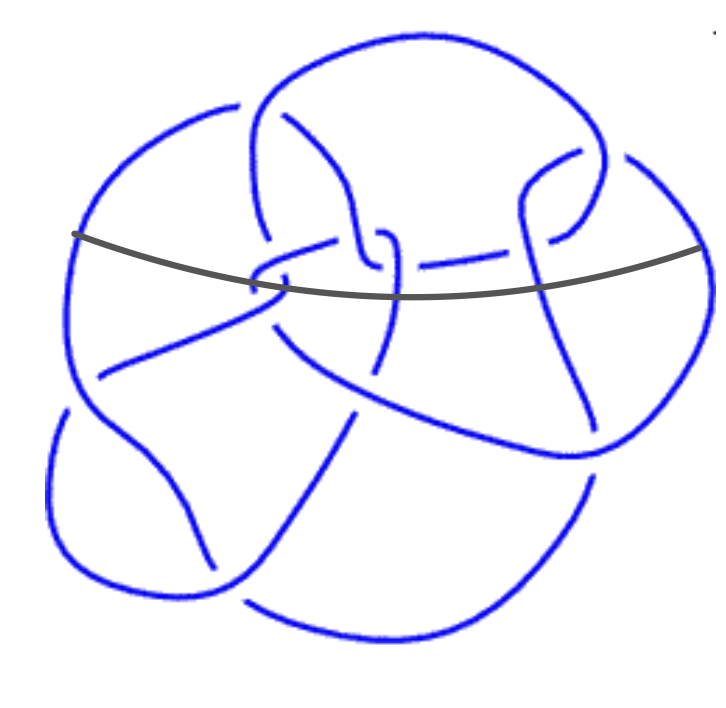}
		\caption{$11n_{122}\stackrel{0}{\longrightarrow} 0_1$}
		
	\end{subfigure}
 ~
	\begin{subfigure}[b]{0.25\textwidth}
		\includegraphics[width=\textwidth]{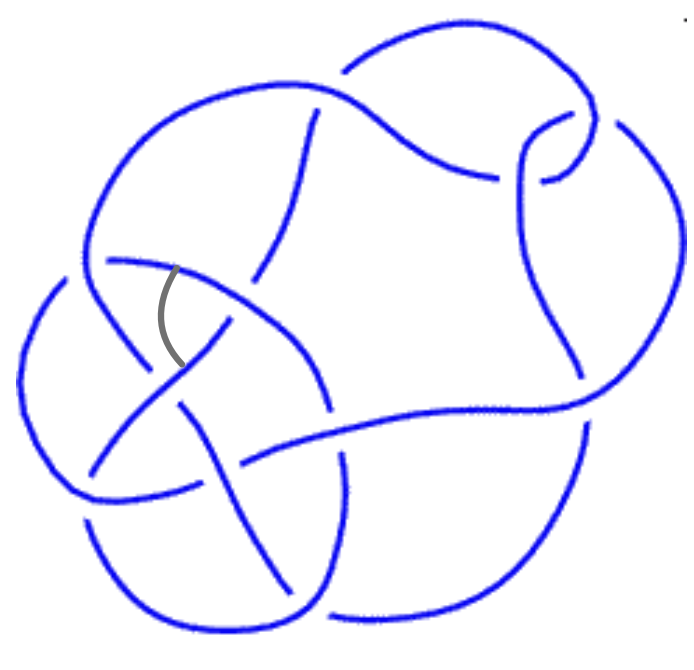}
		\caption{$11n_{123}\stackrel{0}{\longrightarrow} 9_{27}$}
		
	\end{subfigure}
	\vskip3mm
	\caption{Non-oriented band moves from the knots $11n_{105},  11n_{106},  11n_{107},    $ \\ $ 11n_{110},  11n_{111}, 11n_{113}, 11n_{117}, 11n_{118}, 11n_{120}, 11n_{121}, 11n_{122}, \text{ and } 11n_{123} $ \\ to smoothly slice knots.}
\end{figure}
%

\newpage

\begin{figure}[!htbp]
	\centering
	\begin{subfigure}[b]{0.26\textwidth}
		\includegraphics[width=\textwidth]{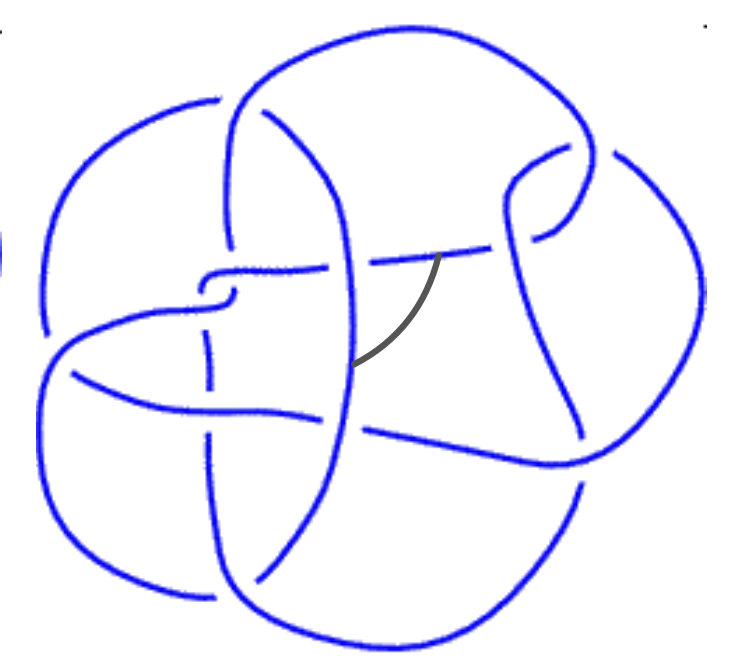}
		\caption{$11n_{124}\stackrel{1}{\longrightarrow} 8_{20}$}
		
	\end{subfigure}
	~
	\begin{subfigure}[b]{0.25\textwidth}
		\includegraphics[width=\textwidth]{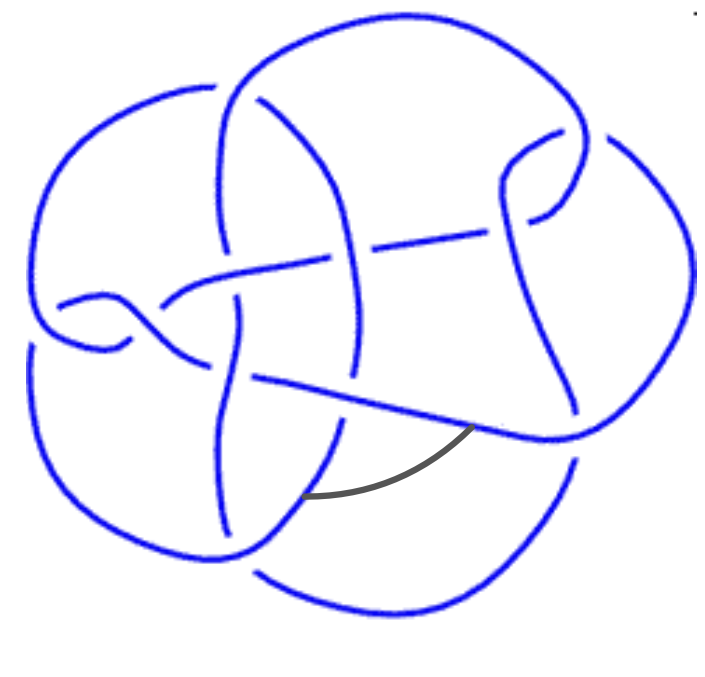}
		\caption{$11n_{126}\stackrel{0}{\longrightarrow} 8_{20}$}
		
	\end{subfigure}
	~
	\begin{subfigure}[b]{0.25\textwidth}
		\includegraphics[width=\textwidth]{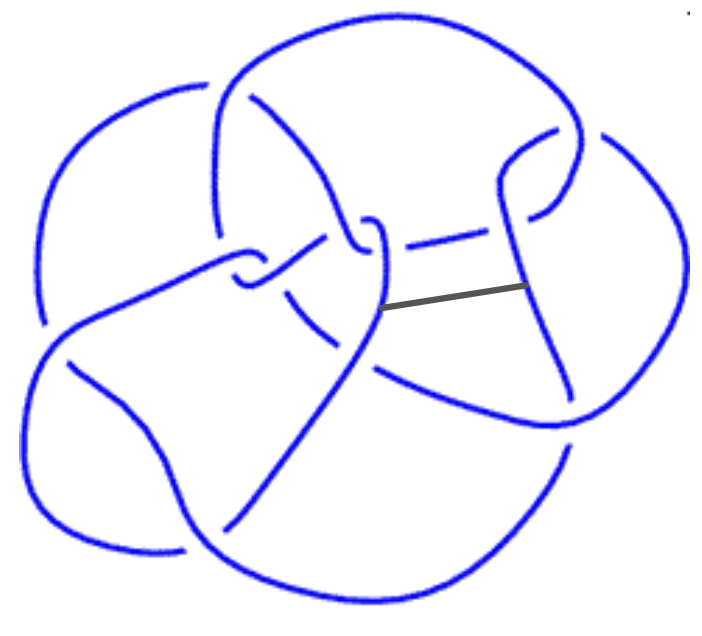}
		\caption{$11n_{127}\stackrel{0\phantom{i}}{\longrightarrow} 0_{1}$}
		
	\end{subfigure}
	\vskip3mm
	\begin{subfigure}[b]{0.25\textwidth}
		\includegraphics[width=\textwidth]{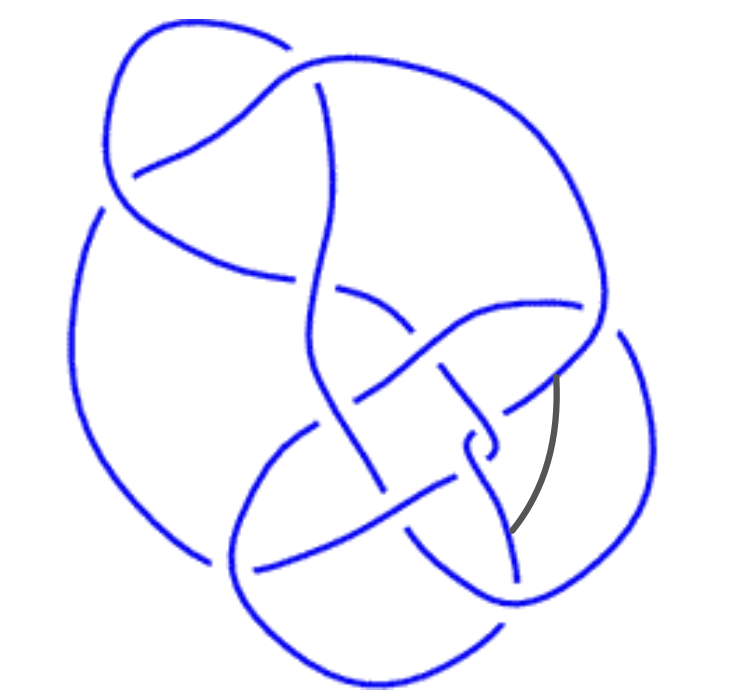}
		\caption{$11n_{128}\stackrel{1}{\longrightarrow} 10_{140}$}
		
	\end{subfigure}
	~
	\begin{subfigure}[b]{0.25\textwidth}
		\includegraphics[width=\textwidth]{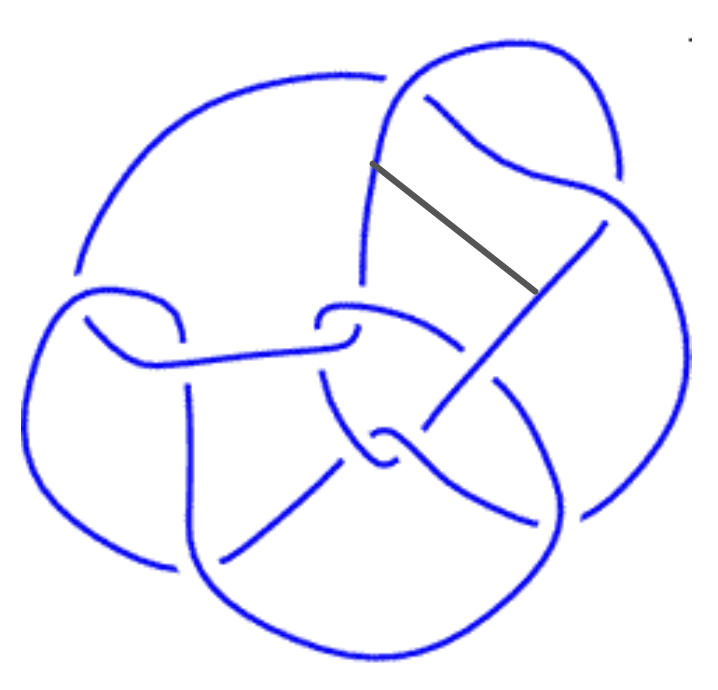}
		\caption{$11n_{134}\stackrel{1}{\longrightarrow} 11_{116}$}
		
	\end{subfigure}
	~
	\begin{subfigure}[b]{0.25\textwidth}
		\includegraphics[width=\textwidth]{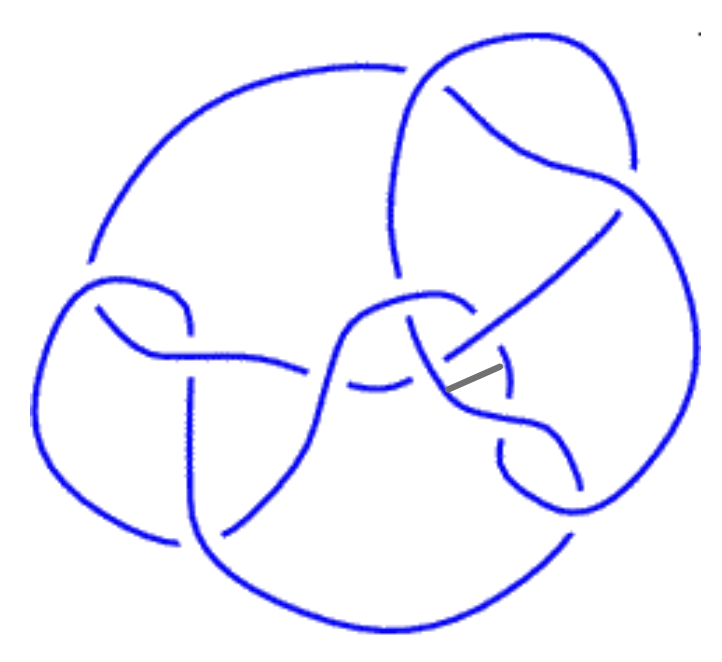}
		\caption{$11n_{135}\stackrel{0}{\longrightarrow} 8_{20}$}
		
	\end{subfigure}
	\vskip3mm
	\begin{subfigure}[b]{0.25\textwidth}
		\includegraphics[width=\textwidth]{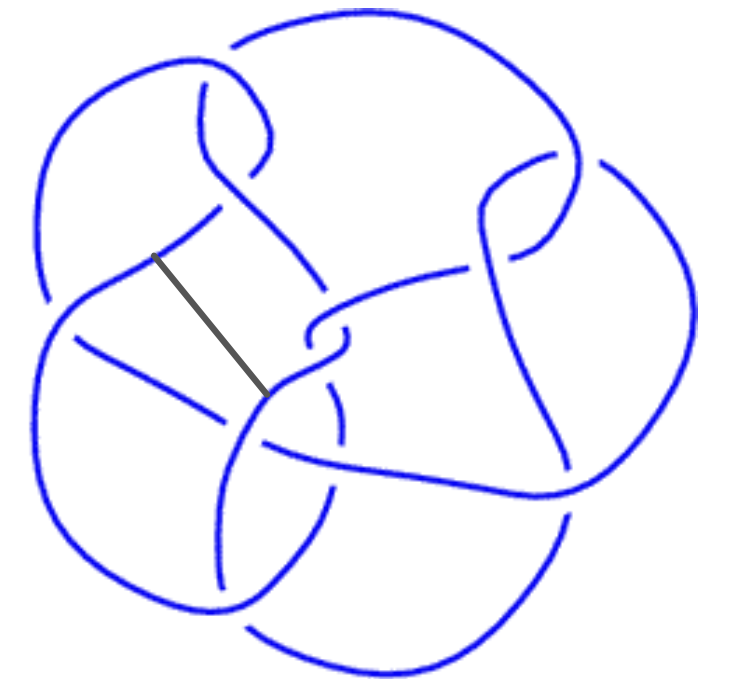}
		\caption{$11n_{136}\stackrel{0}{\longrightarrow} 6_{1}$}
		
	\end{subfigure}
	~
	\begin{subfigure}[b]{0.25\textwidth}
		\includegraphics[width=\textwidth]{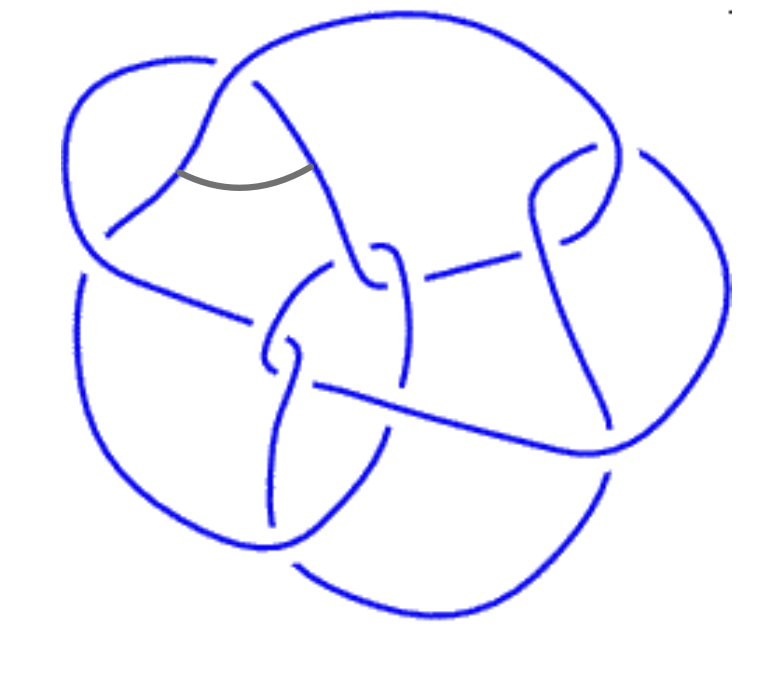}
		\caption{$11n_{142}\stackrel{0}{\longrightarrow} 10_{129}$}
		
	\end{subfigure}
	~
	\begin{subfigure}[b]{0.25\textwidth}
		\includegraphics[width=\textwidth]{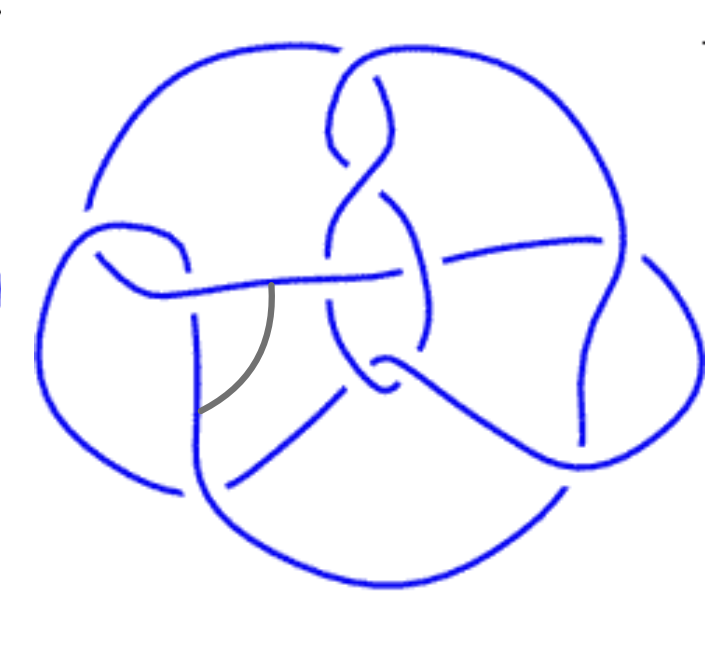}
		\caption{$11n_{143}\stackrel{0}{\longrightarrow} 8_{20}$}
		
	\end{subfigure}
 \vskip3mm
 ~
	\begin{subfigure}[b]{0.25\textwidth}
		\includegraphics[width=\textwidth]{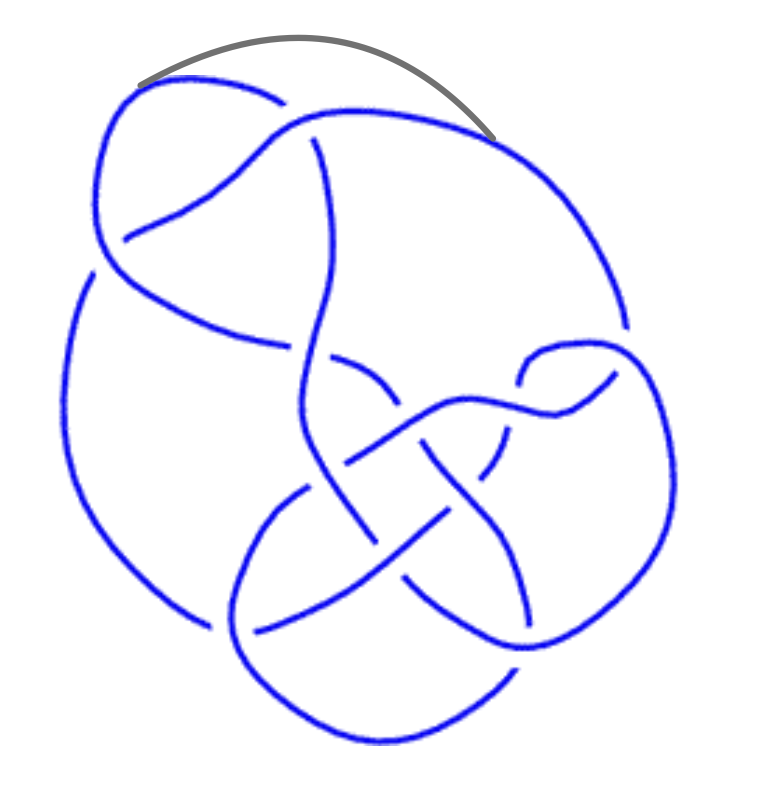}
		\caption{$11n_{145}\stackrel{1}{\longrightarrow} 6_1$}
		
	\end{subfigure}
 ~
	\begin{subfigure}[b]{0.25\textwidth}
		\includegraphics[width=\textwidth]{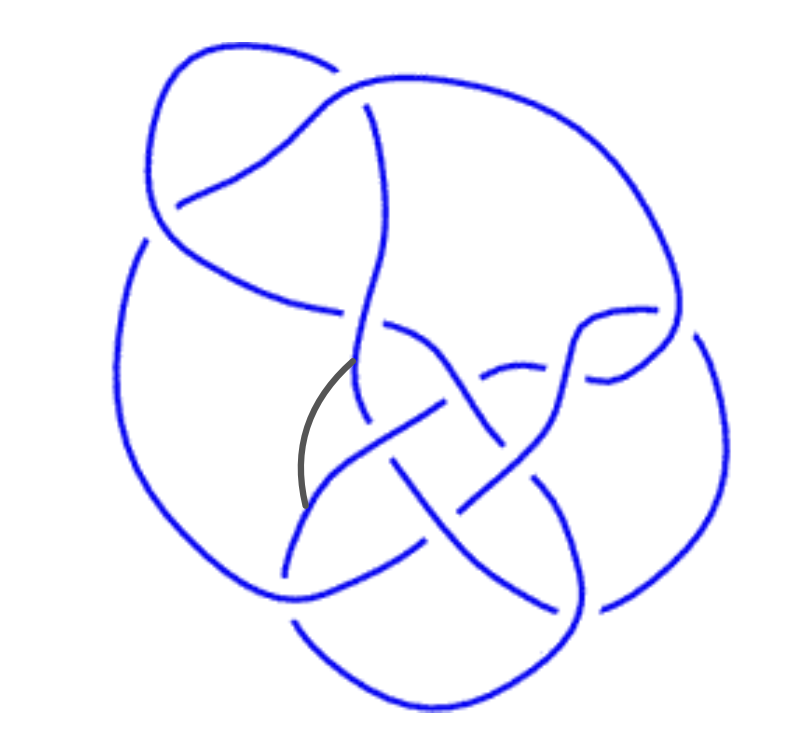}
		\caption{$11n_{146}\stackrel{-1}{\longrightarrow} 10_{137}$}
		
	\end{subfigure}
 ~
	\begin{subfigure}[b]{0.25\textwidth}
		\includegraphics[width=\textwidth]{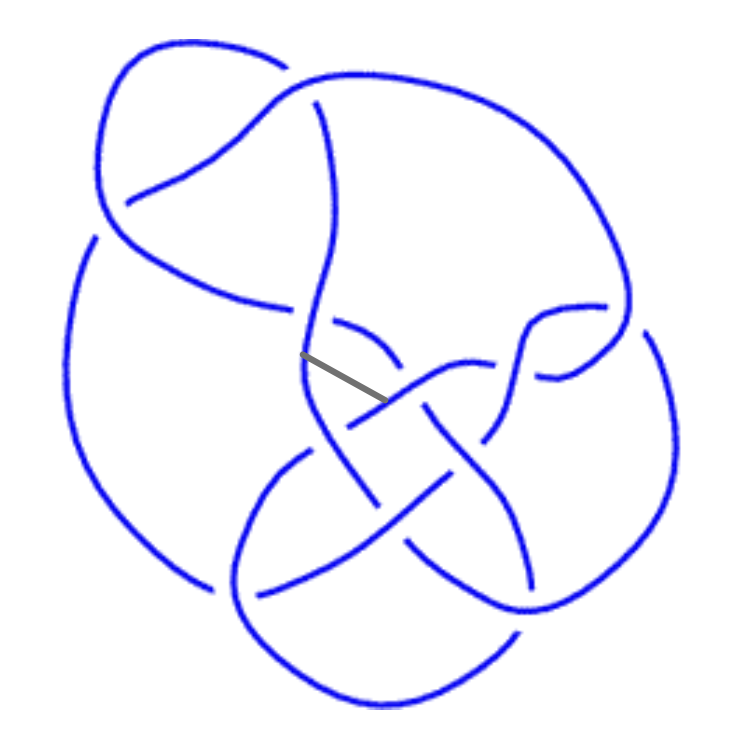}
		\caption{$11n_{147}\stackrel{0}{\longrightarrow} 6_1$}
		
	\end{subfigure}
	\vskip3mm
	\caption{Non-oriented band moves from the knots $11n_{124},  11n_{126},  11n_{127},     $ \\ $ 11n_{128}, 11n_{134}, 11n_{135}, 11n_{136}, 11n_{142}, 11n_{143}, 11n_{145}, 11n_{146}, \text{ and } 11n_{147} $ \\ to smoothly slice knots.}
\end{figure}
%


\begin{figure}[!htbp]
	\centering
	\begin{subfigure}[b]{0.26\textwidth}
		\includegraphics[width=\textwidth]{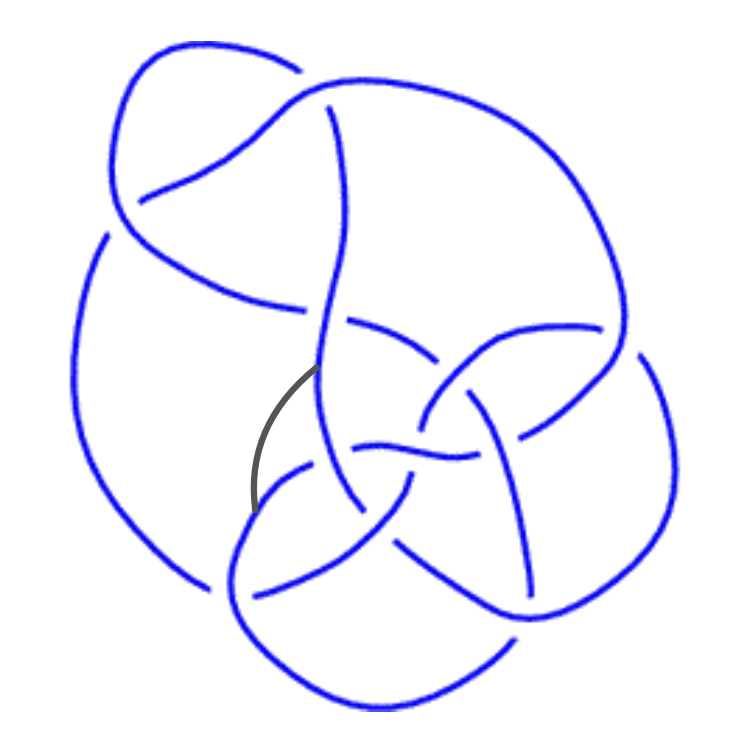}
		\caption{$11n_{148}\stackrel{1}{\longrightarrow} 10_{137}$}
		
	\end{subfigure}
	~
	\begin{subfigure}[b]{0.26\textwidth}
		\includegraphics[width=\textwidth]{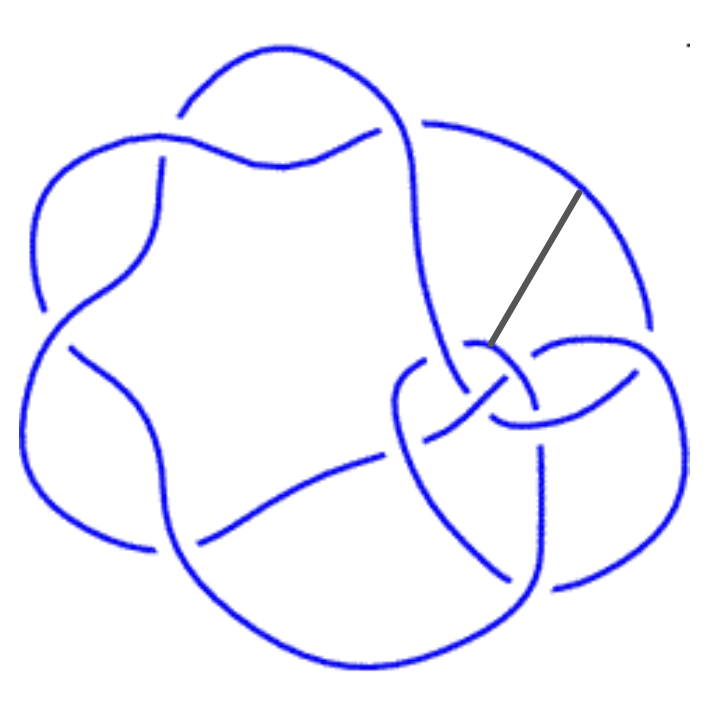}
		\caption{$11n_{150}\stackrel{0}{\longrightarrow} 8_{8}$}
		
	\end{subfigure}
	~
	\begin{subfigure}[b]{0.26\textwidth}
		\includegraphics[width=\textwidth]{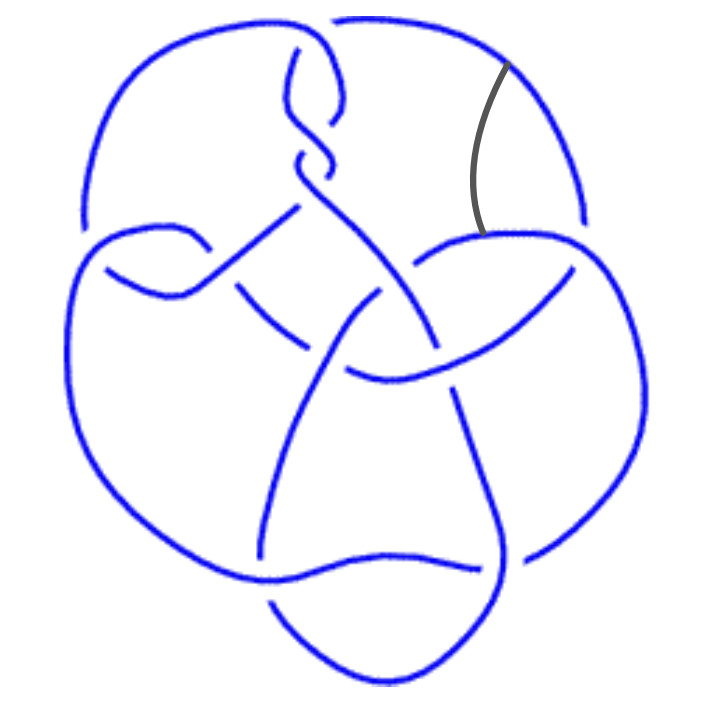}
		\caption{$11n_{151}\stackrel{-1\phantom{i}}{\longrightarrow} 10_{153}$}
		
	\end{subfigure}
	\vskip3mm
	\begin{subfigure}[b]{0.26\textwidth}
		\includegraphics[width=\textwidth]{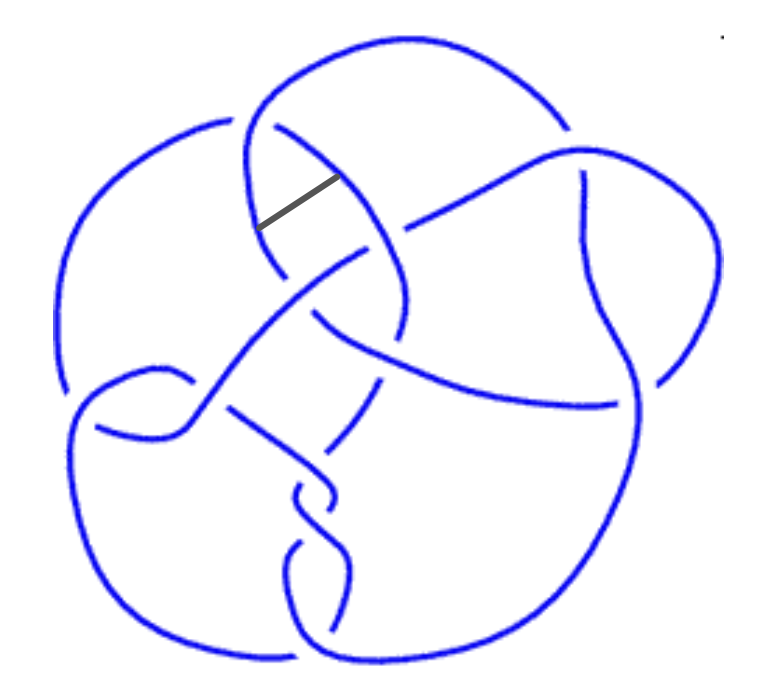}
		\caption{$11n_{152}\stackrel{0}{\longrightarrow} 10_{153}$}
		
	\end{subfigure}
	~
	\begin{subfigure}[b]{0.26\textwidth}
		\includegraphics[width=\textwidth]{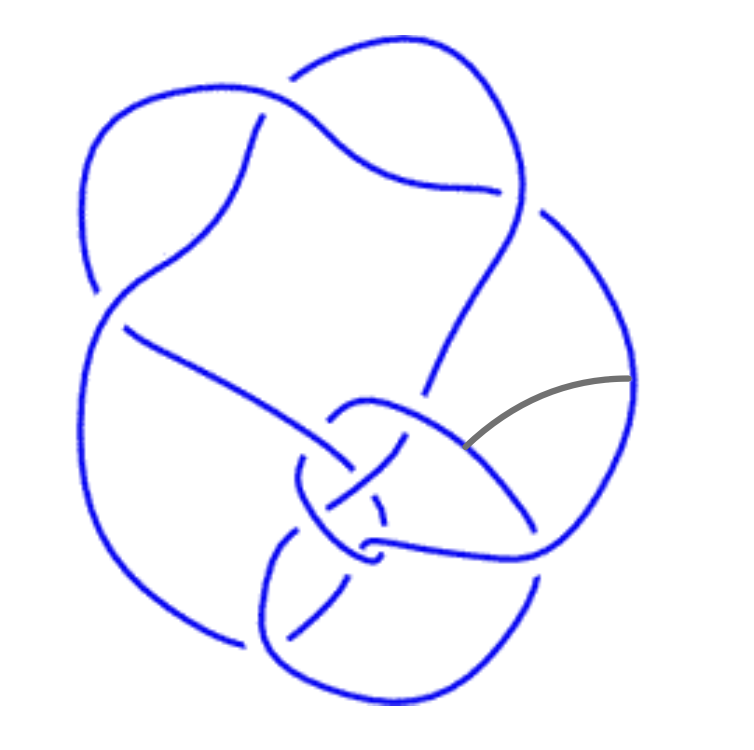}
		\caption{$11n_{153}\stackrel{1}{\longrightarrow} 10_{129}$}
		
	\end{subfigure}
	~
	\begin{subfigure}[b]{0.26\textwidth}
		\includegraphics[width=\textwidth]{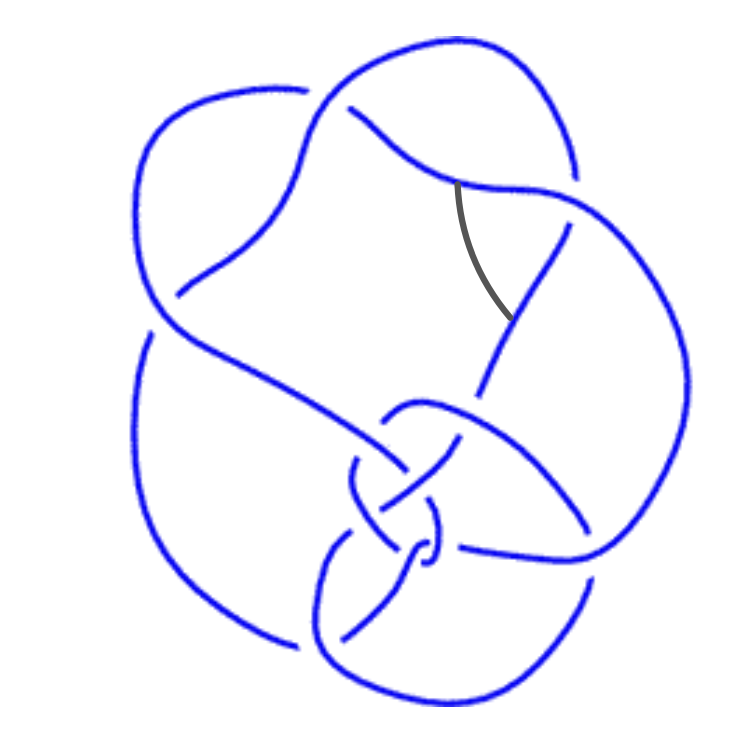}
		\caption{$11n_{154}\stackrel{-1}{\longrightarrow} 12n_{504}$}
		
	\end{subfigure}
	\vskip3mm
	\begin{subfigure}[b]{0.25\textwidth}
		\includegraphics[width=\textwidth]{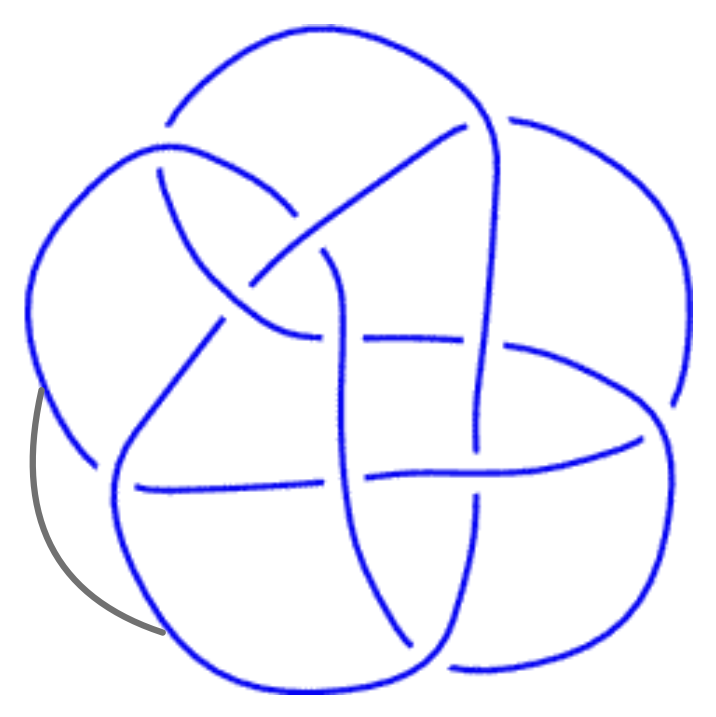}
		\caption{$11n_{157}\stackrel{-1}{\longrightarrow} 9_{27}$}
		
	\end{subfigure}
	~
	\begin{subfigure}[b]{0.25\textwidth}
		\includegraphics[width=\textwidth]{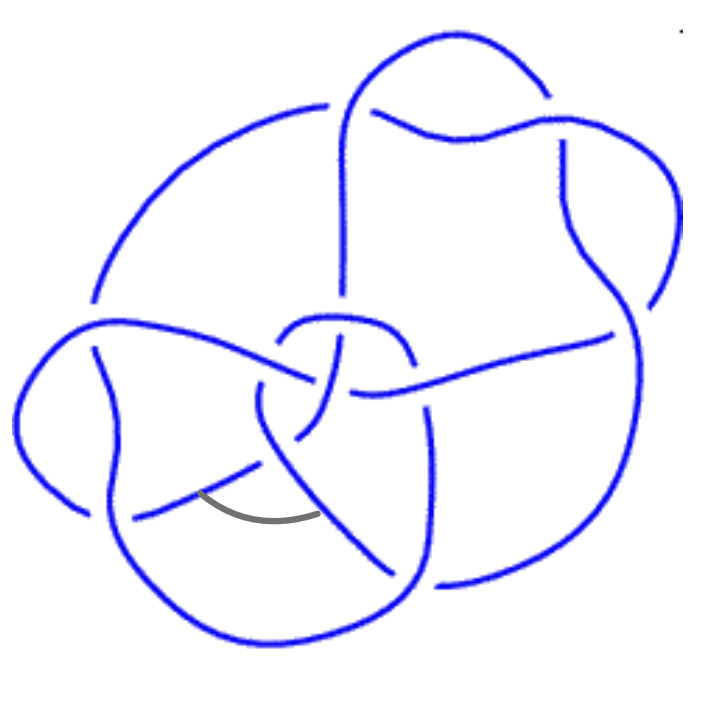}
		\caption{$11n_{158}\stackrel{0}{\longrightarrow} 0_{1}$}
		
	\end{subfigure}
	~
	\begin{subfigure}[b]{0.25\textwidth}
		\includegraphics[width=\textwidth]{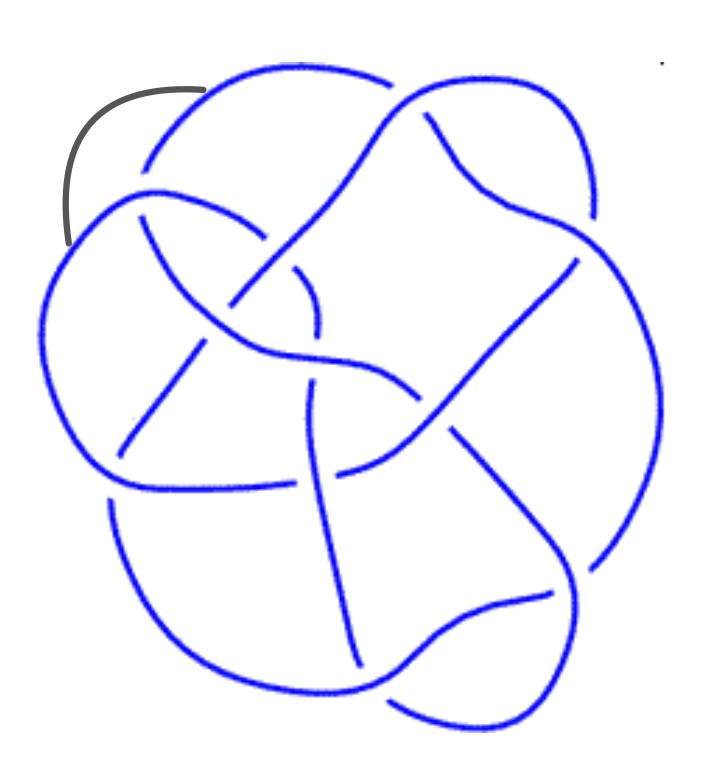}
		\caption{$11n_{160}\stackrel{1}{\longrightarrow} 12n_{802}$}
		
	\end{subfigure}
 \vskip3mm
 ~
	\begin{subfigure}[b]{0.25\textwidth}
		\includegraphics[width=\textwidth]{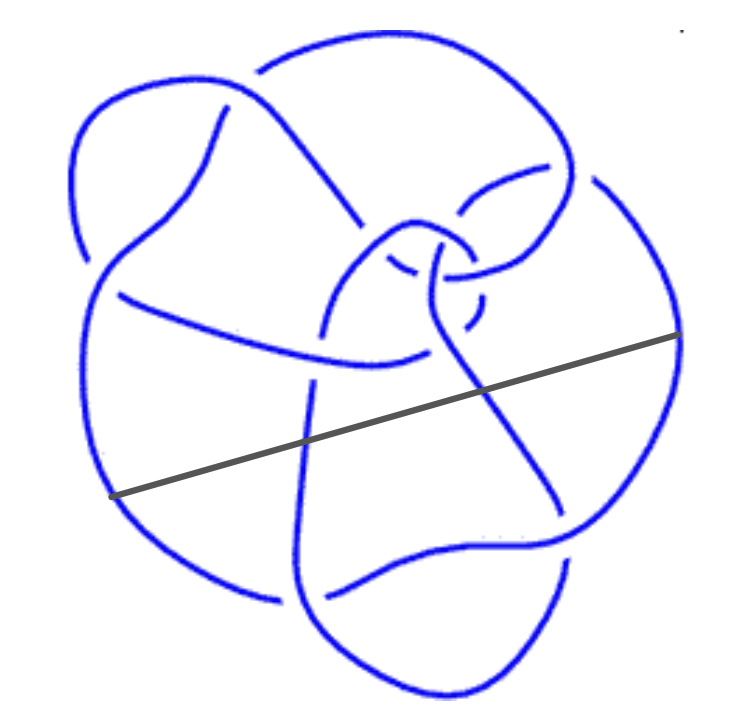}
		\caption{$11n_{162}\stackrel{-1}{\longrightarrow} 10_{140}$}
		
	\end{subfigure}
 ~
	\begin{subfigure}[b]{0.25\textwidth}
		\includegraphics[width=\textwidth]{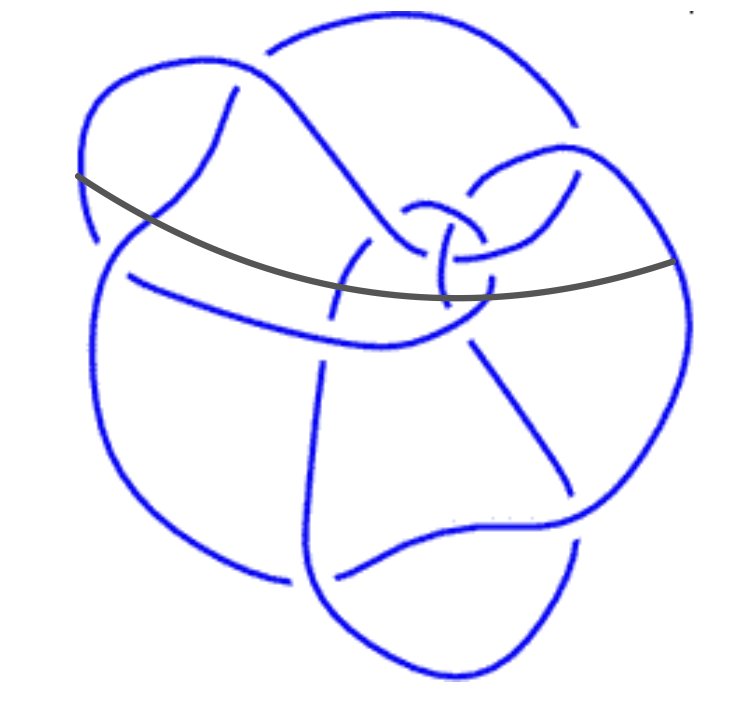}
		\caption{$11n_{163}\stackrel{0}{\longrightarrow} 8_8$}
		
	\end{subfigure}
 ~
	\begin{subfigure}[b]{0.25\textwidth}
		\includegraphics[width=\textwidth]{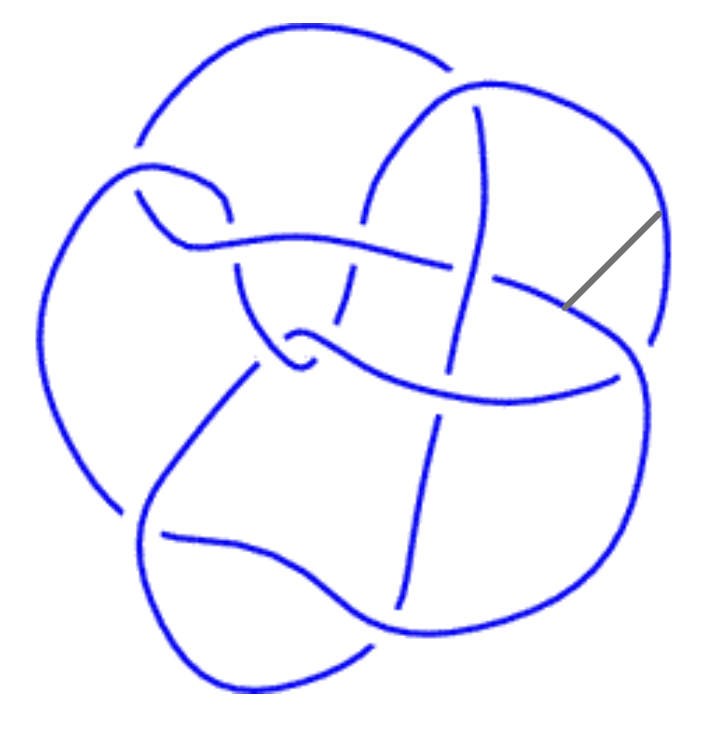}
		\caption{$11n_{164}\stackrel{0}{\longrightarrow} 8_{20}$}
		
	\end{subfigure}
	\vskip3mm
	\caption{Non-oriented band moves from the knots $11n_{148},  11n_{150},  11n_{151},    $ \\ $ 11n_{152},  11n_{153}, 11n_{154}, 11n_{157}, 11n_{158}, 11n_{160}, 11n_{162}, 11n_{163}, \text{ and } 11n_{164} $ \\ to smoothly slice knots.}
\end{figure}
%



\begin{figure}[!htbp]
	\centering
	\begin{subfigure}[b]{0.27\textwidth}
		\includegraphics[width=\textwidth]{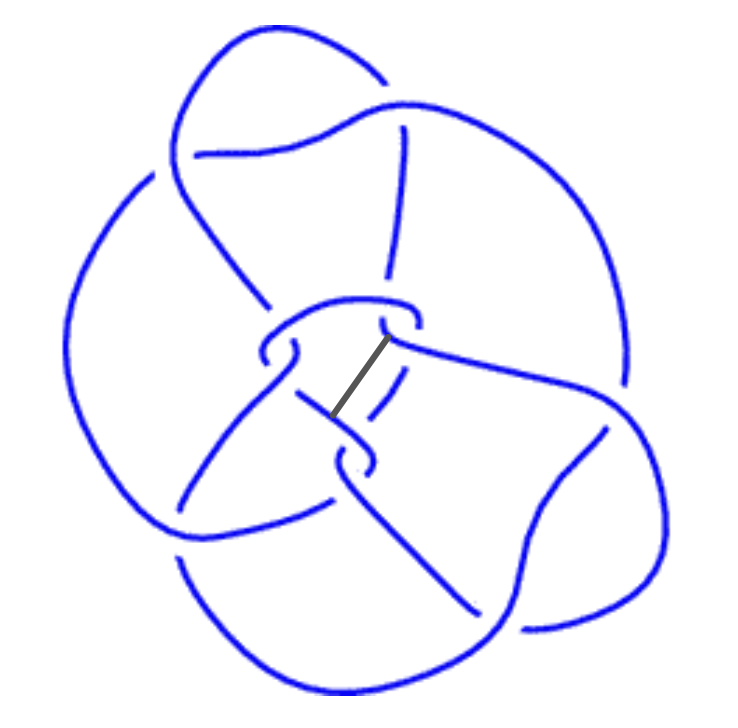}
		\caption{$11n_{167}\stackrel{0}{\longrightarrow} 6_1$}
		
	\end{subfigure}
	~
	\begin{subfigure}[b]{0.25\textwidth}
		\includegraphics[width=\textwidth]{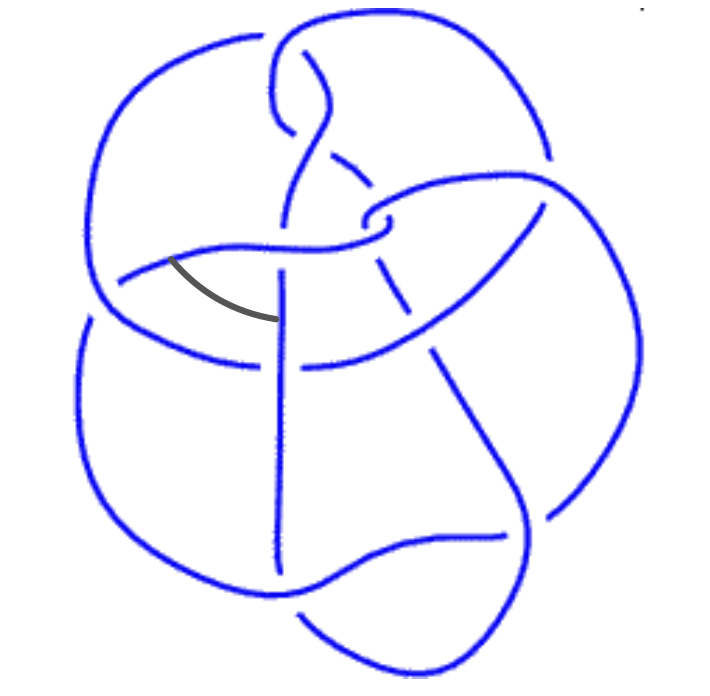}
		\caption{$11n_{168}\stackrel{-1}{\longrightarrow} 10_{137}$}
		
	\end{subfigure}
	~
	\begin{subfigure}[b]{0.25\textwidth}
		\includegraphics[width=\textwidth]{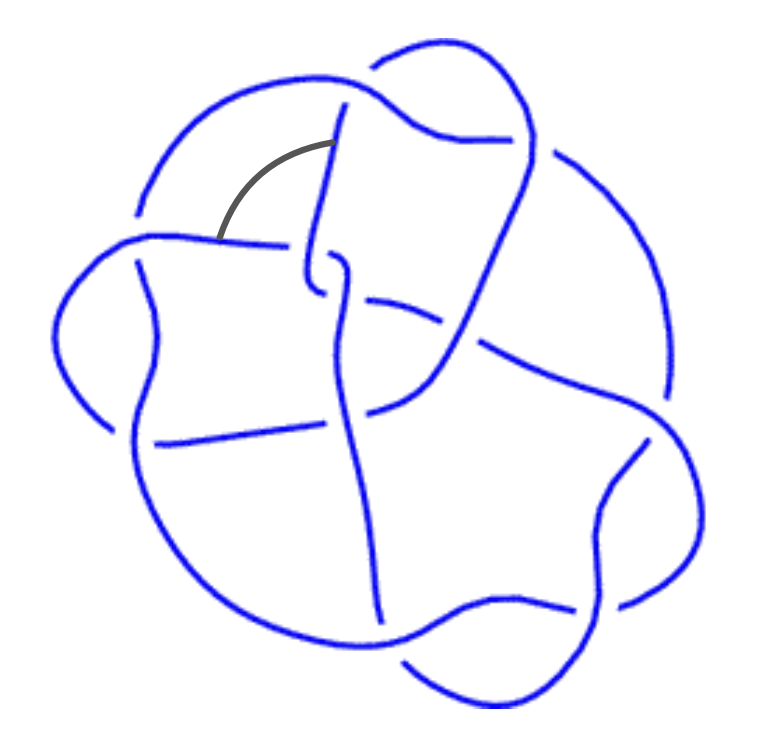}
		\caption{$11n_{169}\stackrel{-1\phantom{i}}{\longrightarrow} 12n_{817}$}
		
	\end{subfigure}
	\vskip3mm
	\begin{subfigure}[b]{0.25\textwidth}
		\includegraphics[width=\textwidth]{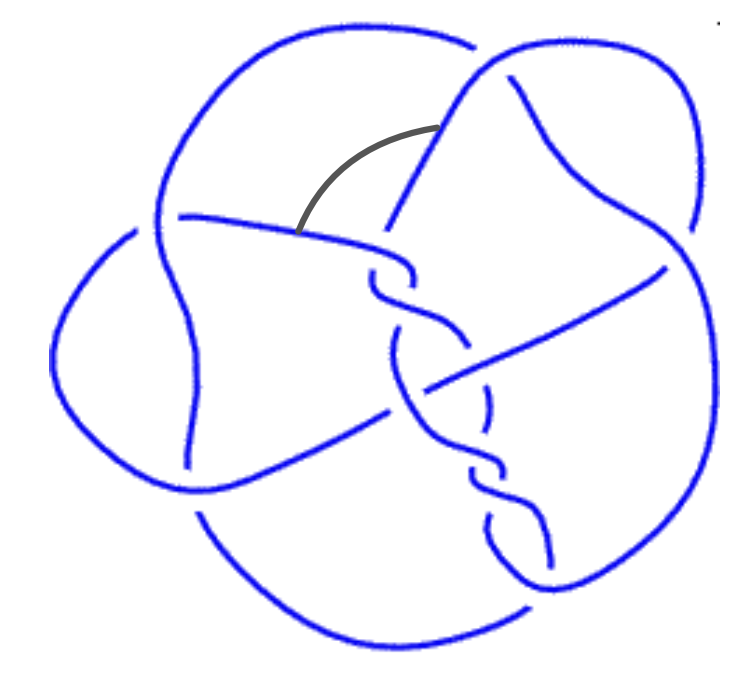}
		\caption{$11n_{170}\stackrel{1}{\longrightarrow} 12n_{876}$}
		
	\end{subfigure}
	~
	\begin{subfigure}[b]{0.25\textwidth}
		\includegraphics[width=\textwidth]{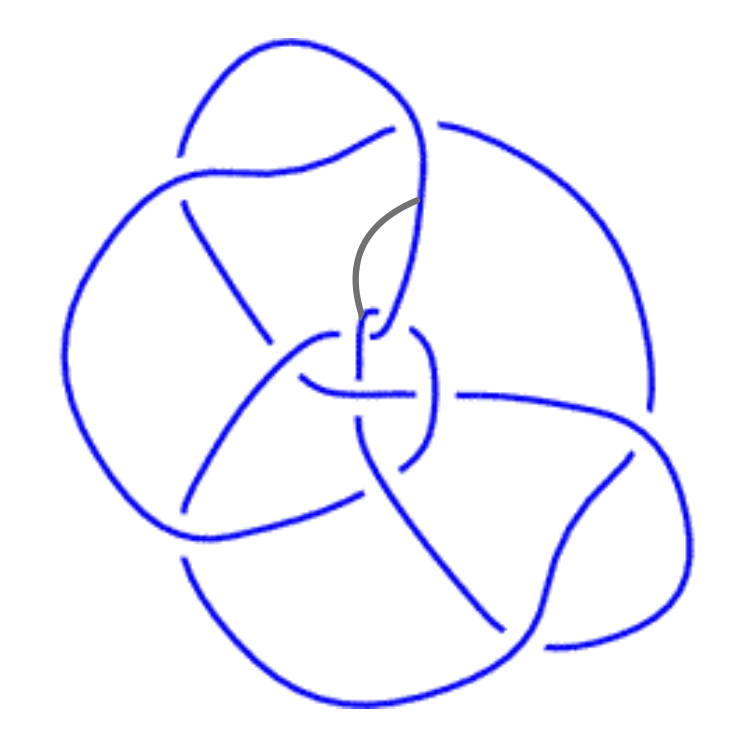}
		\caption{$11n_{173}\stackrel{1}{\longrightarrow} 9_{46}$}
		
	\end{subfigure}
	~
	\begin{subfigure}[b]{0.25\textwidth}
		\includegraphics[width=\textwidth]{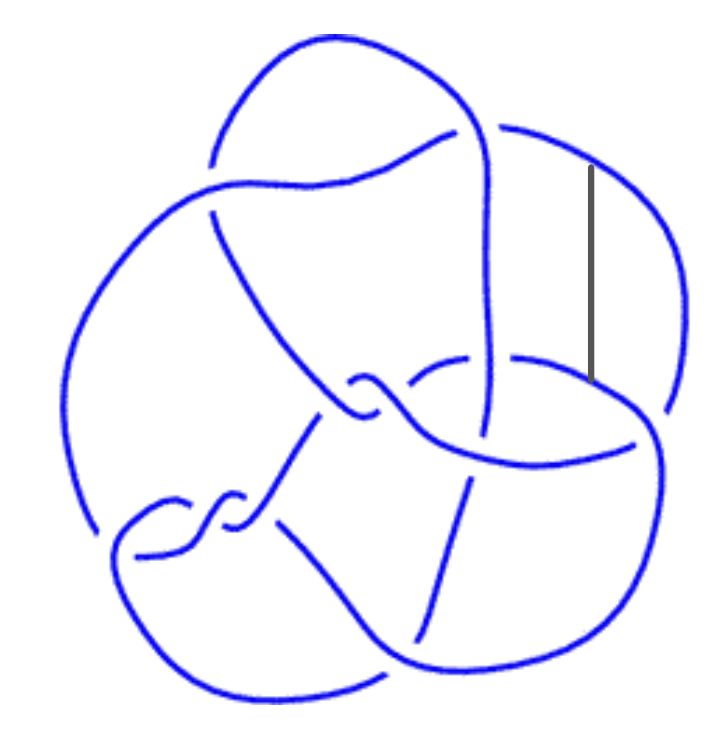}
		\caption{$11n_{180}\stackrel{0}{\longrightarrow} 6_{1}$}
		
	\end{subfigure}
	\vskip3mm
	\begin{subfigure}[b]{0.25\textwidth}
		\includegraphics[width=\textwidth]{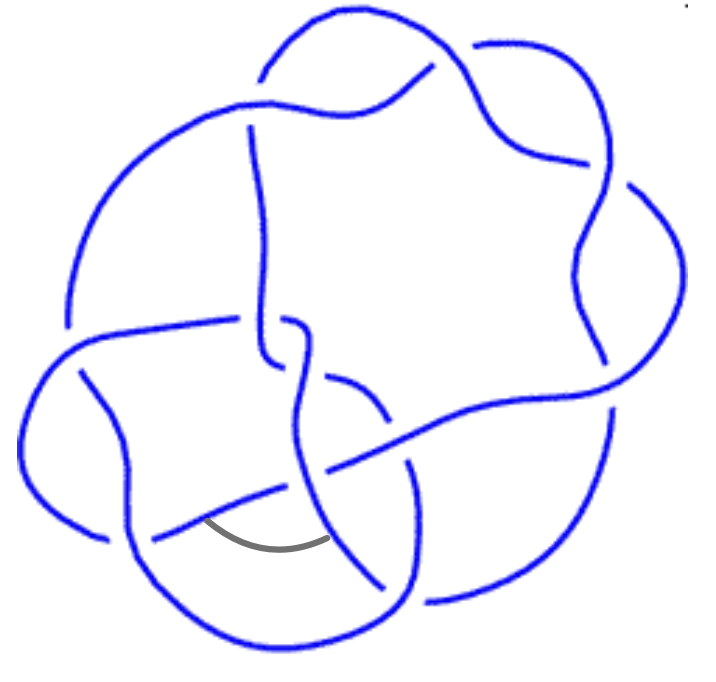}
		\caption{$11n_{181}\stackrel{0}{\longrightarrow} 6_{1}$}
		
	\end{subfigure}
	~
	\begin{subfigure}[b]{0.25\textwidth}
		\includegraphics[width=\textwidth]{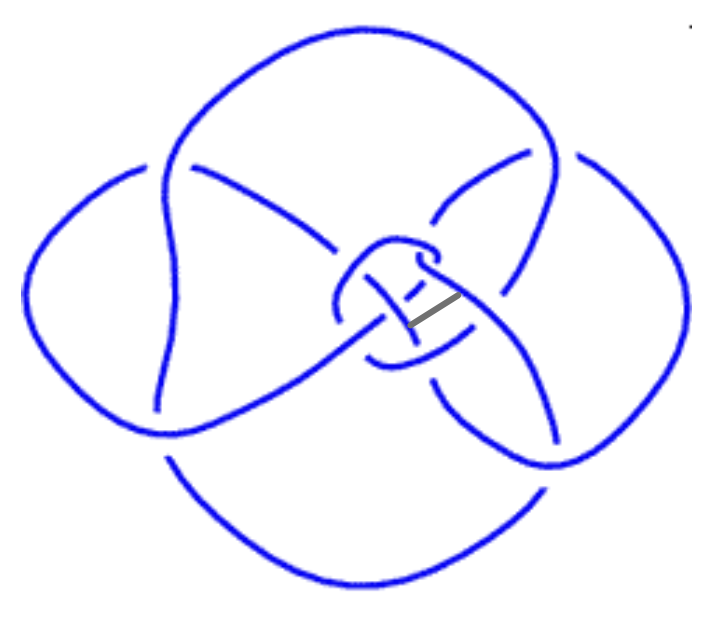}
		\caption{$11n_{183}\stackrel{0}{\longrightarrow} 0_{1}$}
		
	\end{subfigure}

 \vskip3mm
	\caption{Non-oriented band moves from the knots $11n_{167},  11n_{168},  11n_{169},    $ \\ $ 11n_{170},  11n_{173}, 11n_{180}, 11n_{181}, \text{ and } 11n_{183} $ to smoothly slice knots.}\label{lastSlice}
\end{figure}
\clearpage


\newpage

\begin{figure}[!htbp]
	\centering
	\begin{subfigure}[b]{0.27\textwidth}
		\includegraphics[width=\textwidth]{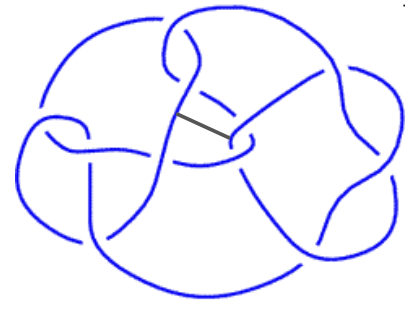}
		\caption{$11n_{10}\stackrel{0}{\longrightarrow} 7_6$}
		
	\end{subfigure}
	~
	\begin{subfigure}[b]{0.25\textwidth}
		\includegraphics[width=\textwidth]{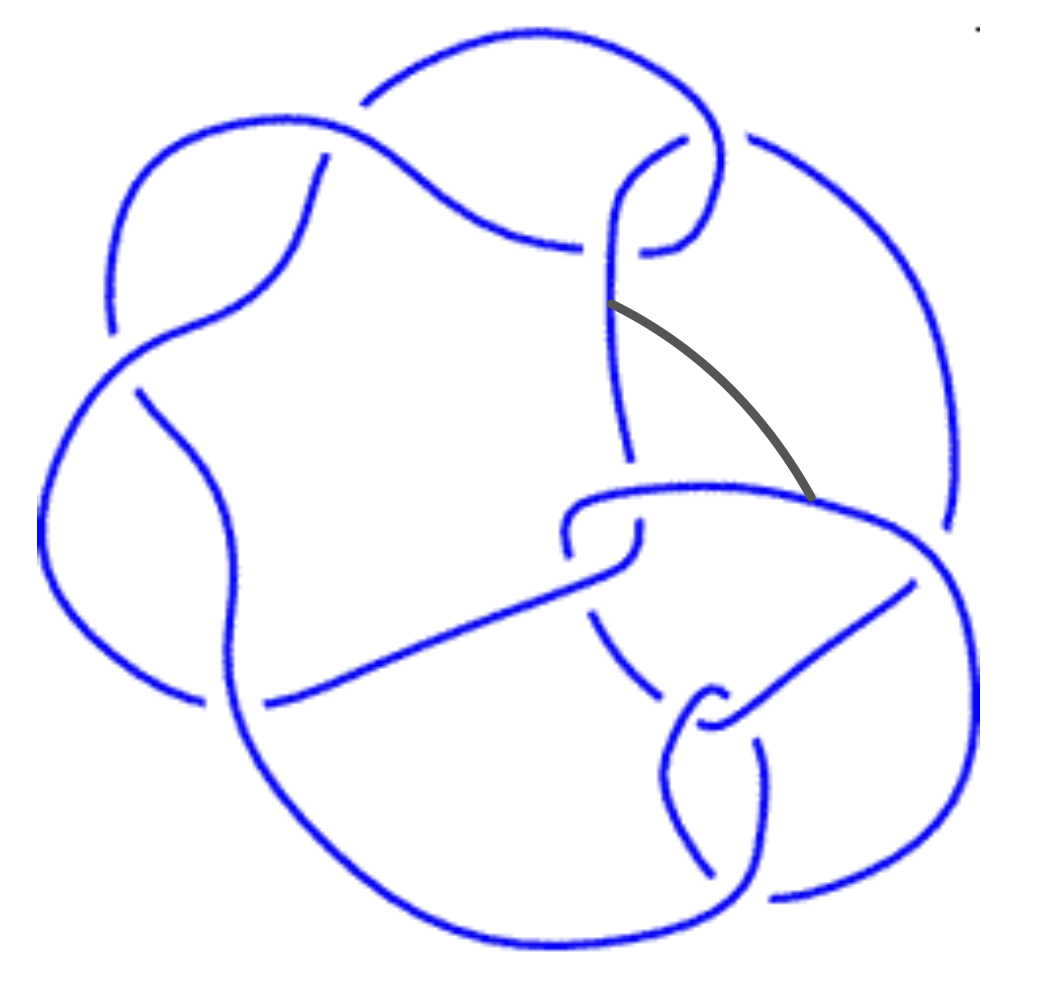}
		\caption{$11n_{12}\stackrel{0}{\longrightarrow} 6_{2}$}
		
	\end{subfigure}
	~
	\begin{subfigure}[b]{0.25\textwidth}
		\includegraphics[width=\textwidth]{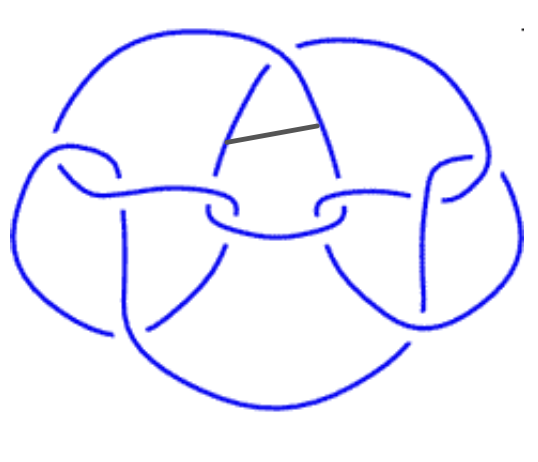}
		\caption{$11n_{22}\stackrel{-1\phantom{i}}{\longrightarrow} 5_{2}$}
		
	\end{subfigure}
	\vskip3mm
	\begin{subfigure}[b]{0.25\textwidth}
		\includegraphics[width=\textwidth]{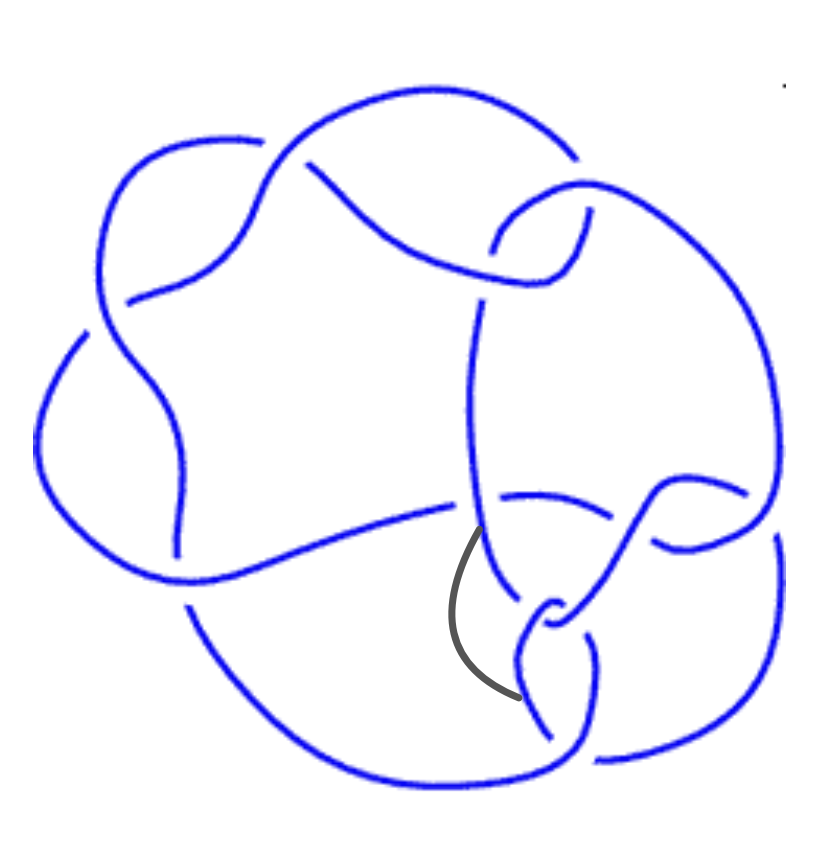}
		\caption{$11n_{29}\stackrel{0}{\longrightarrow} 8_{6}$}
		
	\end{subfigure}
	~
	\begin{subfigure}[b]{0.25\textwidth}
		\includegraphics[width=\textwidth]{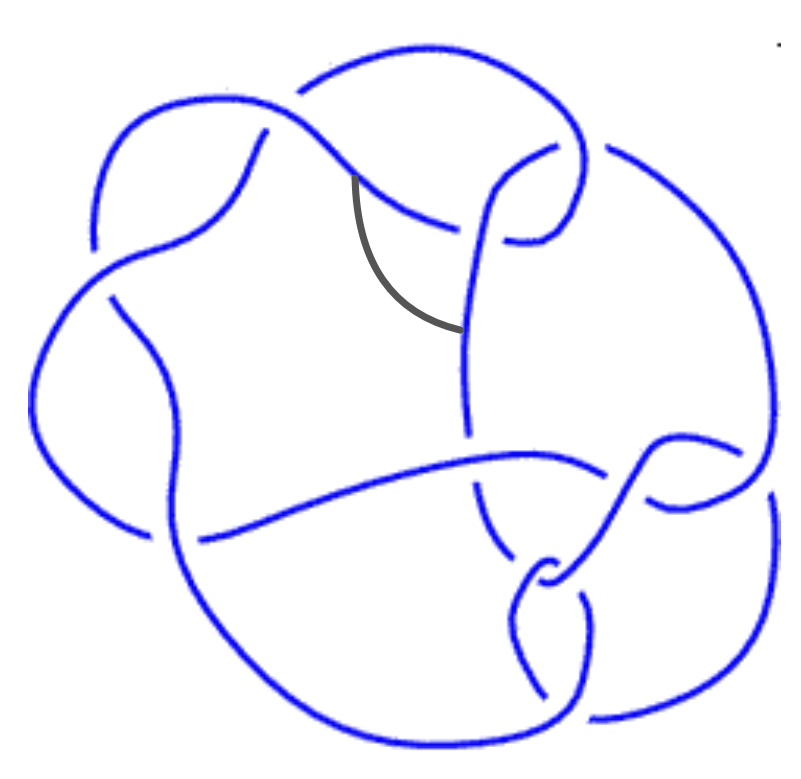}
		\caption{$11n_{30}\stackrel{-1}{\longrightarrow} 10_{126}$}
		
	\end{subfigure}
	~
	\begin{subfigure}[b]{0.25\textwidth}
		\includegraphics[width=\textwidth]{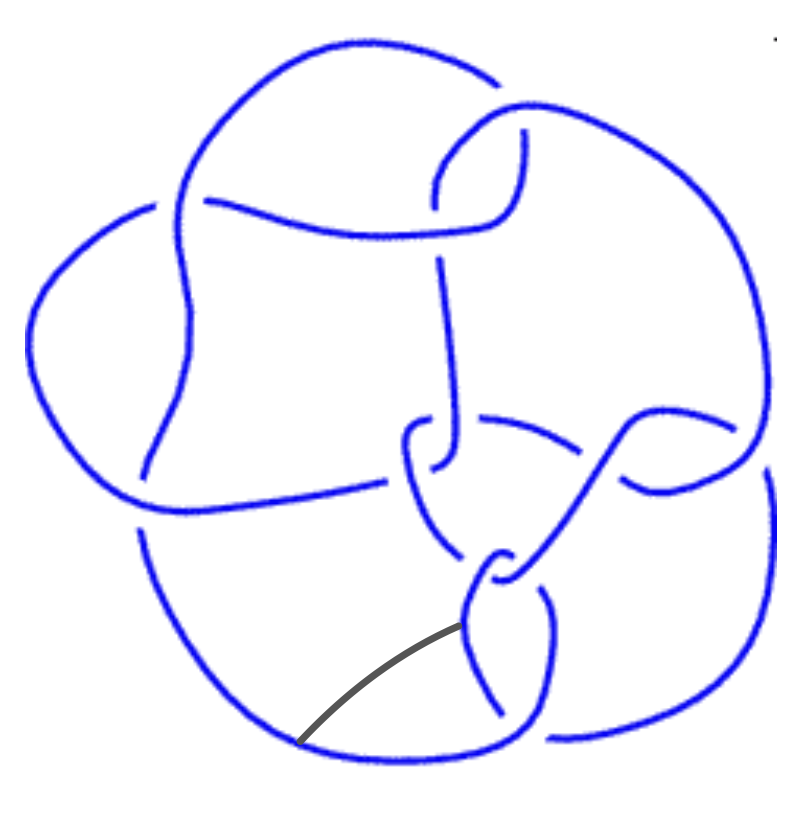}
		\caption{$11n_{32}\stackrel{0}{\longrightarrow} 9_{25}$}
		
	\end{subfigure}
	\vskip3mm
	\begin{subfigure}[b]{0.25\textwidth}
		\includegraphics[width=\textwidth]{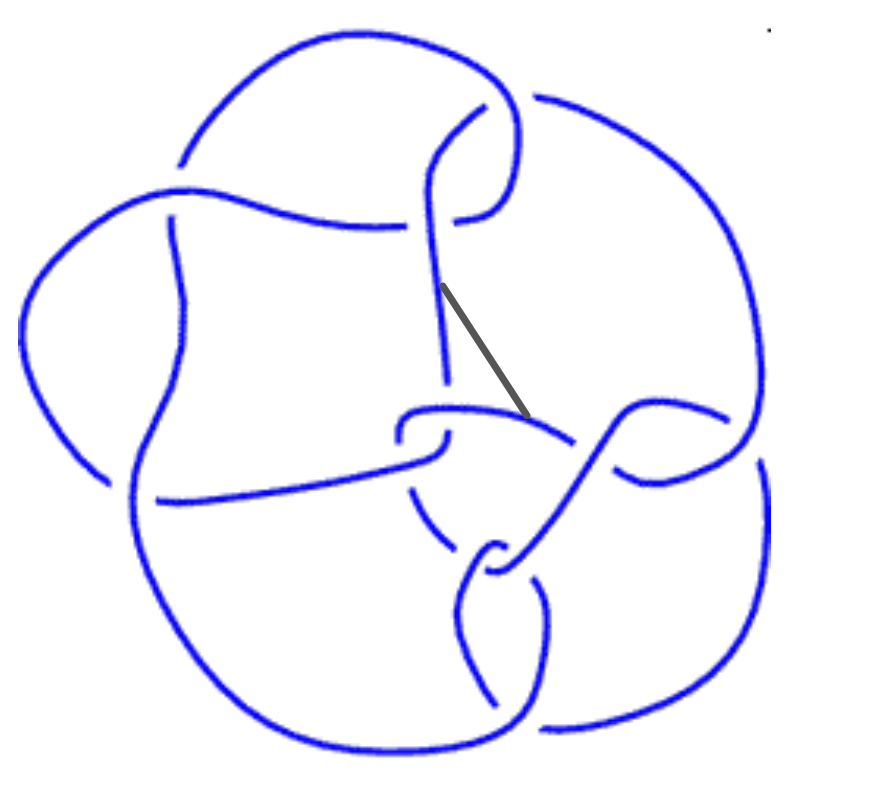}
		\caption{$11n_{33}\stackrel{1}{\longrightarrow} 10_{134}$}
		
	\end{subfigure}
	~
	\begin{subfigure}[b]{0.25\textwidth}
		\includegraphics[width=\textwidth]{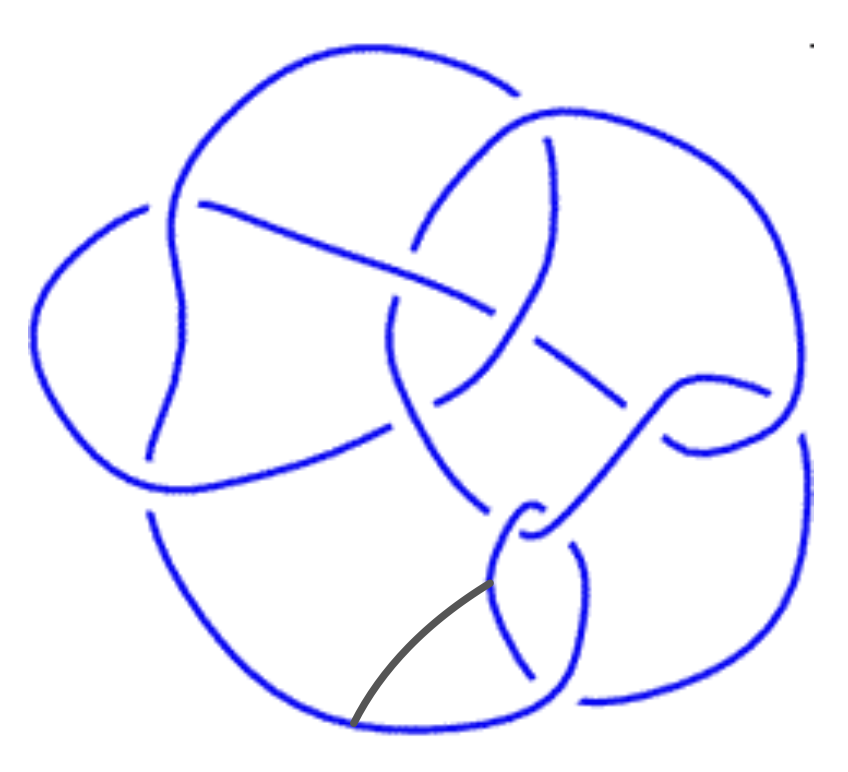}
		\caption{$11n_{43}\stackrel{0}{\longrightarrow} 9_{32}$}
		
	\end{subfigure}
	~
	\begin{subfigure}[b]{0.25\textwidth}
		\includegraphics[width=\textwidth]{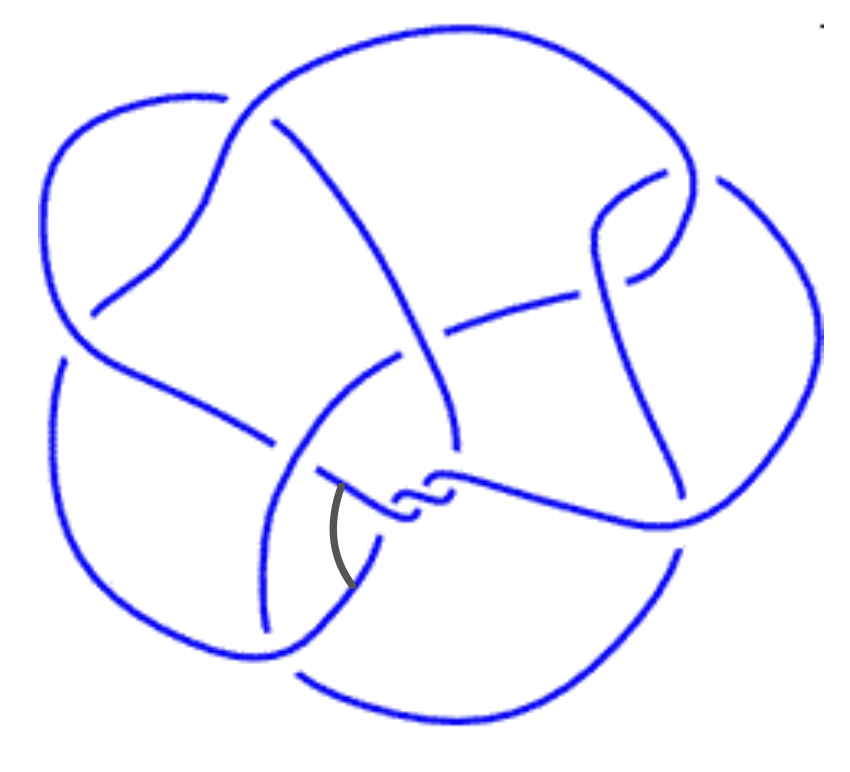}
		\caption{$11n_{48}\stackrel{0}{\longrightarrow} 7_2$}
		
	\end{subfigure}
 \vskip3mm
 ~
	\begin{subfigure}[b]{0.25\textwidth}
		\includegraphics[width=\textwidth]{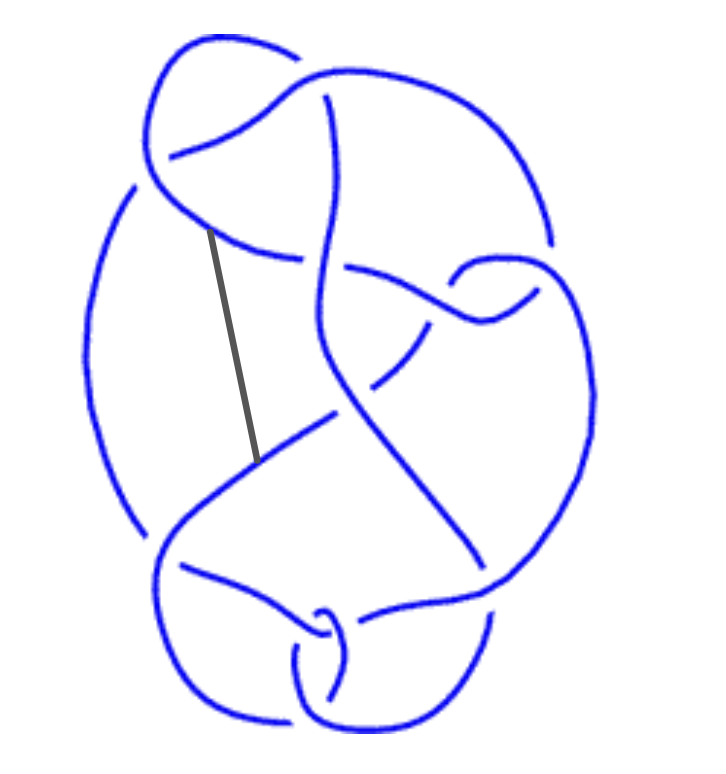}
		\caption{$11n_{51}\stackrel{1}{\longrightarrow} 9_8$}
		
	\end{subfigure}
 ~
	\begin{subfigure}[b]{0.25\textwidth}
		\includegraphics[width=\textwidth]{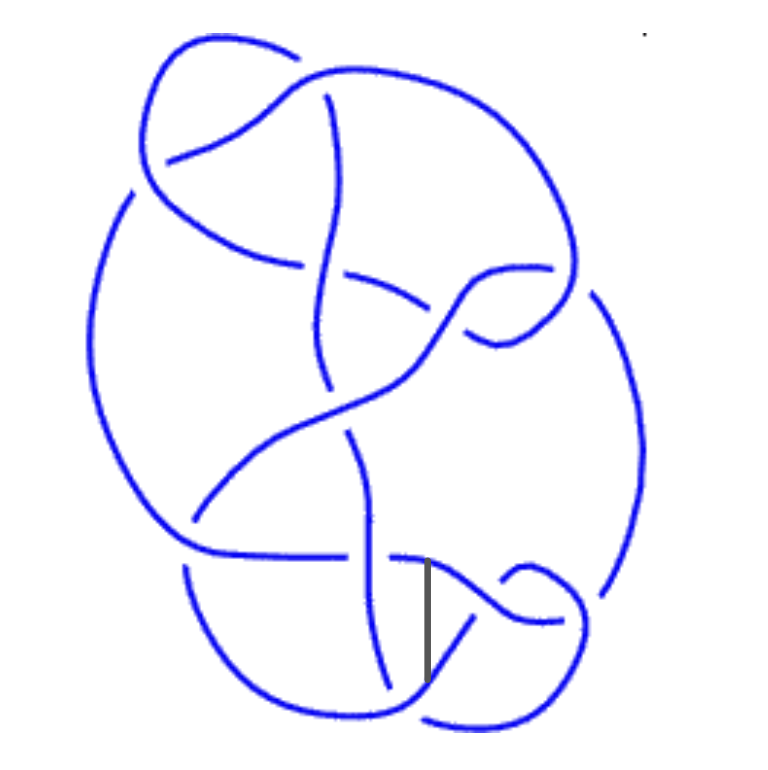}
		\caption{$11n_{55}\stackrel{0}{\longrightarrow} 9_{45}$}
		
	\end{subfigure}
 ~
	\begin{subfigure}[b]{0.25\textwidth}
		\includegraphics[width=\textwidth]{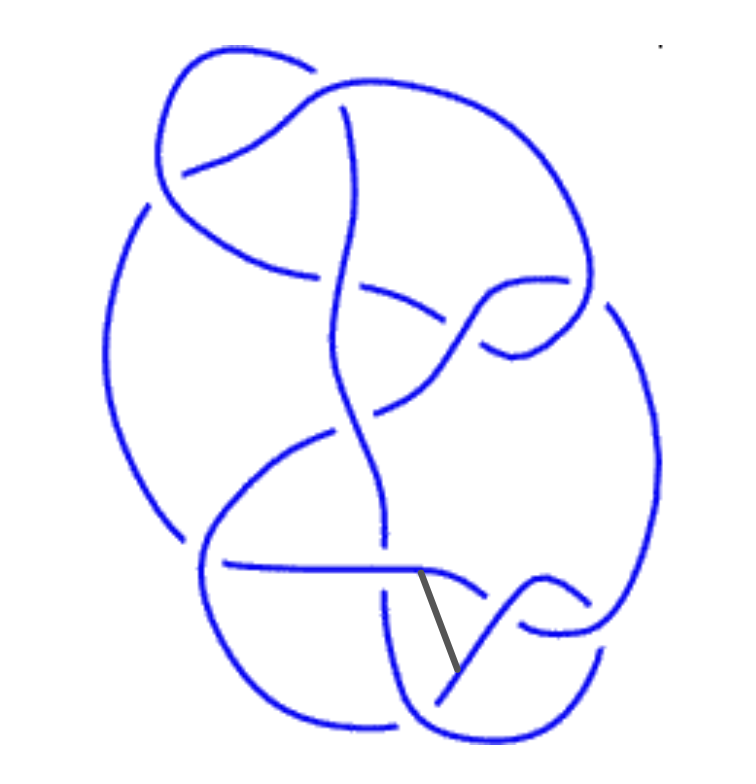}
		\caption{$11n_{56}\stackrel{0}{\longrightarrow} 9_{43}$}
		
	\end{subfigure}
	\vskip3mm
	\caption{Non-oriented band moves from the knots $11n_{10},  11n_{12},  11n_{22}, $ \\ $ 11n_{29}, 11n_{30}, 11n_{32}, 11n_{33}, 11n_{43}, 11n_{48}, 11n_{51}, 11n_{55}, \text{ and } 11n_{56} $ to knots with non- orientable genus 1.}\label{first1G}
\end{figure}
%



\newpage

\begin{figure}[!htbp]
	\centering
	\begin{subfigure}[b]{0.27\textwidth}
		\includegraphics[width=\textwidth]{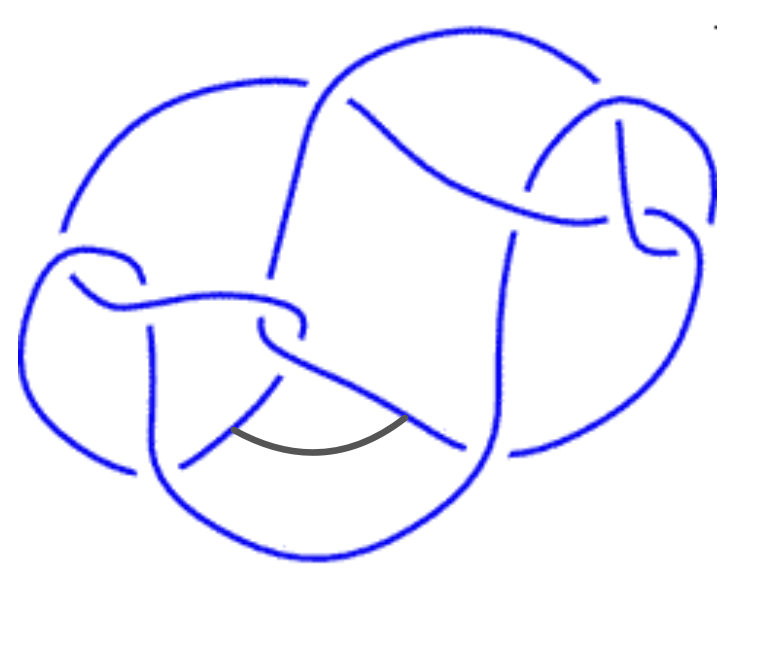}
		\caption{$11n_{61}\stackrel{0}{\longrightarrow} 6_2$}
		
	\end{subfigure}
	~
	\begin{subfigure}[b]{0.25\textwidth}
		\includegraphics[width=\textwidth]{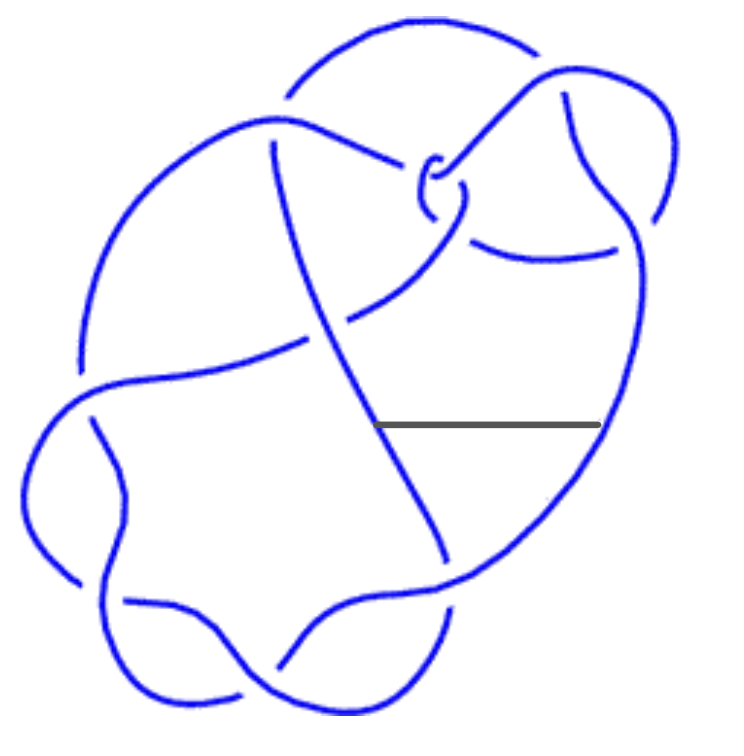}
		\caption{$11n_{63}\stackrel{1}{\longrightarrow} 10_{131}$}
		
	\end{subfigure}
	~
	\begin{subfigure}[b]{0.25\textwidth}
		\includegraphics[width=\textwidth]{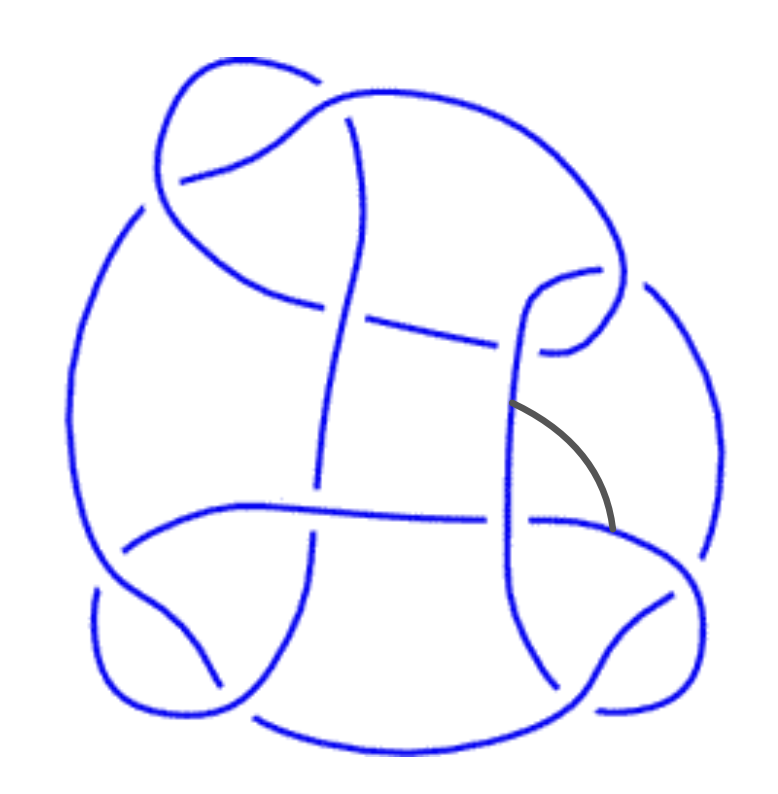}
		\caption{$11n_{72}\stackrel{0\phantom{i}}{\longrightarrow} 9_{28}$}
		
	\end{subfigure}
	\vskip3mm
	\begin{subfigure}[b]{0.25\textwidth}
		\includegraphics[width=\textwidth]{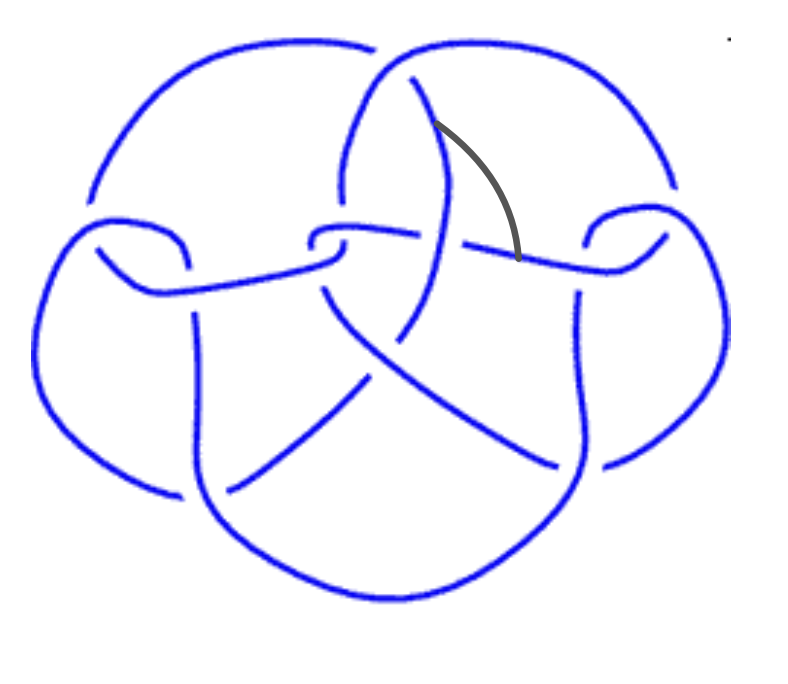}
		\caption{$11n_{84}\stackrel{-1}{\longrightarrow} 9_{44}$}
		
	\end{subfigure}
	~
	\begin{subfigure}[b]{0.25\textwidth}
		\includegraphics[width=\textwidth]{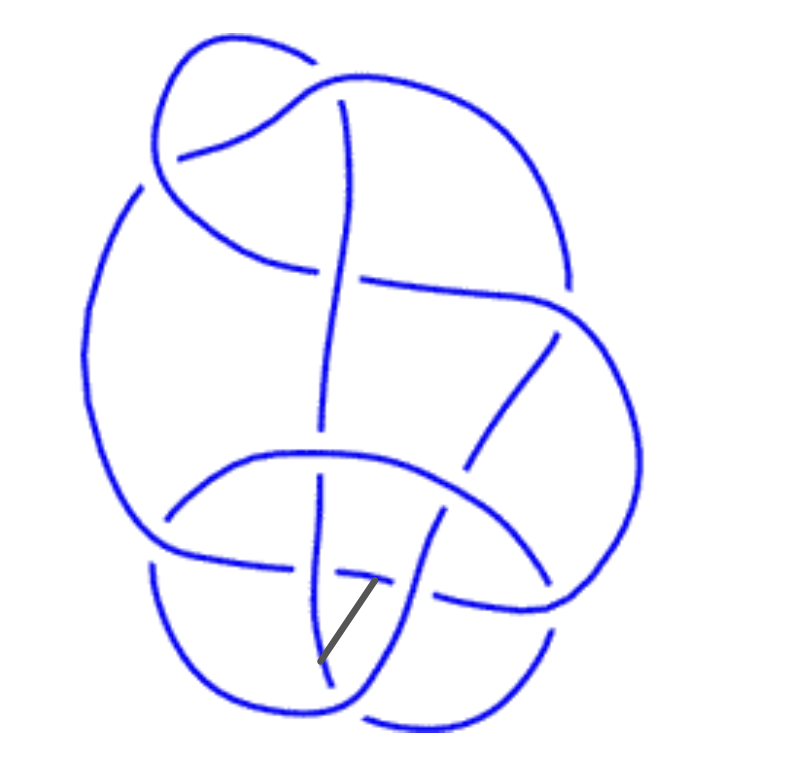}
		\caption{$11n_{85}\stackrel{0}{\longrightarrow} 5_{2}$}
		
	\end{subfigure}
	~
	\begin{subfigure}[b]{0.25\textwidth}
		\includegraphics[width=\textwidth]{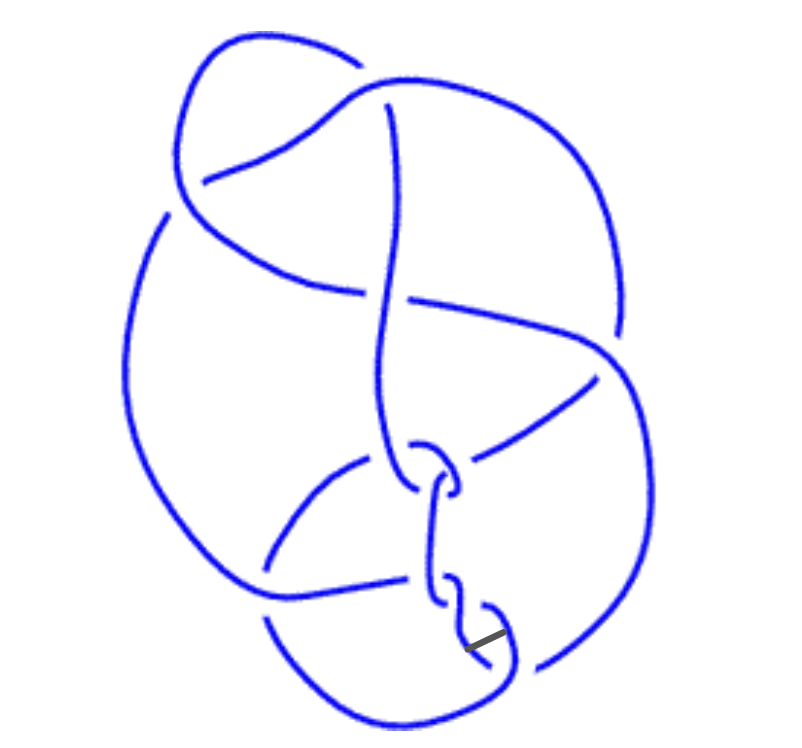}
		\caption{$11n_{90}\stackrel{1}{\longrightarrow} 10_147$}
		
	\end{subfigure}
	\vskip3mm
	\begin{subfigure}[b]{0.25\textwidth}
		\includegraphics[width=\textwidth]{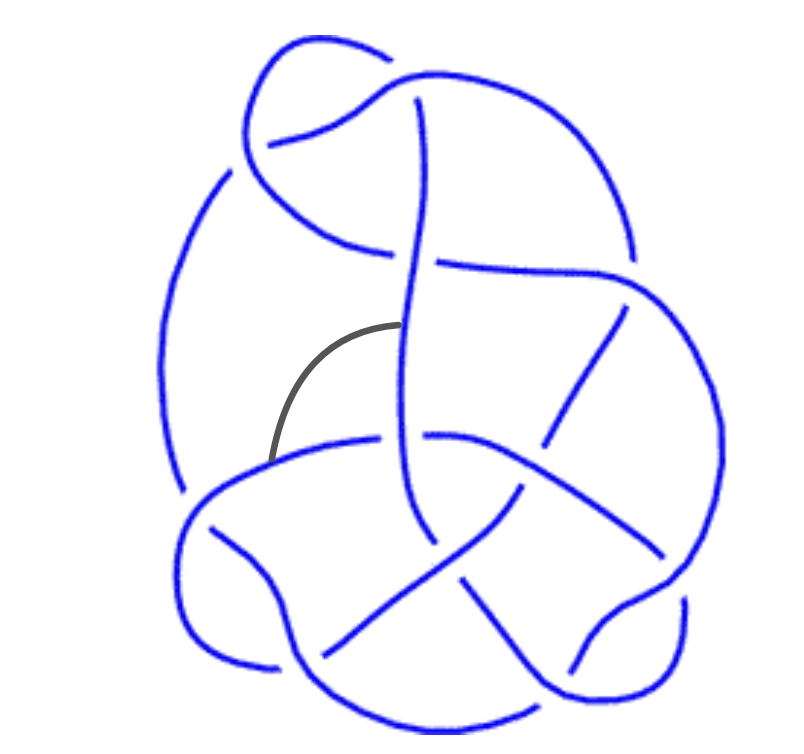}
		\caption{$11n_{92}\stackrel{1}{\longrightarrow} 9_{8}$}
		
	\end{subfigure}
	~
	\begin{subfigure}[b]{0.25\textwidth}
		\includegraphics[width=\textwidth]{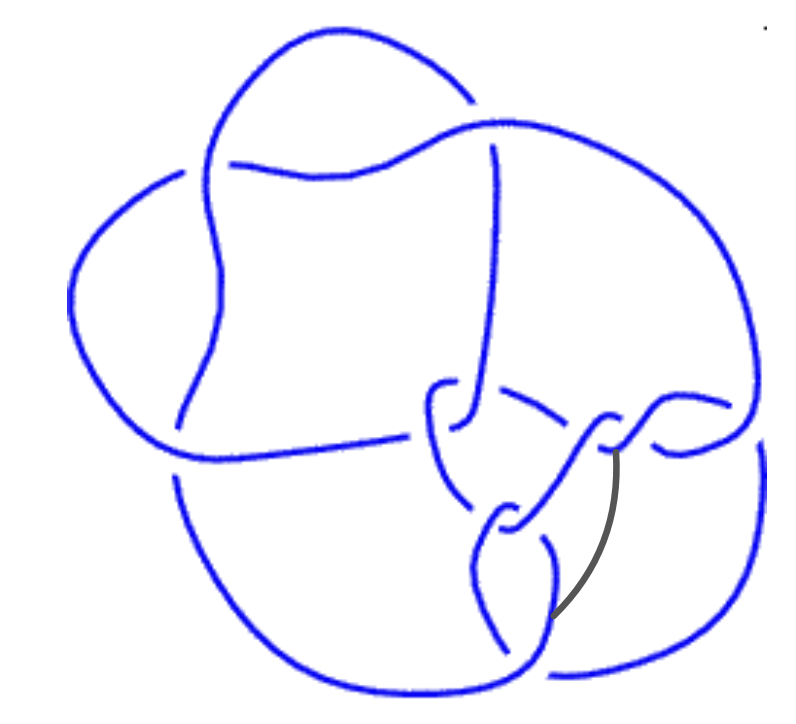}
		\caption{$11n_{98}\stackrel{0}{\longrightarrow} 8_{6}$}
		
	\end{subfigure}
	~
	\begin{subfigure}[b]{0.25\textwidth}
		\includegraphics[width=\textwidth]{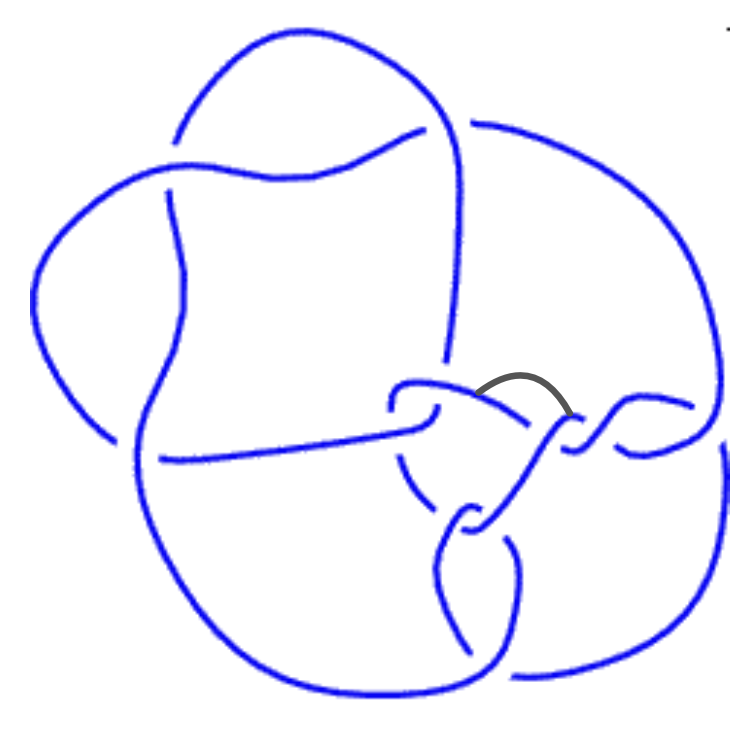}
		\caption{$11n_{99}\stackrel{1}{\longrightarrow} 10_{148}$}
		
	\end{subfigure}
 \vskip3mm
 ~
	\begin{subfigure}[b]{0.25\textwidth}
		\includegraphics[width=\textwidth]{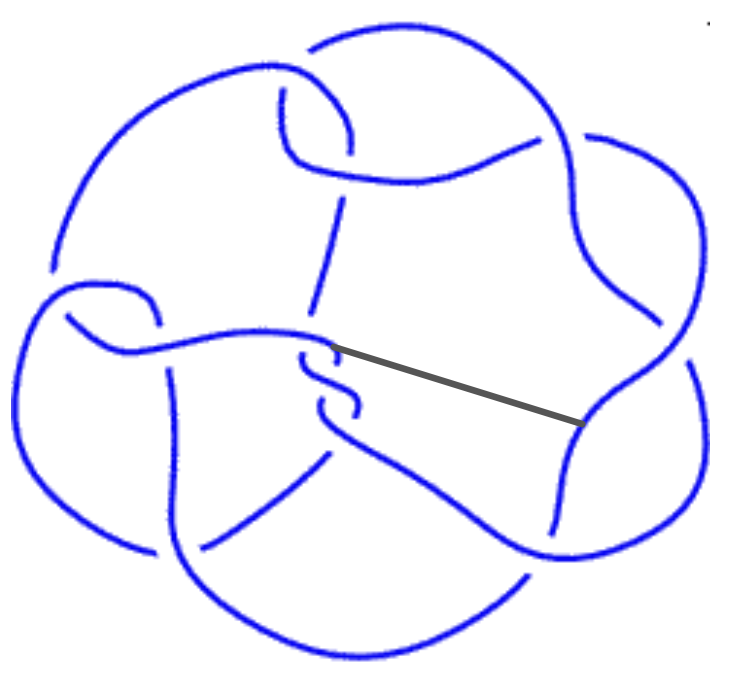}
		\caption{$11n_{101}\stackrel{0}{\longrightarrow} 6_{2}$}
		
	\end{subfigure}
 ~
	\begin{subfigure}[b]{0.25\textwidth}
		\includegraphics[width=\textwidth]{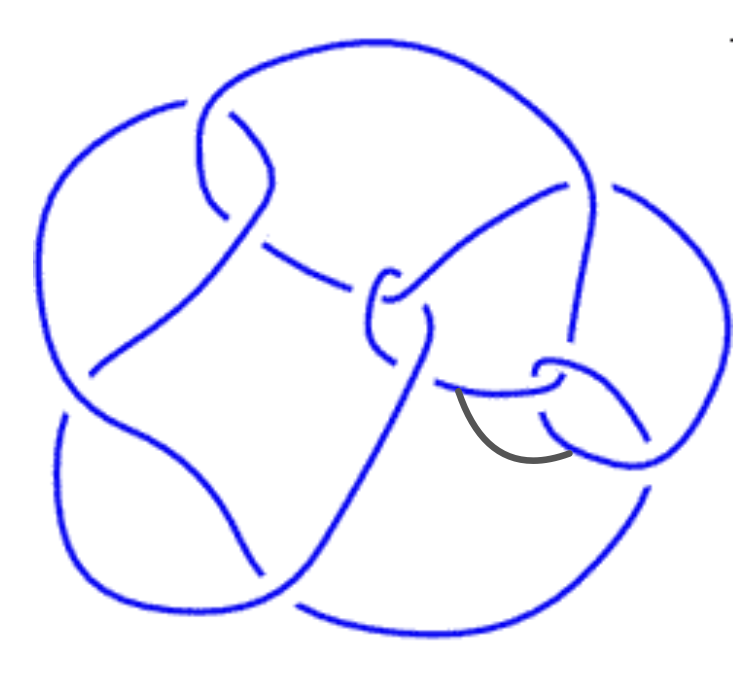}
		\caption{$11n_{103}\stackrel{0}{\longrightarrow} 9_{45}$}
		
	\end{subfigure}
 ~
	\begin{subfigure}[b]{0.25\textwidth}
		\includegraphics[width=\textwidth]{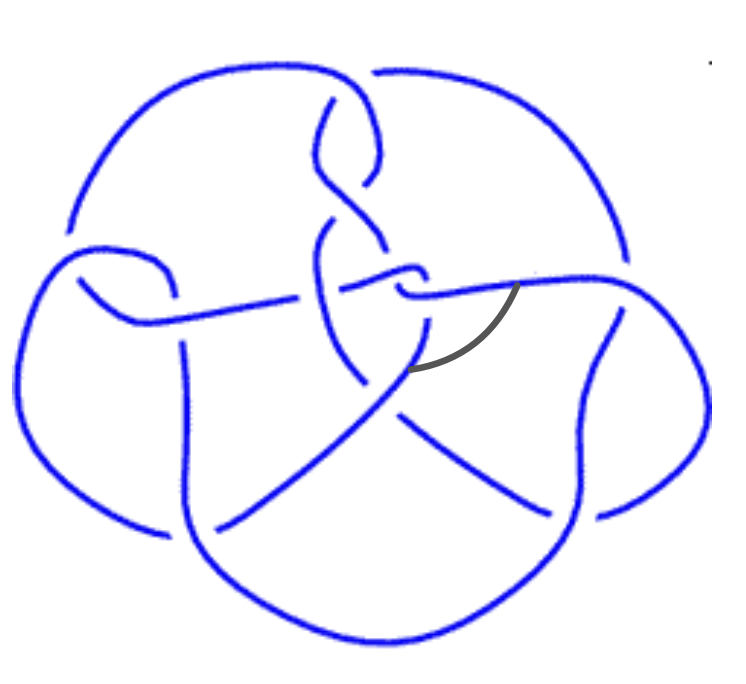}
		\caption{$11n_{112}\stackrel{0}{\longrightarrow} 8_{6}$}
		
	\end{subfigure}
	\vskip3mm
	\caption{Non-oriented band moves from the knots $11n_{61},  11n_{63},  11n_{72}, $ \\ $ 11n_{84}, 11n_{85}, 11n_{90}, 11n_{92}, 11n_{98}, 11n_{99}, 11n_{101}, 11n_{103}, \text{ and } 11n_{112} $ to knots with non- orientable genus 1.}
\end{figure}
%


\newpage

\begin{figure}[!htbp]
	\centering
	\begin{subfigure}[b]{0.27\textwidth}
		\includegraphics[width=\textwidth]{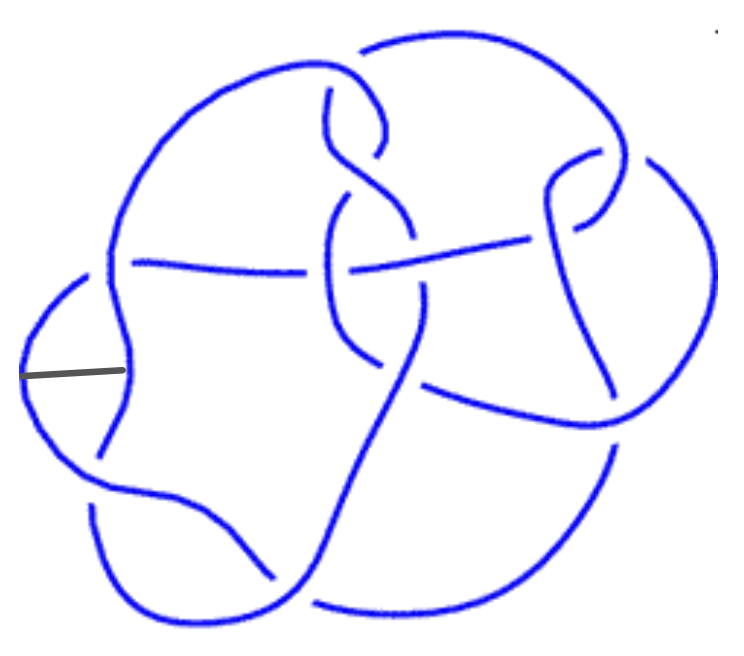}
		\caption{$11n_{125}\stackrel{0}{\longrightarrow} 8_{14}$}
		
	\end{subfigure}
	~
	\begin{subfigure}[b]{0.25\textwidth}
		\includegraphics[width=\textwidth]{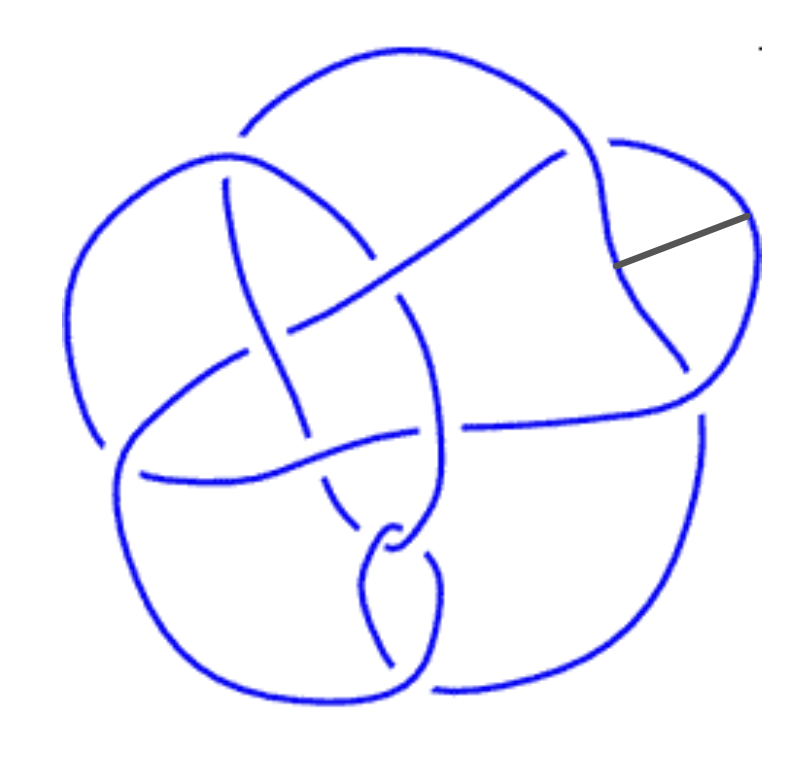}
		\caption{$11n_{130}\stackrel{0}{\longrightarrow} 8_{7}$}
		
	\end{subfigure}
	~
	\begin{subfigure}[b]{0.25\textwidth}
		\includegraphics[width=\textwidth]{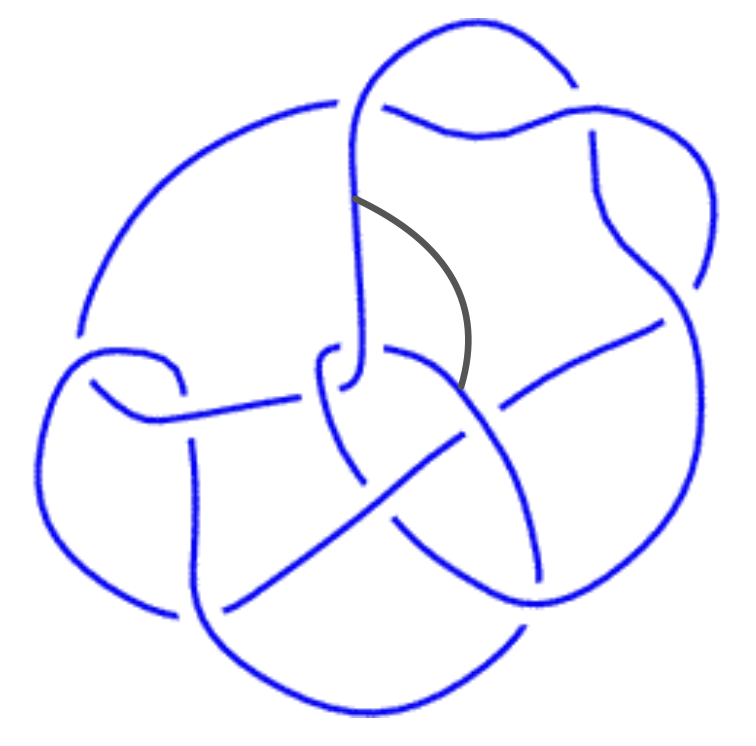}
		\caption{$11n_{131}\stackrel{0\phantom{i}}{\longrightarrow} 8_{14}$}
		
	\end{subfigure}
	\vskip3mm
	\begin{subfigure}[b]{0.25\textwidth}
		\includegraphics[width=\textwidth]{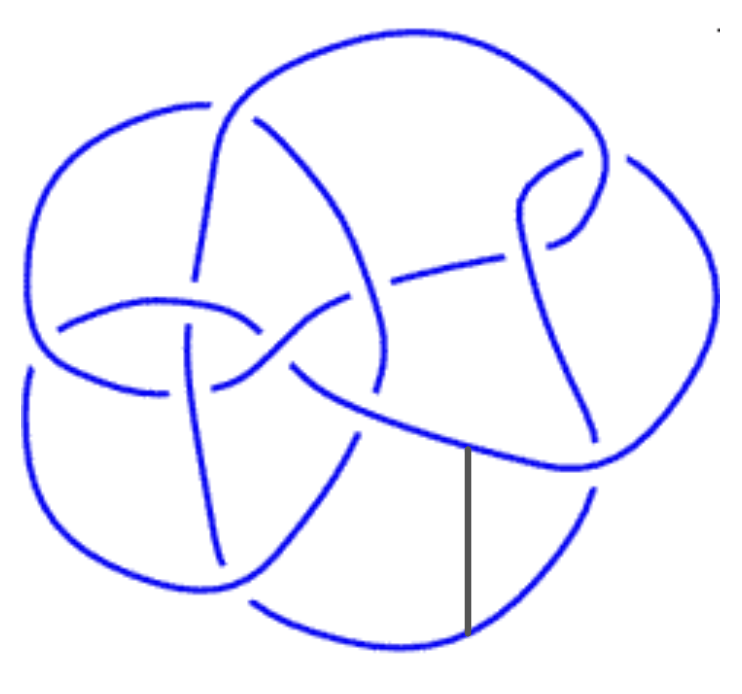}
		\caption{$11n_{133}\stackrel{0}{\longrightarrow} 10_{165}$}
		
	\end{subfigure}
	~
	\begin{subfigure}[b]{0.25\textwidth}
		\includegraphics[width=\textwidth]{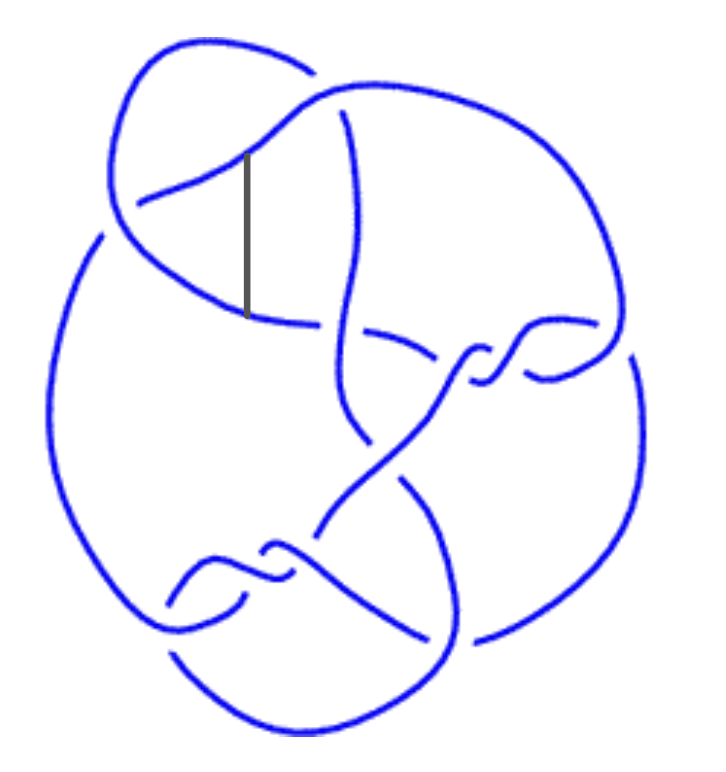}
		\caption{$11n_{137}\stackrel{0}{\longrightarrow} 10_{131}$}
		
	\end{subfigure}
	~
	\begin{subfigure}[b]{0.25\textwidth}
		\includegraphics[width=\textwidth]{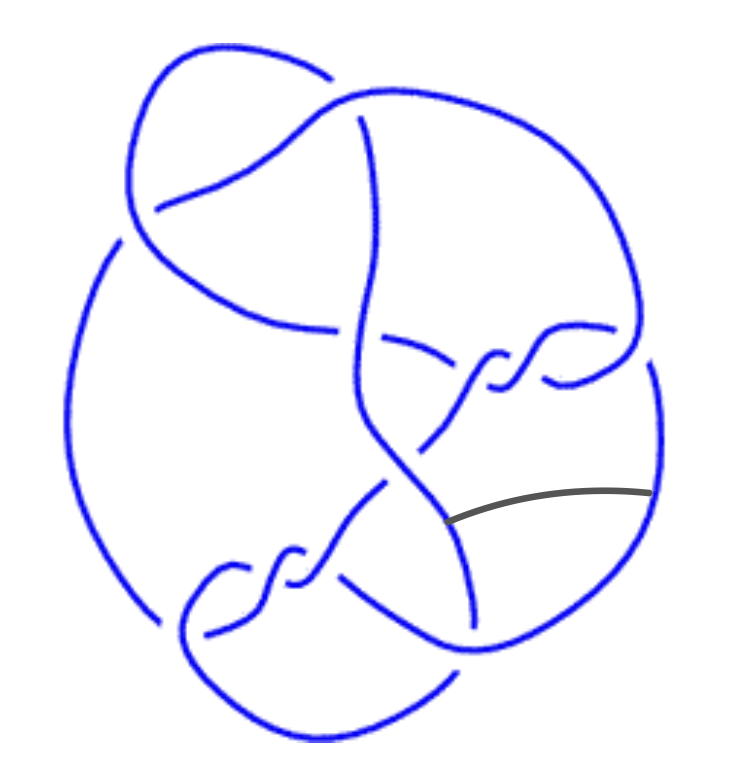}
		\caption{$11n_{138}\stackrel{1}{\longrightarrow} 10_{139}$}
		
	\end{subfigure}
	\vskip3mm
	\begin{subfigure}[b]{0.25\textwidth}
		\includegraphics[width=\textwidth]{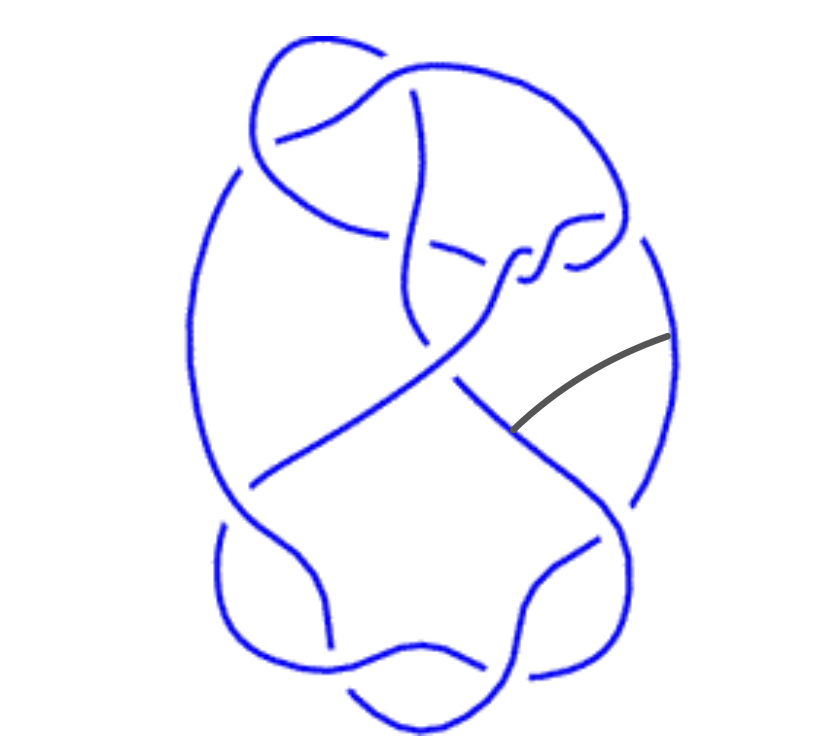}
		\caption{$11n_{140}\stackrel{-1}{\longrightarrow} 10_{144}$}
		
	\end{subfigure}
	~
	\begin{subfigure}[b]{0.25\textwidth}
		\includegraphics[width=\textwidth]{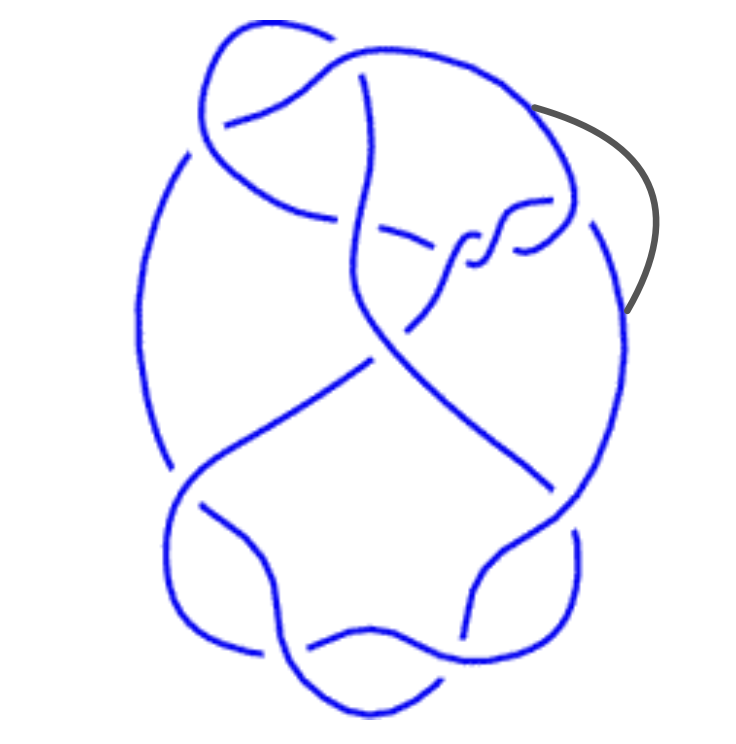}
		\caption{$11n_{141}\stackrel{-1}{\longrightarrow} 10_{126}$}
		
	\end{subfigure}
	~
	\begin{subfigure}[b]{0.25\textwidth}
		\includegraphics[width=\textwidth]{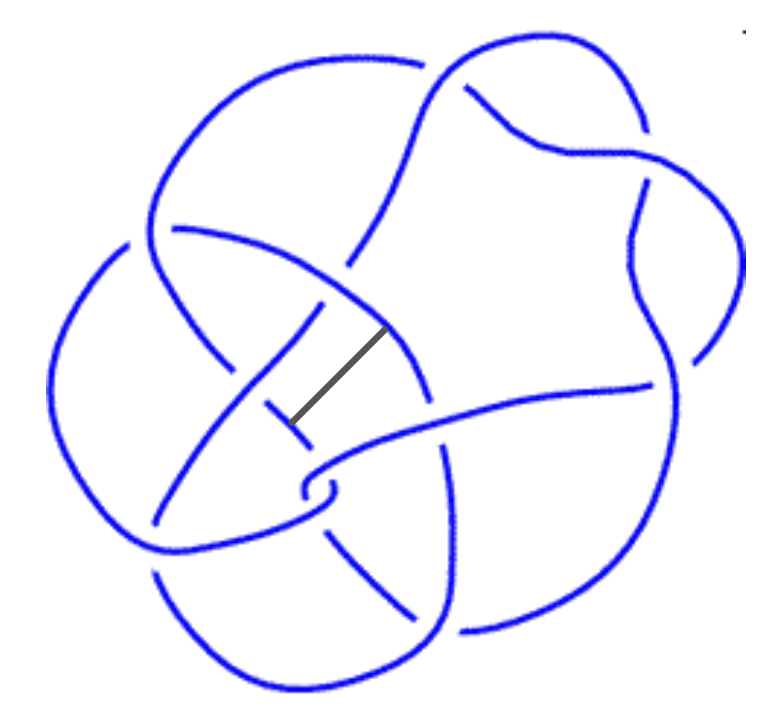}
		\caption{$11n_{155}\stackrel{0}{\longrightarrow} 3_1$}
		
	\end{subfigure}
 \vskip3mm
 ~
	\begin{subfigure}[b]{0.25\textwidth}
		\includegraphics[width=\textwidth]{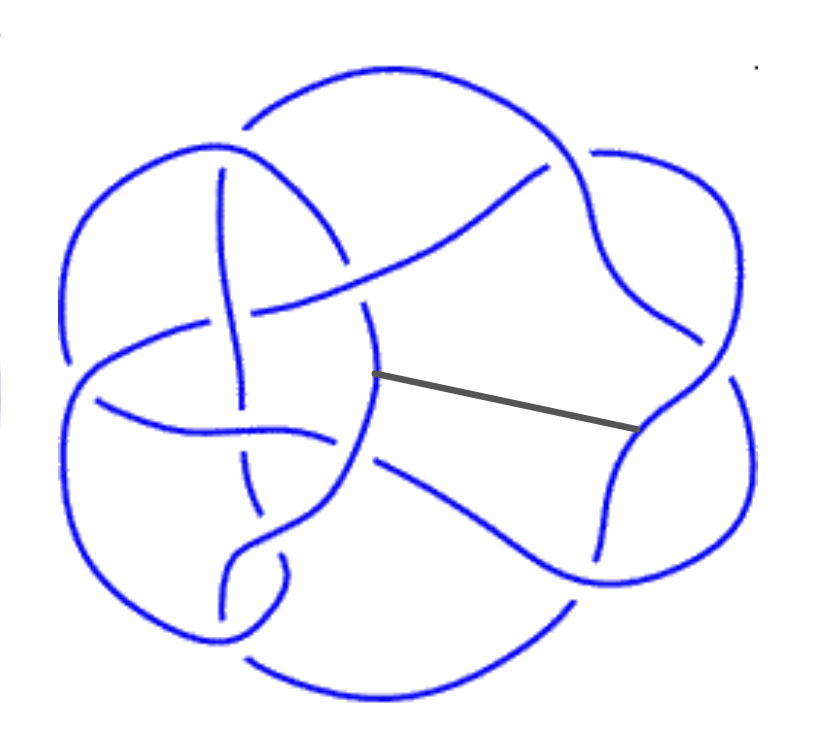}
		\caption{$11n_{161}\stackrel{-1}{\longrightarrow} 6_2$}
		
	\end{subfigure}
 ~
	\begin{subfigure}[b]{0.25\textwidth}
		\includegraphics[width=\textwidth]{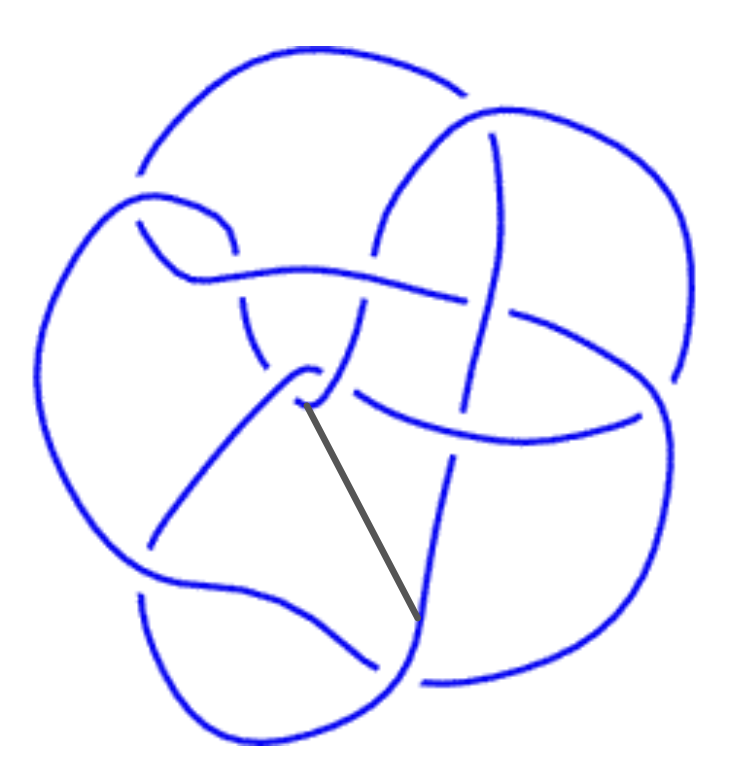}
		\caption{$11n_{165}\stackrel{0}{\longrightarrow} 11n_{46}$}
		
	\end{subfigure}
 ~
	\begin{subfigure}[b]{0.25\textwidth}
		\includegraphics[width=\textwidth]{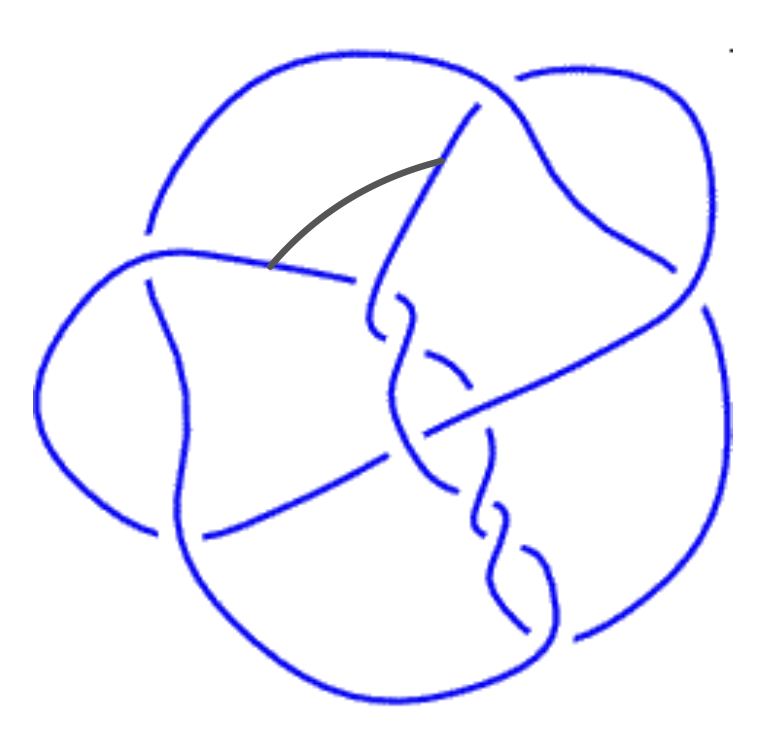}
		\caption{$11n_{171}\stackrel{1}{\longrightarrow} 10_{144}$}
		
	\end{subfigure}
	\vskip3mm
	\caption{Non-oriented band moves from the knots $11n_{125},  11n_{130},  11n_{131}, $ \\ $ 11n_{133}, 11n_{137}, 11n_{138}, 11n_{140}, 11n_{141}, 11n_{155}, 11n_{161}, 11n_{165}, \text{ and } 11n_{171} $ \\ to knots with non- orientable genus 1.}
\end{figure}


\newpage

\begin{figure}[!htbp]
	\centering
	\begin{subfigure}[b]{0.27\textwidth}
		\includegraphics[width=\textwidth]{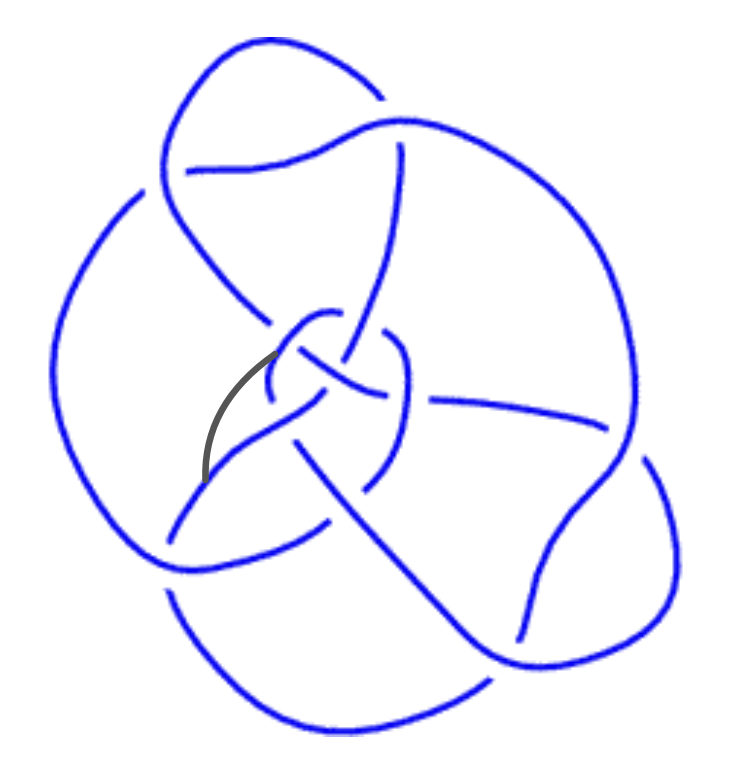}
		\caption{$11n_{176}\stackrel{-1}{\longrightarrow} 8_{10}$}
		
	\end{subfigure}
	~
	\begin{subfigure}[b]{0.27\textwidth}
		\includegraphics[width=\textwidth]{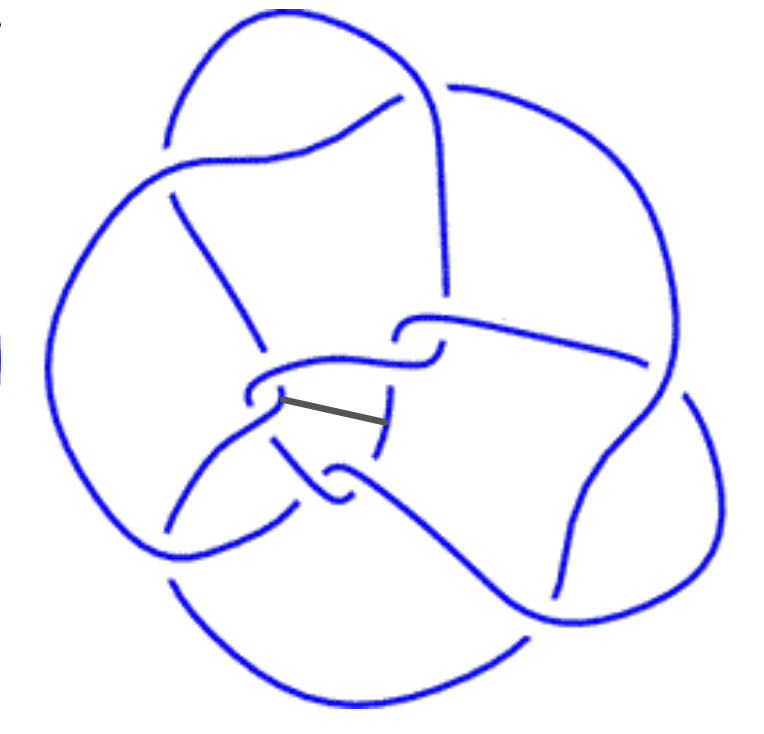}
		\caption{$11n_{179}\stackrel{0}{\longrightarrow} 8_{14}$}
		
	\end{subfigure}
	~
	\begin{subfigure}[b]{0.27\textwidth}
		\includegraphics[width=\textwidth]{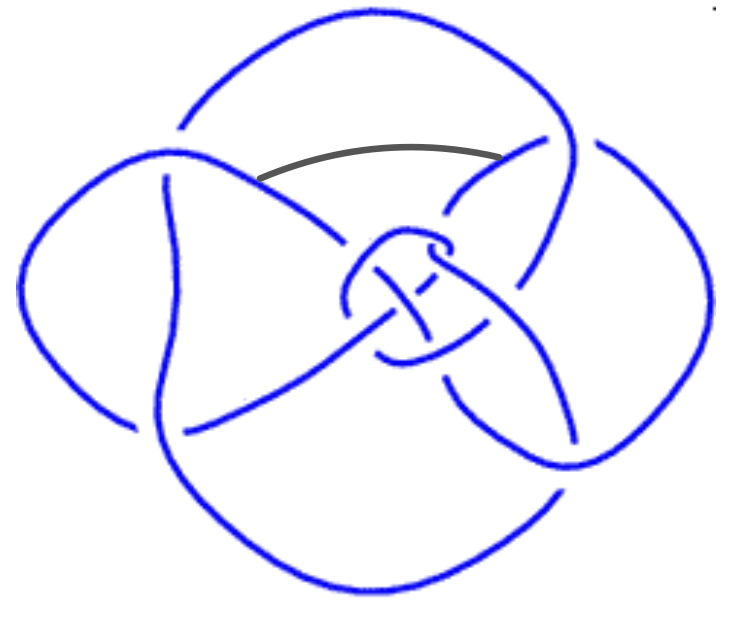}
		\caption{$11n_{184}\stackrel{0\phantom{i}}{\longrightarrow} 6_{2}$}
		
	\end{subfigure}
	\vskip3mm
	\caption{Non-oriented band moves from the knots $11n_{176},  11n_{179}, \text{ and} \\  11n_{184}  $ to knots with non-orientable genus 1.}\label{last1G}
\end{figure}
\clearpage

\newpage
\bibliographystyle{plain}
\bibliography{main}

\end{document}